\documentclass[11pt,reqno]{amsart}
\usepackage[foot]{amsaddr}
\usepackage{amsthm,amssymb,amsfonts,amsmath,amscd}
\usepackage{stmaryrd}
\usepackage{setup}
\usepackage{microtype}
\usepackage[normalem]{ulem}
\usepackage[nameinlink,capitalise]{cleveref}
\newtheorem*{rmk}{Remark}

\crefname{theorem}{Theorem}{Theorems}
\crefname{thm}{Theorem}{Theorems}
\crefname{lem}{Lemma}{Lemmas}
\crefname{remark}{Remark}{Remarks}
\crefname{claim}{Claim}{Claims}
\crefname{conj}{Conjecture}{Conjectures}
\crefname{prop}{Proposition}{Propositions}
\crefname{defn}{Definition}{Definitions}
\crefname{cor}{Corollary}{Corollaries}
\crefname{figure}{Figure}{Figures}

\usepackage{cases}

\newcommand{\N}{\mathbb{N}}
\renewcommand{\P}{\mathbf{P}}

\newcommand{\R}{\mathbb{R}}
\newcommand{\Z}{\mathbb{Z}}
\newcommand{\bbT}{\mathbb{T}}
\usepackage{bbm}

\newcommand{\U}{\mathcal{U}}

\newcommand{\ttS}{\mathtt{S}}
\newcommand{\ttT}{\mathtt{T}}
\newcommand{\ttH}{\mathtt{H}}
\newcommand{\ttR}{\mathtt{R}}

\newcommand{\tti}{\mathtt{i}}
\newcommand{\ttj}{\mathtt{j}}
\newcommand{\ttk}{\mathtt{k}}
\newcommand{\tth}{\mathtt{h}}
\newcommand{\ttr}{\mathtt{r}}
\newcommand{\ttu}{\mathtt{u}}
\newcommand{\ttv}{\mathtt{v}}
\newcommand{\ttw}{\mathtt{w}}
\newcommand{\ttt}{\mathtt{t}}

\newcommand{\be}{\mathbf{e}}

\newcommand{\rT}{\mathrm{T}}
\newcommand{\rS}{\mathrm{S}}

\newcommand{\rE}{\mathrm{E}}

\newcommand{\dd}{\mathrm{d}}

\newcommand{\sT}{\mathscr{T}}

\newcommand{\e}{\varepsilon}

\newcommand{\E}{{\mathbf E}}

\newcommand{\I}[1]{{\mathbf 1}_{\left\{#1\right\}}}
\newcommand{\one}{\mathbf{1}}

\newcommand{\eqdist}{\ensuremath{\stackrel{\mathrm{d}}{=}}}
\newcommand{\convdist}{\ensuremath{\stackrel{\mathrm{d}}{\rightarrow}}}

\newcommand{\prob}{\mathbf{P}}
\newcommand{\ex}{\mathbf{E}}

\newcommand\cE{\mathcal E}

\newcommand\cR{{\mathcal R}}

\usepackage{wrapfig}

\newcommand{\Ht}{\mathrm{Ht}}
\newcommand{\rk}{\mathrm{k}}
\newcommand{\Cut}{\operatorname{Cut}}
\newcommand{\Trim}{\operatorname{Trim}}
\newcommand{\LCS}{\operatorname{LCS}}

\newcommand{\eps}{\varepsilon}
\renewcommand{\emptyset}{\varnothing}
\newcommand{\fw}{\mathrm{fw}}
\newcommand{\bw}{\mathrm{bw}}
\newcommand{\dGH}{\mathtt{d}_{\mathrm{GH}}}

\title{The largest common subtree of two random trees}

\author[O.~Angel]{Omer Angel$^*$}
\address{$^*$Department of Mathematics, University of British Columbia, 
\texttt{angel@math.ubc.ca}}

\author[C.~Atamanchuk]{Caelan Atamanchuk$^\dagger$}
\address{$^\dagger$Department of Mathematics and Statistics, McGill University, 
\texttt{caelan.atamanchuk@gmail.com}}

\author[A.~Brandenberger]{Anna Brandenberger$^\ddagger$}
\address{$^\ddagger$Department of Mathematics, MIT,
\texttt{abrande@mit.edu}}

\author[S.~Donderwinkel]{ Serte Donderwinkel$^\mathsection$}
\address{$^\mathsection$Bernoulli Institute and CogniGron, University of Groningen,
\texttt{s.a.donderwinkel@rug.nl}}

\author[R.~Khanfir]{Robin Khanfir$^\circ$}
\address{$^\circ$Department of Mathematics and Statistics, McGill University, 
\texttt{robin.khanfir@mcgill.ca}}

\begin{document}

\keywords{Largest common subtrees, Bienaymé--Galton--Watson trees, Brownian CRT}
\subjclass[2010]{60C05, 60J80, 05C05, 05C60}

\begin{abstract}
  We study the size and structure of the largest common subtree (LCS) between two independent Bienaym\'{e} trees conditioned to have size $n$.
  When the trees are critical with finite $2$nd and $(2+\kappa)$th moment respectively for some $\kappa>0$, we prove that the LCS has size of order $\sqrt{n}$, and is approximated by the length of three paths meeting at a central node. Moreover, we show that the largest common subtree between two critical independent Bienaymé trees with size $n$ and finite second moments may be much larger than $\sqrt{n}$, implying that our result is tight. We also pose a number of open questions and suggestions for future research.
\end{abstract}

\maketitle

\section{Introduction}

For trees $t$, $t'$ and $t''$, say that $t''$ is a \emph{common subtree} of $t$ and $t'$ if $t''$ is isomorphic to a subtree of both $t$ and $t'$.
The size of the \emph{largest common subtrees} (LCS) of $t$ and $t'$ is denoted by $\LCS(t,t')$.
The purpose of this paper is the study of the LCS when $t$ and $t'$ are two independent random trees.
A natural setting is that of Bienaym\'e trees\footnote{Also known as Galton--Watson trees or family trees of branching processes.} conditioned to have size $n$.
The following theorem is a less detailed version of our main result (\cref{thm:main}).

\begin{thm}
  Let $\tau_n$ and $\tau'_n$ be independent critical Bienaymé trees conditioned to have size $n$, whose offspring distributions have a finite $(2+\kappa)$th, for some $\kappa>0$, and a finite $2$nd moment, respectively.
  Then, there is a finite random variable $X>0$ for which
  \[ n^{-1/2} \LCS(\tau_n,\tau'_n) \xrightarrow[n\to\infty]{\mathrm d} X. \] 
\end{thm}

We remark that we view trees as just graphs with no additional structure such as vertex labels or a plane embedding (even though in our proofs we use the Ulam formalism for trees wherein vertices have a natural order).
In \cref{thm:a_star_is_born}, we show that this result is tight, in the sense that having finite second moments for both offspring distributions is not sufficient. 

Our main result is \cref{thm:main} which relates the limit $X$ to the scaling limits of the Bienaym\'e trees.
To state that result, we need some additional definitions.
\smallskip

Let $\mu=(\mu(k))_{k\geq 0}$ and $\mu'=(\mu'(k))_{k\geq 0}$ be two \emph{critical and non-trivial offspring distributions}, meaning that
\begin{equation}
\label{criticality_assumption}
1=\sum_{k\geq 0}k\mu(k)=\sum_{k\geq 0}k\mu'(k)\quad\text{ and }\quad \mu(0),\mu'(0)>0.
\end{equation}
Moreover, we assume that $\mu$ and $\mu'$ have finite second moments:
\begin{equation}
\label{finite_2-moment_assumption}
\sigma^2:=\sum_{k\geq 0}(k-1)^2\mu(k)<\infty\quad\text{ and }\quad (\sigma')^2:=\sum_{k\geq 0}(k-1)^2\mu'(k)<\infty.
\end{equation}

Let us denote by $\tau$ and $\tau'$ two independent Bienaymé trees with respective offspring distributions $\mu$ and $\mu'$ (a formal definition of Bienaym\'{e} trees is given in \cref{sec:preliminaries}). 
We will hereafter only consider $n$ for which $\P(\# \tau=n)>0$ and $\P(\# \tau'=n)>0$. Then for any such $n$, denote by $\tau_n$ and $\tau_n'$ two independent Bienaymé trees with offspring distributions $\mu$ and $\mu'$, respectively, and conditioned to have size $n$. We view $\tau_n$ and $\tau_n'$ as compact metric spaces equipped with their graph distances.

For any metric space $\rE=(E,d)$ and $\lambda>0$, we write $\lambda\cdot \rE=(E,\lambda d)$ for the metric space obtained by rescaling the distances on $E$ by $\lambda$. It is now well-known \cite{aldous1993continuumIII,haas2012scaling} that 
\begin{equation}
\label{Bienayme->CRT}
\bigg(\frac{1}{\sqrt{n}}\cdot\tau_n,\frac{1}{\sqrt{n}}\cdot\tau_n'\bigg)\xrightarrow[n\to\infty]{\mathrm d}\left(\frac{2}{\sigma}\cdot \sT,\frac{2}{\sigma'}\cdot \sT'\right)
\end{equation}
with respect to the Gromov--Hausdorff topology, where $\sT$ and $\sT'$ are two independent \emph{Brownian continuum random trees (CRTs)}. More precisely, if $\be=(\be_s)_{s\in[0,1]}$ is a standard Brownian excursion, then $\sT$ and $\sT'$ have the same law as the quotient metric space of $[0,1]$ induced by the pseudo-distance $d_{\be}$ defined as follows:
\[\forall 0\leq s_1\leq s_2\leq 1,\quad d_{\be}(s_1,s_2)=\be_{s_1}+\be_{s_2}-2\min_{[s_1,s_2]}\be.\]

Finally, we introduce our limit quantity. For any two compact metric spaces $\rE=(E,d)$ and $\rE'=(E',d')$, set
\begin{multline}
\label{Phi}
\Phi(\rE,\rE')=\sup\big\{\tfrac{1}{2}d(x_1,x_2)+\tfrac{1}{2}d(x_2,x_3)+\tfrac{1}{2}d(x_3,x_1) \, :\, x_1,x_2,x_3\in E\text{ and }x_1',x_2',x_3'\in E'\\
\text{ such that }d(x_i,x_j)=d'(x_i',x_j')\text{ for all }1\leq i,j\leq 3\big\}.
\end{multline}
When $\rE$ and $\rE'$ are compact real trees (\cref{real-tree_def}), then $\Phi(\rE,\rE')$ corresponds to \emph{the maximum total length of a common Y-shaped subtree} of $\rE$ and $\rE'$: see \cref{fig:largest_y}. For details on compact real trees, the Brownian continuum random tree, and the Gromov--Hausdorff topology, we direct the reader to \cite{legall2005random} (see also \cref{sec:compact_real_tree} for a brief summary).

Now, let $\tau_*$ and $\tau'_*$ be two independent rooted trees whose laws are characterized by the following:
\begin{enumerate}
    \item[(a)] For all integers $k\geq 0$, the root of $\tau_*$ has $k$ children with probability $(k+1)\mu(k+1)$, and the root of $\tau_*'$ has $k$ children with probability $(k+1)\mu'(k+1)$.
    \item[(b)] Conditionally given the offspring size of the root in $\tau_*$ (resp.~in $\tau_*'$), the subtrees of $\tau_*$ (resp.~of $\tau_*'$) above these children are independent Bienaymé trees with offspring distribution $\mu$ (resp.~$\mu'$).
\end{enumerate}
We refer to these random rooted trees, $\tau_*$ and $\tau_*'$, as \emph{root-biased Bienaym\'{e} trees} with offspring distributions $\mu$ and $\mu'$ respectively. 

For rooted trees $t$, $t'$ and $t''$, we say that $t''$ is a \emph{common rooted subtree} of $t$ and $t'$ if there are graph isomorphisms from $t''$ to both $t$ and $t'$ that map the root to the root. Write $\LCS^\bullet(t,t')$ for the maximum size of a common rooted subtree of $t$ and $t'$.
\smallskip

We are now ready to state our main result in its full form.

\begin{thm}
\label{thm:main}
Assume that \eqref{criticality_assumption} and \eqref{finite_2-moment_assumption} hold, keep the above notation, and set $c_{\mu,\mu'}=\E\big[\LCS^\bullet(\tau_*,\tau_*')\big]$. If there exists $\kappa>0$ such that $\sum_{k\geq 0}k^{2+\kappa}\mu(k)<\infty$, then the following convergence in distribution on $\R$ holds jointly with \eqref{Bienayme->CRT}:
\[\frac{1}{\sqrt{n}}\LCS(\tau_n,\tau_n') \xrightarrow[n\to\infty]{\mathrm d} c_{\mu,\mu'}\Phi\big(\tfrac{2}{\sigma}\cdot\sT,\tfrac{2}{\sigma'}\cdot\sT'\big).\]
\end{thm}
\begin{rmk}
Our proofs in fact give, under the assumptions of \cref{thm:main}, that for any $c>0$, 
\[\frac{1}{\sqrt{n}}\LCS(\tau_n,\tau_{\lfloor cn\rfloor}') \xrightarrow[n\to\infty]{\mathrm d} c_{\mu,\mu'}\Phi\big(\tfrac{2}{\sigma}\cdot\sT,\tfrac{2\sqrt{c}}{\sigma'}\cdot\sT'\big),\]
but to ease notation we will not prove this more general theorem.
\end{rmk}

\begin{figure}
    \centering
    \includegraphics[scale=0.7]{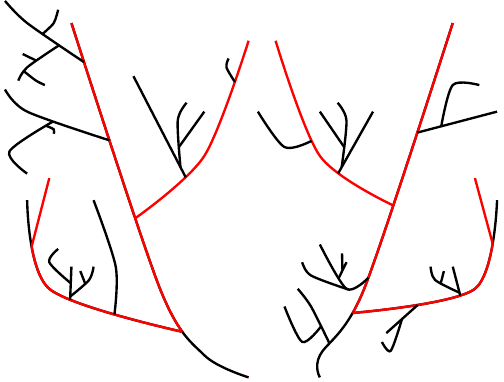}
    \caption{Two real trees with their largest common Y-shaped subtree highlighted red.}
    \label{fig:largest_y}
\end{figure}

\subsection{Discussions}

Let us briefly present some concrete examples of cases where \cref{thm:main} applies. The family of size-conditioned critical Bienaym\'{e} trees with finite variance offspring distributions contains several usual combinatorial families of random trees. In fact, one does not need to even look past offspring distributions with finite exponential moments to find many well-studied models. Here are a few examples:
\begin{itemize}
    \item If we set $\mu(d) =1-\mu(0) =  1/d$ for some $d\geq 2$, then $\tau_n$ is distributed as a uniform $d$-ary tree on $n$ vertices (assuming that $n = md + 1$ for some integer $m \geq 0$).
    \item If we set $\mu(k) = 2^{-(k+1)}$ for all $k \geq 0$, then $\tau_n$ is distributed as a uniform rooted plane tree on $n$ vertices.
    \item If we set $\mu(k) = {e^{-1}}/{k!}$ for all $k \geq 0$, then after randomly labelling the vertices, $\tau_n$ is distributed as a uniform labelled tree on $n$ vertices.
\end{itemize}
\cref{thm:main} applies to any pair of trees from the above list (again, ignoring the labels and plane order when looking for common subtrees).

One may wonder if the condition that one of the two offspring distributions has to have a finite $(2+\kappa)$th moment is necessary, or just a consequence of techniques. The following result shows that \cref{thm:main} is false if we allow $\kappa=0$, implying that our theorem is tight in this generality. 

\begin{thm}\label{thm:a_star_is_born}
    For any $\gamma < 1/2$, there exists a critical and non-trivial offspring distribution $\mu$ satisfying $\sum_{k\ge 0}k^2\mu(k)<\infty$ such that for $\mu'=\mu$, with high probability $\LCS(\tau_n,\tau'_n) \geq \sqrt{n} \log^\gamma n$ as $n\to\infty$.
\end{thm}

We prove this theorem in \cref{sec:a_star_is_born}.

\subsection{Outline and proof overview}
Here we sketch the contents of each section of the rest of the paper and offer some heuristic explanations for how one could derive our main result.  
\medskip

{\bf \cref{sec:preliminaries}:} We present some standard structural results for Bienaym\'{e} trees that we use throughout the paper, as well as a few key lemmas about heavy-tailed random walks.
\medskip

{\bf \cref{sec:LCS_rooted_unconditioned}:}  We consider the largest common \emph{rooted} subtree for two independent \emph{unconditioned} Bienaym\'{e} trees with finite variance offspring distributions, $\tau$ and $\tau'$. In short, we prove that the size of the largest common rooted subtree of two such trees, which we denote by $\LCS^\bullet(\tau,\tau')$, is well approximated by $\Ht(\tau) \wedge \Ht(\tau')$, the minimum of their respective heights. Specifically, we show that
\begin{equation}\label{eq:expression_for_bootstrap}
\prob\big( \Ht(\tau) \wedge \Ht(\tau') \leq h \,;\,  h^{-\varepsilon} \LCS^\bullet(\tau,\tau') > \Ht(\tau) \wedge \Ht(\tau')+1 \big)
\end{equation}
decays super-polynomially in $h$ for any fixed $\varepsilon > 0$ (\cref{prop:height=size}).
\smallskip

Let us explain how knowing that $\LCS^\bullet(\tau,\tau')$ behaves like $\Ht(\tau) \wedge \Ht(\tau')$ is useful for dealing with (reasonably) large degrees. Suppose that $u \in \tau$ and $u' \in \tau'$ are two vertices that correspond to the same vertex in the largest common rooted subtree, and let $u_1,\ldots,u_K$ and $u_1',\ldots,u_{K'}'$ be the children of $u$ and $u'$ in $\tau$ and $\tau'$ respectively. Then there are at least $(K \wedge K')!$ ways to pair the children to try to build large common subtrees, a union bound that we cannot afford when degrees are not bounded. However, if we know that $\LCS^\bullet$ behaves as the minimum of the heights, then we know that 
$$
\LCS^\bullet\big(\tau(u),\tau'(u')\big) \lesssim \left(\sum_{i=1}^{K} \Ht\big(\tau(u_i)\big)\right) \wedge \left(\sum_{i=1}^{K'} \Ht\big(\tau'(u_i')\big)\right),
$$
where we temporarily let $\tau(\cdot)$ and $\tau'(\cdot)$ denote the subtree rooted at $(\cdot)$ in $\tau$ and $\tau'$ respectively. This leaves us with a sum that we can control instead of a gigantic union bound. 
\smallskip

From hereon, the presence of large degrees makes the argument more technically complex, so for simplicity of exposition we will focus the remainder of the proof outline on the case where $\mu=\mu'=\tfrac{1}{2}\delta_0+\tfrac{1}{2} \delta_2$, so that $\tau_n$ and $\tau'_n$ are binary.
\smallskip

We now sketch how we prove that \eqref{eq:expression_for_bootstrap} decays quickly. 
We control the probability through bootstrapping. 
If $B_\eps$ is the event in \eqref{eq:expression_for_bootstrap}, we bound $\prob(B_\eps)$ in terms of $\prob(B_{\eps-\nu})$ for some small $\nu$, improving the exponent in the bound each time this is applied. 
More precisely, let $\ttT^\bullet$ be a largest common rooted subtree of $\tau$ and $\tau'$, and let $P$ be a path in $\ttT^\bullet$ starting from the root that walks up the tree, at each step proceeding into whichever of the two child subtrees is larger.
There are essentially two options for how the bad event $B_\eps$ can occur:
\begin{enumerate}
    \setlength{\itemsep}{0.01em}
    \item[(i)] For all $v \in P$, the subtree of $\ttT^\bullet$ rooted at $v$ that branches off from $P$ has size bounded above by $h^{\varepsilon-\nu}(\Ht(\tau) \wedge \Ht(\tau')+1)$, for some fixed $\nu>0$, but the sum of the sizes is nevertheless bigger than $h^\varepsilon (\Ht(\tau) \wedge \Ht(\tau')+1)$.
    \item[(ii)] There is some vertex $v\in P$ so that the subtree rooted at $v$ which branches off from $P$ has size bigger than $h^{\varepsilon-\nu}(\Ht(\tau) \wedge \Ht(\tau')+1)$.
\end{enumerate}

In the first case, we can use known tail bounds for the sizes of Bienaym\'{e} trees (\cref{single_tree_height}), along with a big jumps principle for random walks (\cref{obj_0-1and1-2}), to show that (i) has exponentially small probability. By the construction of $P$, if (ii) occurs, then both subtrees stemming from $v$ in $\ttT^\bullet$ are larger than $h^{\varepsilon-\nu}(\Ht(\tau) \wedge \Ht(\tau')+1)$. By independence of disjoint subtrees in Bienaym\'{e} trees, we then approximately obtain two independent copies of the bad event $B_{\varepsilon-\nu}$, allowing a divide-and-conquer strategy using the previous iteration of the bootstrapping. This allows us to improve the upper bound on \eqref{eq:expression_for_bootstrap}, and recursively applying the argument then gives an arbitrarily good polynomial bound. One key consequence of having good bounds for \eqref{eq:expression_for_bootstrap} is that $\prob(\LCS^\bullet(\tau,\tau') > h) \lesssim h^{-2}$ (\cref{cor:LCS-tail}).
\medskip

{\bf \cref{sec:Unrooted_LCS}:} We convert our understanding of common rooted subtrees of unconditioned trees from \cref{sec:LCS_rooted_unconditioned} to \emph{unrooted} common subtrees of \emph{conditioned} trees. Specifically, we show that $\LCS(\tau_n,\tau_n')$ is well approximated by a constant multiple of $\LCS_N(\tau_n,\tau_n')$, where $\LCS_N$ denotes the total length of the largest common subtree with only $N\geq 3$ leaves. More formally, we provide arbitrarily good polynomial bounds for 
\begin{equation}\label{eq:sec4-outline}
\prob\left( \big|\LCS(\tau,\tau') - c_{\mu,\mu'}\LCS_N(\tau,\tau')\big| > \delta \sqrt{n} \right),
\end{equation}
when we take $N$ large but constant (\cref{thm:size-to-length}, \cref{prop:size-to-length_estimate}). Essentially, this says that, typically, a largest common subtree consists of a large skeleton with at most $N$ leaves and $\LCS_N(\tau_n,\tau_n')$ edges, combined with small trees hanging off this skeleton that collectively obey a law-of-large-numbers behaviour. Indeed, this is where the factor $c_{\mu,\mu'}$ comes from: for a given vertex on the skeleton, approximately, the corresponding pendant trees in $\tau$ and $\tau'$ are independent, and the expected size of their largest common rooted subtrees is $c_{\mu,\mu'}$. Bounding \eqref{eq:sec4-outline} amounts to showing that a largest common subtree contains a bounded number of macroscopic mass-splits, and to making the approximate law of large numbers precise. Finally, we use that our bounds decay quickly enough for our conclusions to also hold for the conditioned trees $\tau_n$ and $\tau'_n$.
\medskip

{\bf \cref{sec:Scaling_limit}:} We finally argue that for any $N\ge 3$, it holds that $\LCS_N(\tau_n, \tau'_n)$ is not much bigger than $\LCS_3(\tau_n, \tau'_n)$. Thereafter, the convergence of $\frac{\sigma}{\sqrt{n}}\cdot \tau_n$ and $\frac{\sigma'}{\sqrt{n}}\cdot \tau_n'$ towards two rescaled copies of the Brownian CRT, $2\cdot \sT$ and $2\cdot \sT'$, implies the convergence of $n^{-1/2}\LCS_3(\tau_n, \tau'_n)$ to $\Phi\big(\tfrac{2}{\sigma}\cdot\sT,\tfrac{2}{\sigma'}\cdot\sT'\big)$, that is the length of the largest common Y in these two continuum trees, finally concluding the proof of \cref{thm:main}.
\smallskip

The fact that $\LCS_N(\tau_n, \tau'_n)$ and $\LCS_3(\tau_n, \tau'_n)$ are close hinges on a simple structural property of size-conditioned critical Bienaym\'{e} trees with finite variance offspring distributions: There are few vertices where the tree branches into three (or more) components of height $\Theta(\sqrt{n})$, and distances between pairs of such macroscopic branch-points are unlikely to repeat between $\tau_n$ and $\tau'_n$. 
Therefore, a common subtree of $\tau_n$ and $\tau_n'$ can have at most one macroscopic branch-point, and thus the extra branches that the extra $N-3$ leaves allow for increase the maximum length of a common subtree by a negligible amount.
\smallskip

The proof uses the scaling convergence of $\tau_n$ and $\tau_n'$ to the Brownian CRT and the continuity properties of the longest length of a common subtree with few leaves in the Gromov--Hausdorff topology. Then the result follows by the fact that, in this continuum tree, distances between typical points have an absolutely continuous law, implying that pairs of branch-points in $\tfrac{2}{\sigma}\cdot \sT$ are not at the same distance as any pair of branch-points in $\tfrac{2}{\sigma'}\cdot \sT'$.
We remark that in general the size of the largest common subtree is not a continuous function with respect to the Gromov--Hausdorff topology (see \cref{sec:proof_main} for a discussion), which is a small difficulty that we overcome.

\medskip
{\bf \cref{sec:open-problems}:} We conclude with a number of potential directions for future work.

\subsection{Related Work}

Although this work is the first on the topic of largest common subtrees of two Bienaym\'{e} trees, largest common subgraph problems have been a growing subject of interest in recent years. Here we review some of the work that has been done, and highlight their various motivations.
\smallskip

Some of the earliest work on largest common subgraph problems in a mathematical context comes from the study of maximum agreement subtrees for uniform random leaf-labelled binary trees. Let $t$ and $t'$ be two binary trees with $n$ leaves labelled by $\{1,\ldots,n\}$. For a subset $A \subseteq \{1,\ldots,n\}$, the tree $t$ restricted to $A$ is the tree obtained by contracting all edges incident to vertices of degree two within the minimally connected subgraph of $t$ that contains all leaves labelled by $A$. The tree $t'$ restricted to $A$ is defined analogously. We call a leaf-labelled binary tree $t^*$ an agreement subtree of $t$ and $t'$ of size $m$ if there is a set $A \subseteq \{1,\ldots,n\}$ of size $m$ such that $t$ and $t'$ restricted to $A$ both equal $t^*$. A maximum agreement subtree is one of maximum size. The study of maximum agreement subtrees gets its motivation from computational biology, where leaf-labelled binary trees are used as a model for phylogenetic trees~\cite{semple2003phylogenetics}, and sizes of maximum agreement subtrees are meant to quantify what information two phylogenetic trees share~\cite{david2003size}. 

The size of maximum agreement subtrees between two independent uniformly random binary trees with $n$ labelled leaves, denoted by $M_n$, has received the most attention. In the work~\cite{david2003size} that initiated the mathematical study of $M_n$, the authors showed that $\ex[M_n] = O(\sqrt{n})$. Since then, after a couple iterations of improvements~\cite{bernstein2015bounds,aldous2022largest}, the current best lower bound is $\ex[M_n] = \Omega(n^{0.4464})$~\cite{khezeli2024improved}. When $t$ and $t'$ are conditioned to have the same shape, the expected size has been shown to be $\Theta(\sqrt{n})$~\cite{misra2019bounds}.

Typical values of $M_n$ have also been investigated. In the pioneer work~\cite{david2003size}, the authors identified a $\lambda > 0$ such that $\prob(M_n \geq \lambda \sqrt{n})$ tends to zero as $n \to \infty$. This fact has been shown to remain true when the two uniform leaf-labelled binary trees are replaced with critical Bienaym\'{e} trees conditioned to have $n$ leaves which are given a uniform labelling~\cite{pittel2023expected}. In~\cite{budzinski2023maximum}, the authors improve the result of~\cite{david2003size}, identifying an $\varepsilon > 0$ such that $\prob(M_n \geq n^{1/2-\varepsilon})$ tends to zero as $n \to \infty$ . The problem of identifying the correct asymptotic order of $M_n$ remains an ongoing topic of study.

\smallskip
The study of maximal agreement subtrees in the Yule--Harding model has also been studied, with an upper bound of $O(\sqrt{n})$ on the expected size and a lower bound of approximately $\Omega(n^{0.344})$ \cite{david2003size,bernstein2015bounds}.
\smallskip

In an upcoming paper, Baumler, Kerriou, Martin, Lodewijks, Powierski, R\'acz, and Sridhar~\cite{RRT-LCS} consider the size of the largest common subtree of two independent random recursive trees. 
They find a polynomial lower bound $n^\alpha$ for some $\alpha>0.8$, and an upper bound of $(1-\eps)n$ for some $\eps$. In particular, the largest common subtree of two random recursive trees of size $n$ is asymptotically much larger than $n^{1/2}$, which is the order of the largest common subtree of two Bienaymé trees of size $n$ under our assumptions. They conjecture that the asymptotic size is in fact $n^{1-o(1)}$.
\smallskip

Despite being a relatively recent topic of study in mathematics, the computational complexity of identifying large common subtrees has been discussed for some time within the theoretical computer science community~\cite{akutsu1992RNC,shamir1999faster}. The focus has not been limited to trees; see~\cite{ehrlich2011maximum} for algorithms concerning large common subgraphs in more general settings.
\smallskip

Another model for which the largest common subgraph problem has received attention is Erd\"{o}s--R\'{e}nyi random graphs $G(n,p)$ in the dense regime. In contrast to our results for Bienaym\'{e} trees, the size of the largest common induced subgraphs between independent Erd\"{o}s--R\'{e}nyi random graphs has been shown to exhibit strong concentration: for two independent $G(n,1/2)$ graphs, \cite{chatterjee2023isomorphisms} shows that it concentrates on two values, both with leading order term $4 \log_2 n$. This two-point-concentration phenomenon has been extended to $G(n,p)$ and $G(n,q)$ for $p \ne q$  not depending on $n$~\cite{diamantidis2024combinatorial,surya2025isomorphisms}. The largest common induced subgraph problem for $G(n,p)$ is related to the graph isomorphism problem, a classical topic in the computer science literature~\cite{mccreesh2018subgraph}. For a topic beyond simple graphs, the largest common hypergraph between two uniform $d$-hypergraphs was studied in \cite{lenoir2024isomorphisms}.
\smallskip

As a final note, we highlight the work on longest common subsequences for random words, which predates the literature on largest common subgraphs by some time. With its origins in computational biology, the problem of studying the longest common subsequence between two random strings of length $n$ over a finite alphabet has received a large amount of attention over the years~\cite{Chvatal1975,Danvcik1995,Deken1979,Houdre2023,kiwi2005expected,Lueker2009}. The longest common subsequence problem for independent random permutations has also been extensively discussed~\cite{Romik_mono,Jin2019}.

\subsection*{Acknowledgements}
This work was initiated at the  Nineteenth Annual Workshop on Probability and Combinatorics held from March 28--April 4, 2025 at McGill University's Bellairs Institute in Holetown, Barbados. We thank the organizers for inviting us and creating a supportive research environment, Peleg Michaeli for posing this problem, and the participants for their valuable input. OA was supported by the NSERC. 
AB was supported by NSF GRFP 2141064 and Simons Investigator Award 622132. AB and CA acknowledge the financial support of the NSERC Canada Graduate Scholarship -- Doctoral. SD acknowledges the financial support of the CogniGron research center and the Ubbo Emmius Funds (University of Groningen). Her research was also supported by
the Marie Skłodowska-Curie grant GraPhTra (Universality in phase transitions in random
graphs), grant agreement ID 101211705. RK was supported by the NSERC via a Banting postdoctoral fellowship [BPF-198443].

\section{Preliminaries}
\label{sec:preliminaries}

This section contains a number of formal definitions and technical lemmas that we use in our proof of \cref{thm:main}. We first present the plane tree formalism that we will use to talk about Bienaym\'{e} trees, as well as the useful Many-to-One principle. We also present a number of other results concerning statistics like size, height, and large degrees in critical Bienaym\'{e} trees. At the end, we record a number of versions of the big jumps principle, which provides control over sums of i.i.d.~heavy-tailed random variables.

\subsection{Plane tree formalism for Bienaymé trees, and Many-to-One principle.}

We begin by introducing the plane tree formalism. Note that the specific ordering of offspring is irrelevant for the notion of common (rooted) subtrees; we use the notation only for convenience of labelling vertices in our trees.

Denote by $\N^*=\{1,2,3,\ldots\}$ the set of positive integers and by $\U$ the set of finite words written with the alphabet $\N^*$, that is,
\[\U=\bigcup_{\ell\geq 0}(\N^*)^\ell\quad\text{ with the convention }\quad(\N^*)^0=\{\varnothing\}.\]
As a set of words, $\U$ is totally ordered by the \emph{lexicographic order} $\leq$, so that $\varnothing<(1)<(1,2)<(2)$ for example. For two words $u=(u_1,\ldots,u_\ell)$ and $v=(v_1,\ldots,v_m)$, we write $u*v=(u_1,\ldots,u_\ell,v_1,\ldots,v_m)$ for their concatenation. We also denote by $|u|=\ell$ the length of $u$, with $|\varnothing|=0$, which we call \emph{the height of $u$}. Moreover, if $u\neq\varnothing$ then we call $\overleftarrow{u}:=(u_1,\ldots,u_{\ell-1})$ \emph{the parent of $u$}. If $u,v\in \U\setminus\{\varnothing\}$ have the same parent and if $u\leq v$, we say that $u$ is an \emph{older sibling} of $v$ and that $v$ is a \emph{younger sibling} of $u$.  We also define the \emph{genealogical order} $\preceq$ on $\U$, which is the partial ordering generated by the covering relation\footnote{For a partially ordered set $(\mathcal{P},\prec)$, $y\in \mathcal{P}$ covers $x\in \mathcal{P}$ if $x\prec y$ and for all $z\in \mathcal{P}$, if $x\preceq z \preceq y$ then $z=x$ or $z=y$.} $u\in t$ covers $v\in t$ if and only if $v=\overleftarrow{u}$. Then $u\preceq v$ if and only if there is a $w\in \mathbb{U}$ so that $v=u*w$. If $u\preceq v$ we say that $u$ is an \emph{ancestor of $v$}. We also write $u \prec v$ to express that $u \preceq v$ and $u \ne v$. 

\begin{defn}
\label{def:tree}
A \emph{plane tree} is a (potentially infinite) subset $t$ of $\U$ satisfying the following conditions:
\begin{enumerate}
    \item[(a)] $\varnothing\in t$,
    \item[(b)] for all $u\in t$, if $u\neq\varnothing$ then $\overleftarrow{u}\in t$,
    \item[(c)] for all $u\in t$, there is an integer $\rk_u(t)\geq 0$ such that $u * (i)\in t \Leftrightarrow 1\leq i\leq  \rk_u(t)$ for all $i \in \N^*$.
\end{enumerate}
\end{defn}

Let $t$ be a plane tree. We call $\varnothing$ the \emph{root} of $t$. For $u\in t$, we call $\{u\! *\! (i) : 1\leq i \leq \rk_u(t)\}$ the \emph{children of $u$ in $t$}. A \emph{leaf of $t$} is a vertex $u\in t$ such that $\rk_u(t)=0$. We shall always view a plane tree $t$ as a rooted graph whose set of vertices is $t$, whose edges are the $\{\overleftarrow{u},u\}$ for $u\in t\setminus\{\varnothing\}$, and whose root is $\varnothing$. In particular, when $t$ and $t'$ are two plane trees, the quantities $\LCS(t,t')$ and $\LCS^\bullet(t,t')$ are well-defined. Also note that an injection $\phi:t'\to t$ is an isomorphic embedding of rooted graphs if and only if $\phi(\emptyset)=\emptyset$ and $\overleftarrow{\phi(u)}=\phi(\overleftarrow{u})$ for all $u\in t'\setminus\{\varnothing\}$. Finally, see that $\rk_u(t)$ is the out-degree of $u$ in the rooted graph $(t,\varnothing)$.
\smallskip

For $u\in t$, we respectively define the \emph{subtree rooted at $u$} and the subtree pruned above $u$:
\[\theta_u t=\{v\in \U :u*v\in t\}\quad\text{ and }\quad \Cut_u t=t\setminus\{u*v\in t: v\in (\theta_u t)\setminus\{\varnothing\}\}.\]
Also see \cref{fig:basic-defs}. 
Note that $\theta_u t$ and $\Cut_u t$ are plane trees in the sense of \cref{def:tree}. Now assume that $u\neq\varnothing$ and write $u=v*(i)$ with $1\leq i\leq \rk_v(t)$. We define
\[\Trim_u t=\{\varnothing\}\cup\{(j)*w\, :\, 1\leq j<i,w\in\theta_{v*(j)}t\}\cup\{(j-1)*w\, :\, i< j\leq \rk_v(t),w\in\theta_{v*(j)}t\}.\]
Also see \cref{fig:basic-defs}. 
Informally, $\Trim_u t$ is obtained by removing the subtree of $t$ rooted at $u$ from the subtree rooted at its parent $\overleftarrow{u}$.

With our plane tree notation, we can formally define Bienaym\'{e} trees and two variants of Bienaym\'{e} trees which are relevant to us. For the following definition, and the rest of the paper, we use the notation $\N = \{0,1,2,...\}$.

\begin{defn}
Let $\mu$ be a measure on $\N$. A random plane tree $\tau$ is called a \emph{Bienaym\'{e} tree with offspring distribution $\mu$} if it satisfies the following properties:
\begin{enumerate}
    \item[(a)] The law of $\rk_\emptyset(\tau)$ is $\mu$.
    \item[(b)] For all $k \geq 1$ such that $\mu(k) > 0$, conditionally given that $\rk_\emptyset(\tau) = k$, the subtrees $\theta_1\tau,...,\theta_k\tau$ are i.i.d.~and with the same law as $\tau$.
\end{enumerate}
We frequently call (b) \emph{the branching property} of Bienaymé trees. If $\sum_{k \geq 0} k\mu(k) = 1$ and $\mu(0) > 0$, then we call $\tau$ and $\mu$ \emph{critical}, and it is well-known that $\tau$ is almost surely finite.
\end{defn}

\begin{defn}
\label{def:root-biased}
Let $\mu$ be a critical measure on $\N$. A random plane tree $\tau_*$ is called a \emph{root-biased Bienaym\'{e} tree with offspring distribution $\mu$} if it satisfies the following properties:
\begin{enumerate}
    \item[(a)] For all $k \geq 0$, $\prob(\rk_\emptyset(\tau_*) = k) = (k+1)\mu(k+1)$.
    \item[(b)] For all $k \geq 1$ such that $\prob(\rk_\emptyset(\tau_*) = k) > 0$, conditionally given that $\rk_{\emptyset}(\tau_*) = k$, $\theta_1\tau,...,\theta_k\tau$ are independent Bienaym\'{e} trees with offspring distribution $\mu$.
\end{enumerate}
\end{defn}

Let $\mu$ be a probability measure on $\N$ such that $\mu(0)>0$ and $m_\mu:=\sum_{k\geq 0}k\mu(k)\in(0,\infty)$.

\begin{defn}\label{def:size-biased-tree}
Let $\tau_\infty$ be a random infinite plane tree and $U_\infty=(J_i)_{i\geq 1}$ be a sequence of random positive integers. Set $U_0=\varnothing$ and $U_n=(J_1,\ldots,J_n)$ for all $n\geq 1$. We say that $(\tau_\infty,U_\infty)$ is a \emph{size-biased Bienaymé tree with offspring distribution $\mu$} if it satisfies the following properties:
\begin{enumerate}
    \item[(a)] For all $n\geq 0$, $U_n$ is almost surely an element of $\tau_\infty$.
    \item[(b)] The random pairs $(\rk_{U_n}(\tau_\infty),J_{n+1})$ for $n\geq 0$ are i.i.d.~and distributed as follows:
    \[\forall k,j\in\N^*,\quad \P(\rk_{U_n}(\tau_\infty)=k\, ;\, J_{n+1}=j)=\I{j\leq k}\mu(k)/m_\mu.\]
    \item[(c)] Conditionally given $(\rk_{U_n}(\tau_\infty),J_{n+1})_{n\geq 0}$, the random plane trees $\theta_{U_n*(j)}\tau_\infty$ for $n\geq 0$ and $j\in \{1,\ldots,\rk_{U_n}(\tau_\infty)\}\setminus\{J_{n+1}\}$ are independent Bienaymé trees with offspring distribution $\mu$.
\end{enumerate}
\end{defn}

Comparing \cref{def:root-biased,def:size-biased-tree} readily gives the following result.

\begin{prop}
\label{root-biased_VS_size-biased}
Let $(\tau_\infty,U_\infty)$ be a size-biased Bienaymé tree with offspring distribution $\mu$ as above. Then, the random plane trees $\Trim_{U_n}\tau_\infty$ for $n\geq 1$ are independent root-biased Bienaymé trees with offspring distribution $\mu$.
\end{prop}

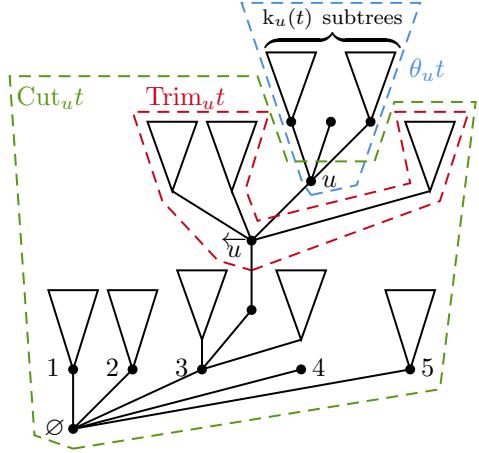
\begin{SCfigure}[.5]
    \centering
    \tikzset{every picture/.style={line width=0.75pt}} 

\begin{tikzpicture}[x=0.75pt,y=0.75pt,yscale=1,xscale=1]

\draw    (160,120) -- (120,145) ;
\draw    (70,25) -- (185,55) ;
\draw [shift={(185,55)}, rotate = 14.62] [color={rgb, 255:red, 0; green, 0; blue, 0 }  ][fill={rgb, 255:red, 0; green, 0; blue, 0 }  ][line width=0.75]      (0, 0) circle [x radius= 2.01, y radius= 2.01]   ;
\draw    (70,25) -- (70,55) ;
\draw [shift={(70,55)}, rotate = 90] [color={rgb, 255:red, 0; green, 0; blue, 0 }  ][fill={rgb, 255:red, 0; green, 0; blue, 0 }  ][line width=0.75]      (0, 0) circle [x radius= 2.01, y radius= 2.01]   ;
\draw [shift={(70,25)}, rotate = 90] [color={rgb, 255:red, 0; green, 0; blue, 0 }  ][fill={rgb, 255:red, 0; green, 0; blue, 0 }  ][line width=0.75]      (0, 0) circle [x radius= 2.01, y radius= 2.01]   ;
\draw    (70,25) -- (100,55) ;
\draw [shift={(100,55)}, rotate = 45] [color={rgb, 255:red, 0; green, 0; blue, 0 }  ][fill={rgb, 255:red, 0; green, 0; blue, 0 }  ][line width=0.75]      (0, 0) circle [x radius= 2.01, y radius= 2.01]   ;
\draw    (135,55) -- (160,85) ;
\draw [shift={(160,85)}, rotate = 50.19] [color={rgb, 255:red, 0; green, 0; blue, 0 }  ][fill={rgb, 255:red, 0; green, 0; blue, 0 }  ][line width=0.75]      (0, 0) circle [x radius= 2.01, y radius= 2.01]   ;
\draw    (70,25) -- (240,55) ;
\draw [shift={(240,55)}, rotate = 10.01] [color={rgb, 255:red, 0; green, 0; blue, 0 }  ][fill={rgb, 255:red, 0; green, 0; blue, 0 }  ][line width=0.75]      (0, 0) circle [x radius= 2.01, y radius= 2.01]   ;
\draw    (135,55) -- (135,70) ;
\draw    (190,150) -- (200,180) ;
\draw [shift={(200,180)}, rotate = 71.57] [color={rgb, 255:red, 0; green, 0; blue, 0 }  ][fill={rgb, 255:red, 0; green, 0; blue, 0 }  ][line width=0.75]      (0, 0) circle [x radius= 2.01, y radius= 2.01]   ;
\draw    (135,55) -- (185,70) ;
\draw    (190,150) -- (180,180) ;
\draw [shift={(180,180)}, rotate = 108.43] [color={rgb, 255:red, 0; green, 0; blue, 0 }  ][fill={rgb, 255:red, 0; green, 0; blue, 0 }  ][line width=0.75]      (0, 0) circle [x radius= 2.01, y radius= 2.01]   ;
\draw    (190,150) -- (220,180) ;
\draw [shift={(220,180)}, rotate = 45] [color={rgb, 255:red, 0; green, 0; blue, 0 }  ][fill={rgb, 255:red, 0; green, 0; blue, 0 }  ][line width=0.75]      (0, 0) circle [x radius= 2.01, y radius= 2.01]   ;
\draw    (160,120) -- (150,145) ;
\draw    (160,120) -- (190,150) ;
\draw [shift={(190,150)}, rotate = 45] [color={rgb, 255:red, 0; green, 0; blue, 0 }  ][fill={rgb, 255:red, 0; green, 0; blue, 0 }  ][line width=0.75]      (0, 0) circle [x radius= 2.01, y radius= 2.01]   ;
\draw  [line width=0.75]  (70,55) -- (82.5,95) -- (57.5,95) -- cycle ;
\draw  [line width=0.75]  (100,55) -- (112.5,95) -- (87.5,95) -- cycle ;
\draw    (70,25) -- (135,55) ;
\draw [shift={(135,55)}, rotate = 24.78] [color={rgb, 255:red, 0; green, 0; blue, 0 }  ][fill={rgb, 255:red, 0; green, 0; blue, 0 }  ][line width=0.75]      (0, 0) circle [x radius= 2.01, y radius= 2.01]   ;
\draw  [line width=0.75]  (135,70) -- (147.5,105) -- (122.5,105) -- cycle ;
\draw    (160,120) -- (250,145) ;
\draw  [line width=0.75]  (185,70) -- (197.5,105) -- (172.5,105) -- cycle ;
\draw  [line width=0.75]  (240,55) -- (252.5,95) -- (227.5,95) -- cycle ;
\draw  [line width=0.75]  (120,145) -- (132.5,180) -- (107.5,180) -- cycle ;
\draw  [line width=0.75]  (150,145) -- (162.5,180) -- (137.5,180) -- cycle ;
\draw  [line width=0.75]  (250,145) -- (262.5,180) -- (237.5,180) -- cycle ;
\draw  [line width=0.75]  (180,180) -- (192.5,215) -- (167.5,215) -- cycle ;
\draw  [line width=0.75]  (220,180) -- (232.5,215) -- (207.5,215) -- cycle ;
\draw  [color={rgb, 255:red, 74; green, 144; blue, 226 }  ,draw opacity=1 ][dash pattern={on 4.5pt off 3pt}] (190,143) -- (213,148) -- (245,240) -- (155,240) -- (185.5,149.75) -- cycle ;
\draw  [color={rgb, 255:red, 208; green, 2; blue, 27 }  ,draw opacity=1 ][dash pattern={on 4.5pt off 3pt}] (160,105) -- (255,140) -- (269.5,185) -- (232.5,185) -- (240,149) -- (163,130) -- (157,146) -- (169,185) -- (100,185) -- (115,140) -- (143,110) -- cycle ;
\draw  [color={rgb, 255:red, 88; green, 152; blue, 17 }  ,draw opacity=1 ][dash pattern={on 4.5pt off 3pt}] (70,15) -- (255,45) -- (273,190) -- (232,190) -- (222,160) -- (178,160) -- (163,203) -- (38,203) -- (50.5,21.75) -- cycle ;
\draw    (160,85) -- (160,120) ;
\draw [shift={(160,120)}, rotate = 90] [color={rgb, 255:red, 0; green, 0; blue, 0 }  ][fill={rgb, 255:red, 0; green, 0; blue, 0 }  ][line width=0.75]      (0, 0) circle [x radius= 2.01, y radius= 2.01]   ;

\draw (54,19.4) node [anchor=south west][inner sep=0.75pt]  [font=\small,color={rgb, 255:red, 0; green, 0; blue, 0 }  ,opacity=1 ]  {$\emptyset $};
\draw (55,50.4) node [anchor=south west][inner sep=0.75pt]  [font=\small,color={rgb, 255:red, 0; green, 0; blue, 0 }  ,opacity=1 ]  {$1$};
\draw (189,50.4) node [anchor=south west][inner sep=0.75pt]  [font=\small,color={rgb, 255:red, 0; green, 0; blue, 0 }  ,opacity=1 ]  {$4$};
\draw (244,50) node [anchor=south west][inner sep=0.75pt]  [font=\small,color={rgb, 255:red, 0; green, 0; blue, 0 }  ,opacity=1 ] [align=left] {$\displaystyle 5$};
\draw (85,50.4) node [anchor=south west][inner sep=0.75pt]  [font=\small,color={rgb, 255:red, 0; green, 0; blue, 0 }  ,opacity=1 ]  {$2$};
\draw (120,50.4) node [anchor=south west][inner sep=0.75pt]  [font=\small,color={rgb, 255:red, 0; green, 0; blue, 0 }  ,opacity=1 ]  {$3$};
\draw (194,145.4) node [anchor=south west][inner sep=0.75pt]  [font=\small,color={rgb, 255:red, 0; green, 0; blue, 0 }  ,opacity=1 ]  {$u$};
\draw (143,110) node [anchor=south west][inner sep=0.75pt]  [font=\small,color={rgb, 255:red, 0; green, 0; blue, 0 }  ,opacity=1 ]  {$\overleftarrow{u}$};
\draw (238,199) node [anchor=south west][inner sep=0.75pt]  [font=\small,color={rgb, 255:red, 74; green, 144; blue, 226 }  ,opacity=1 ] [align=left] {$\displaystyle \theta _{u} t$};
\draw (164,214.4) node [anchor=south west][inner sep=0.75pt]  [font=\normalsize]  {$\overbrace{\ \ \ \ \ \ \ \ \ \ \ \ \ \ }^{\rk_{u}(t) \ \text{subtrees}}$};
\draw (105,185) node [anchor=south west][inner sep=0.75pt]  [font=\small,color={rgb, 255:red, 208; green, 2; blue, 27 }  ,opacity=1 ] [align=left] {$\displaystyle \mathrm{Trim}_{u} t$};
\draw (40,185) node [anchor=south west][inner sep=0.75pt]  [font=\small,color={rgb, 255:red, 88; green, 152; blue, 17 }  ,opacity=1 ] [align=left] {$\displaystyle \mathrm{Cut}_{u} t$};

\end{tikzpicture}
    \caption{Illustration of $\theta_u t$, $\Cut_u t$ and $\Trim_u t$}
    \label{fig:basic-defs}
\end{SCfigure}

Size-biased Bienaymé trees naturally appear as local limits of Bienaymé trees with finite variance offspring distributions, conditioned to be large (see for example~\cite{abraham2015introduction} for a general overview). The relations between standard Bienaymé trees and size-biased Bienaymé trees can also be highlighted more directly via a distributional identity. We use this so-called \emph{Many-to-One principle} under the following form frequently in the present paper. Part of the folklore, this principle can be verified by a simple induction argument (see e.g.~\cite[Equation (24)]{duquesne09}) that we omit here.

\begin{prop}[Many-to-One principle]
\label{many-to-one}
Let $\tau$ be a Bienaymé tree with offspring distribution $\mu$ and let $(\tau_\infty,U_\infty)$ be a size-biased Bienaymé tree with offspring distribution $\mu$ as above. Then, for all $n\geq 0$ and for all non-negative functions $F$ and $G$, it holds that
 \[\E\bigg[\sum_{u\in \tau}\I{|u|=n}F(\Cut_u \tau,u)G(\theta_u\tau)\bigg]=m_\mu^n\, \E\big[F(\Cut_{U_n} \tau_\infty,U_n)\big]\, \E\big[G(\tau)\big].\]
\end{prop}

Note that we often apply the Many-to-One principle to bivariate functions. Indeed, let $\tau$ and $\tau'$ be two independent Bienaymé trees with offspring distributions $\mu$ and $\mu'$ respectively, and let $(\tau_\infty, U_\infty)$ and $(\tau_\infty', U_\infty')$ be two corresponding independent size-biased Bienaymé trees. By applying the Many-to-One principle iteratively to each tree while holding the other fixed, we have that for all $n, m \geq 0$ and all $F,G$ non-negative,
\begin{multline*}
    \E\Big[\sum_{u\in \tau}\sum_{u'\in \tau'}\I{|u|=n}\I{|u'|=m} F(\Cut_u \tau, u ; \Cut_{u'} \tau', u') G(\theta_u\tau ; \theta_{u'}\tau')\Big] \\ 
    = m_\mu^n\, m_{\mu'}^m\, \E\left[F(\Cut_{U_n} \tau_\infty, U_n ; \Cut_{U_m'} \tau'_\infty, U_m')\right]\, \E\left[G(\tau ; \tau')\right].
\end{multline*}

\subsection{Shape characteristics of critical Bienaym\'{e} trees with finite variance}

In all this subsection, fix a probability measure $\mu$ on $\N$ and assume that $\mu$ has mean $m_u=1$, variance $\sigma^2\in (0,\infty)$, and $\mu(0)>0$. Here, we gather several estimates about the asymptotic behaviour of the size, height, and maximum degree of a Bienaymé tree with offspring distribution $\mu$.

We begin by considering the size of Bienaymé trees.

\begin{prop}\label{single_tree_size} 
   Let $\tau$ be a Bienaymé tree with offspring distribution $\mu$. Let $\mathrm{h}$ be the largest common divisor of the support of $\mu$, and let $\mathrm{M}$ be the set of integers $m$ such that $\P(\#\tau=m)>0$. Then, there exists a constant $c> 0$ that only depends on $\sigma^2$ such that
   \begin{align*}
   \lim_{\substack{m\to\infty \\ m\in \mathrm{M}}} m^{3/2}\P(\#\tau=m)&=\mathrm{h}c,\\
   \quad \lim_{m\to\infty}m^{1/2}\P(\#\tau\geq m)&= 2c.
   \end{align*}
\end{prop}
\begin{proof}
Let $X$ be a sample from $\mu$, and let $(Y_j)_{j\ge 0}$ be a random walk starting from $0$ with i.i.d.~steps distributed as $X-1$. Then, for $T_{-1}=\min\{j\geq 0:Y_j=-1\}$, we have by standard results on random walk encodings of Bienaymé trees that $T_{-1}\eqdist \#\tau$ (see e.g., \cite[Corollary~1.6]{legall2005random}). Moreover, Kemperman's formula (see, e.g., \cite{van2008elementary}) yields that 
\[\P(T_{-1}=m )=\frac{1}{m}\P(Y_m=-1).\]
Thus, it is left to estimate $\P(Y_m=-1)$ as $m$ tends to infinity. But since $\mu$ has finite variance, this is given by the local limit theorem, which states that
\begin{equation}
\label{local-limit_thm}
\lim_{m\to \infty}  \sup_{k\in -m+\mathrm{h}\Z}\left|m^{1/2}\P(Y_m=k)-\mathrm{h}\phi(m^{-1/2}k)\right|=0,
\end{equation}
where $\phi$ is the continuous density of the normal distribution with mean $0$ and variance $\sigma^2$ (see, e.g., \cite[Chapter 7]{petrov2012sums}). We see that $Y_m$ is supported by $-m+\mathrm{h}\Z$, and so if $m\notin 1+\mathrm{h}\Z$ then $\P(Y_m=-1)=0=\P(\#\tau=m)$. In other words, if $m\in\mathrm{M}$ then $-1\in -m+\mathrm{h}\Z$. Combining this with \eqref{local-limit_thm} then yields the first desired estimate with $c=\phi(0)$.

Recall that $\mathrm{M}\subseteq 1+\mathrm{h}\Z$. Moreover, since $m_\mu = 1$ and $\mu(0)>0$, it holds that $\#\tau$ is finite almost surely. Thus, we get
\[\P(\#\tau\geq m)=\sum_{i=\lfloor \frac{m-1}{\mathrm{h}}\rfloor}^\infty \P(\#\tau=1+\mathrm{h}i)\sim \frac{c}{\sqrt{\mathrm{h}}}\int_{m/\mathrm{h}}^\infty x^{-3/2}\, \mathrm{d} x=2c m^{-1/2}.\qedhere\]
\end{proof}

We also present the following result on the size of subtrees in root-biased Bienaym\'{e} trees.

\begin{lem}\label{lem:size-biased_tail}
    Let $\tau_*$ be a root-biased Bienaymé tree with critical offspring distribution $\mu$. Then, there exists a constant $c > 0$ such that $\prob(\#\tau_* \geq m) \leq cm^{-1/2}$ for all $m>0$.
\end{lem}

\begin{proof}
    As in the previous proof, let $X$ be a sample from $\mu$, and let $(Y_j)_{j\ge 0}$ be a random walk starting from $0$ with i.i.d.~steps distributed as $X-1$.
    By definition, when we condition on $\rk_\emptyset(\tau_*)$, the tree $\tau_*$ is exactly $\rk_\emptyset(\tau^*)$ independent Bienaym\'{e} trees all attached to a mutual root. In particular, again using the well-known connection between Bienaym\'{e} trees and random walks~\cite[Corollary 1.6]{legall2005random}, this implies that
    $$
    \prob(\# \tau_*-1 = m  \ | \  \rk_\emptyset(\tau_*) = k) = \prob( T_{-k} = m),
    $$
    where $T_{-k} - \inf\{ j \geq 0 : Y_j = -k \}$. Applying Kemperman's formula~\cite{van2008elementary}, we obtain
    \begin{equation}
    \label{kemperman_formula}
    \prob(\# \tau_*-1 = m  \ | \  \rk_\emptyset(\tau_*) = k) = \frac{k}{n}\prob(Y_m = -k).
    \end{equation}
    Then, the estimate \eqref{local-limit_thm} from the local limit theorem implies that there is $c > 0$ such that
    $$
    \prob(\# \tau_*-1 = m  \ | \  \rk_\emptyset(\tau_*) = k) \leq \frac{k}{m}\left(\frac{c}{\sqrt{m}}\right) = \frac{ck}{m^{3/2}}.
    $$
    We can show that the expected degree of the root is finite by a direct computation:
    \begin{equation}
    \label{first-moment_root-biased}
    \ex[\rk_\emptyset(\tau_*)] = \sum_{k = 1}^\infty (k-1)k\mu(k) = \sigma^2 < \infty.
    \end{equation}
    Altogether, we have that
    $$
    \prob(\# \tau_* - 1 = m) = \ex\left[ \prob(\# \tau_*-1 = m  \ | \  \rk_\emptyset(\tau_*)) \right] \leq \frac{c\ex[\rk_\emptyset(\tau_*)]}{m^{3/2}}.
    $$
    From here, summing over $m$ yields the result.
\end{proof}

Next, we recall two classical estimates about the height of critical Bienaymé trees with finite variance. The first one is an unconditioned tail estimate.

\begin{prop}[see e.g.~\cite{kesten1966galton}]
  \label{single_tree_height}
  If $\tau$ is a Bienaymé tree with offspring distribution $\mu$, then
  \[ \lim_{h\to\infty} h\P(\Ht(\tau)\geq h) = \frac{2}{\sigma^2}. \]
\end{prop}

The second one gives the order of magnitude of the height of a critical Bienaymé tree with finite variance and conditioned to have large size. This is a direct consequence of the scaling limit theorem of Aldous~\cite[Theorem~2.23]{aldous1993continuumIII}, already presented above as \eqref{Bienayme->CRT}. See also Flajolet~\cite{flajolet1982average}.

\begin{prop}
\label{prop:conditioned_tree_height}
  Let $\tau$ be a Bienaym\'{e} tree with offspring distribution $\mu$,  and let $\mathrm{M}$ be the set of integers $n$ such that $\P(\#\tau=n)>0$. Then,
  \[\limsup_{h\to\infty}\limsup_{\substack{n\to\infty\\ n\in\mathrm{M}}} \P(\Ht(\tau)\geq h n^{1/2} \,|\, \#\tau=n)=0.\]
\end{prop}

Finally, we state a rough upper bound on the maximum out-degree of critical Bienaymé trees with finite variance, conditioned to be large.

\begin{lem}
\label{lem:maximum_degree}
    Let $\tau$ be a Bienaym\'{e} tree with offspring distribution $\mu$, and let $\mathrm{M}$ be the set of integers $n$ such that $\P(\#\tau=n)>0$. If $\sum_{k\geq 0} k^{\gamma}\mu(k) < \infty$ for some $\gamma \geq 2$, then
    $$
    \lim_{A \to \infty}\lim_{\substack{n \to \infty\\ n\in\mathrm{M}}} \prob\left( \max_{v \in \tau} \rk_v(\tau) \geq An^{1/\gamma} \, \big| \, \#\tau =n \right) = 0.
    $$
\end{lem}

\begin{proof}
Let $X_1,X_2,\ldots$ be i.i.d.~samples from $\mu$, and set $Y_i=\sum_{j=1}^i (X_j-1)$ for all $i\geq 0$, so that $(Y_i)_{i\ge 0}$ is a random walk starting from $0$ with steps distributed as $X-1$. Also set $T_{-1}=\min\{i\geq 0:Y_i=-1\}$. Then, by the relations between random walks and Bienaymé trees~\cite[Corollary 1.6]{legall2005random}, we have the identity in distribution
\[\Big(\#\tau,\max_{v\in\tau}\rk_v(\tau)\Big) \eqdist \Big(T_{-1},\max_{1\leq i\leq T_{-1}}X_i\Big).\]
Furthermore, Dwass's cycle lemma~\cite{dwass} (see also~\cite[Lemma~15.3]{JansonSurvey}) entails that if $I$ is an independent random variable uniform on $\{0,\ldots,n-1\}$, then the conditional law of $(X_i)_{1\leq i\leq n}$ given $Y_n={-1}$ is equal to the conditional law of $(X_{I+i})_{1\leq i\leq n}$ given $T_{-1}=n$, where the indices $I+i$ are taken modulo $n$. It follows that
\begin{align*}
\prob\left( \max_{v \in \tau} \rk_v(\tau) \geq An^{1/\gamma} \,\big|\, \#\tau =n \right)&=\prob\left(\max_{1\leq i\leq n} X_i\geq An^{1/\gamma} \,\big|\, Y_n=-1\right)\\
&\leq n\prob\left(X_1\geq An^{1/\gamma} \,\big|\, Y_n=-1\right).
\end{align*}
Now, we see from the local limit estimate \eqref{local-limit_thm} that there is a constant $C > 0$ such that $\prob(Y_{n-1}=-x)\leq C\prob(Y_n=-1)$ for all $n\in\mathrm{M}$ and all $x\in\R$. Therefore,
\[n\prob\left(X_1\geq An^{1/\gamma} \,\big|\, Y_n=-1\right)=n\sum_{k=An^{1/\gamma}}^{\infty} \prob(X_1 = k)\frac{\prob(Y_{n-1} = -k)}{\prob(Y_n=-1)}\leq Cn\prob(X_1\geq An^{1/\gamma}).\]
Applying Markov's inequality completes the proof.
\end{proof}

\subsection{Big jumps principle}

Here, we state a manifestation of the \emph{big jumps principle}, which is the informal dictum that the most likely way for a real-valued random walk to reach an exceptionally large value is by far to have at least one jump of the same order as that value. More precisely, we give a result tailored for random walks with heavy-tailed jumps, which corresponds to the needs of the present paper.

\begin{prop}\label{obj_0-1and1-2}
Let $c>0$, $\alpha\in (0,1)\cup (1,2)$, and let $X_1,X_2,\dots$ be i.i.d.~random variables such that $\P(|X_1|>x)\leq c x^{-\alpha}$ for all $x>1$. If $1<\alpha<2$, then also assume that $X_1$ is integrable and that $\E[X_1]=0$. For each integer $m\geq 1$, set $S_m=\sum_{i=1}^m (X_i\wedge s m^{1/\alpha})$. Then, there exists a constant $C>0$ such that for any $m\ge 1$, $t\geq 0$, and $s\geq 1$,
\[\P(S_m\ge t m^{1/\alpha})\le C \exp(-t/s) .\]
\end{prop}

\begin{rmk} Some assumption on the negative tail of the $X_1$ is unavoidable. Indeed, if we have $\P(X_1<-s) \sim s^{-\beta}$ for some $1<\beta<\alpha$, then $(S_m)_{m \geq 0}$ behaves like a spectrally negative $\beta$-stable process, and will typically be of order $m^{1/\beta}$.
\end{rmk}

\begin{proof}
To lighten notation in this proof, set $\lambda= sm^{1/\alpha}\geq 1$ and
\[Y=\max(X_1,-\lambda)\wedge \lambda=X_1\I{|X_1|\leq \lambda}+\tfrac{X_1}{|X_1|}\lambda\I{|X_1|>\lambda}.\]
Note that $S_1=X_1\wedge \lambda \leq Y\leq\lambda$. We see that, by Markov's inequality,
\[\P(S_m\ge tm^{1/\alpha}) = \P\big(\exp(S_m/\lambda)\ge e^{t/s}\big) \le e^{-t/s} \E\big[\exp(S_1/\lambda) \big]^m \le e^{-t/s} \E\big[\exp(Y/\lambda) \big]^m,\]
so we are done if we find a constant $C_0>0$ such that $\E[\exp(m^{-1/\alpha}S_1/s) ]\leq 1+ C_0 \lambda^{-\alpha}$ for all $\lambda\geq 1$, since $m\lambda^{-\alpha}=s^{-\alpha}\leq 1$.

Clearly, there is a constant $C_1>0$ such that $e^x\leq 1+x+C_1 x^2$ for all $x\in (-\infty,1]$, so
\[\E\big[\exp(Y/\lambda) \big]\leq 1 + \lambda^{-1}\E\big[Y\big]+ C_1\lambda^{-2}\E\big[Y^2\big].\]
We control the two terms $\lambda^{-2}\E[Y^2]$ and $\lambda^{-1}\E[Y]$ separately. For the term $\lambda^{-2}\E[Y^2]$, remark that $|Y|=|X_1|\wedge\lambda$, and then write
\[\E\big[Y^2\big]=\E\big[(|X_1|\wedge\lambda)^2\big]=\int_0^{\lambda}2x\P(|X_1|>x)\, \dd x\leq 2 +\tfrac{2c}{2-\alpha}\lambda^{2-\alpha}.\]
Since $\alpha<2$, we get $\lambda^{-2}\E[Y^2]=O(\lambda^{-\alpha})$ as $\lambda\to\infty$, as desired. For the term $\lambda^{-1}\E[Y]$, we distinguish the cases $0<\alpha<1$ and $1<\alpha<2$. If $\alpha<1$, then write, similarly as above, 
\[\big|\E[Y]\big|\leq \E\big[|X_1|\wedge\lambda\big]\leq \int_0^\lambda \P(|X_1|>x)\, \dd x\leq 1 +\tfrac{c}{1-\alpha}\lambda^{1-\alpha},\] which gives $\lambda^{-1}\E[Y]=O(\lambda^{-\alpha})$. If $\alpha>1$, then we use the assumption that $\E[X_1]=0$ to get
\[\big|\E[Y]\big|=\Big|\E\Big[\Big(\tfrac{X_1}{|X_1|}\lambda-X_1\Big)\I{|X_1|>\lambda}\Big]\Big| \leq \E\big[(|X_1|-\lambda)\I{|X_1|> \lambda}\big]=\int_\lambda^\infty \P(|X_1|>x) \,\dd x,\]
which entails that $\lambda^{-1}\E[Y]=O(\lambda^{-\alpha})$ as $\lambda\to\infty$.
\end{proof}

We deduce from \cref{obj_0-1and1-2} a bound for the important harmonic case where $\alpha=1$.

\begin{cor}
\label{obj_1}
Suppose that $X_1,X_2,\dots$ are i.i.d.\ with $\P(|X_1|>x)\le c x^{-1}$ for all $x>1$. For each integer $m\geq 1$, set $S_m=\sum_{i=1}^m (X_i\wedge sm^{1+\gamma})$. Then, for any $\gamma>0$, there exists a constant $C>0$ such that for any $m\ge 1$, $t\ge 0$, and $s\geq 1$,
\[\P(S_m\ge t m^{1+\gamma})\le C \exp(-t/s ) .\]
\end{cor}

\begin{proof}
This result follows from \cref{obj_0-1and1-2} by picking $0<\alpha<1$ such that $1+\gamma=1/\alpha$. Then, $\P(X_1>x)\le cx^{-1}\le c x^{-\alpha}$ for all $x>1$, and the result follows. 
\end{proof}

\section{Largest common rooted subtree for independent unconditioned trees}
\label{sec:LCS_rooted_unconditioned}

Throughout this section, we fix $\tau,\tau'$ to be two independent critical Bienaymé trees with respective offspring distributions $\mu$ and $\mu'$ that both have finite second moments.
Before presenting the main results of this section, we first formally define a largest common rooted subtree.
A rooted tree $\ttT^\bullet$, with root $\ttr$, is called a \emph{largest common rooted subtree} of $\tau$ and $\tau'$ if it has the following two properties:
\begin{enumerate}
\item[(a)] there exist $\phi:\ttT^\bullet\rightarrow \tau$ and $\phi':\ttT^\bullet\rightarrow \tau'$ isomorphic embeddings of $\ttT^\bullet$ into $\tau$ and $\tau'$ where $\phi(\ttr) = \phi'(\ttr) = \emptyset$;
\item[(b)] $\ttT^\bullet$ has the maximal possible size among trees with such embeddings.
\end{enumerate}
Note that (as with unrooted subtrees) $\ttT^\bullet$ is not a plane tree and that we do not require its isomorphic embeddings to have any relationship with the plane orderings of $\tau$ and $\tau'$.
We define the size of any largest common rooted subtree by $\LCS^\bullet(\tau,\tau') = \#\ttT^\bullet$.

The obvious lower bound on $\LCS^\bullet(\tau,\tau')$ is the minimum of the heights, with the common rooted subtree being a single path.
The main goal of this section is the following proposition, stating that $\LCS^\bullet(\tau,\tau')$ is unlikely to be much larger than this minimum.
For any $\eps > 0$ and $h >0$, set
\begin{equation}\label{eq:p-eps}
p_{\eps}(h)= \P\Big(\Ht(\tau)\wedge\Ht(\tau')\leq h\, ;\, h^{-\varepsilon}\LCS^\bullet(\tau,\tau')> \Ht(\tau)\wedge\Ht(\tau') +1 \Big).
\end{equation}

\begin{prop}\label{prop:height=size}
  For any $\gamma,\eps > 0$ and $\eps > 0$, there exists $C > 0$ such that $p_\eps(h) \leq Ch^{-\gamma}$ for all $h > 0$.
\end{prop}

Thus, $p_\eps(h)$ decays super-polynomially in $h$.
Before delving into the proof of \cref{prop:height=size}, we derive some of its consequences.
First, using the asymptotic distribution of the height of an unconditioned Bienaym\'{e} tree (\cref{single_tree_height}), we can prove the following useful corollary.

\begin{cor}\label{cor:LCS-tail}
  For any $\eps > 0$, there exists $C > 0$ such that for all $h>0$, it holds that 
  \[
    \P(\LCS^\bullet(\tau,\tau') > h) \leq C h^{-\frac{2}{1+\e}}.
  \]
\end{cor}

\begin{proof}
  By \cref{single_tree_height}, the probability that both trees have height at least $\ell$ is bounded by $C\ell^{-2}$.
  Thus,
  \[
    \P(\LCS^\bullet(\tau,\tau') > k^\eps \ell) \leq p_\eps(k) + C\ell^{-2}.
  \]
  Take $k=\ell=h^{\frac{1}{1+\eps}}$ and use that $p_\eps(k)$ decays faster than any power of $h$ to get the claim.
\end{proof}

Combining \cref{single_tree_height,prop:height=size}, we will prove a similar result for root-biased Bienaym\'{e} trees (cf.~\cref{def:root-biased}).

\begin{prop}\label{prop:LCS_sb_tail}
  Let $\tau_*$ and $\tau_*'$ be root-biased Bienaym\'{e} trees with offspring distributions $\mu$ and $\mu'$ respectively. For any $\eps>0$, there exists $C>0$ such that for all $h>0$, it holds that 
  \[
    \prob(\LCS^\bullet(\tau_*,\tau_*') > h) \leq C h^{-\frac{2}{1+\e}}.
  \]
\end{prop}

We use the following deterministic lemma. 

\begin{lem}\label{lem:cases_sb}
    Let $h\geq 1$. If $\LCS^\bullet(\tau_*,\tau_*') > h+1$ then, for any $0<\alpha<1$, one of the following situations happens: 
    \begin{enumerate}
        \item[\emph{(i)}] There is $v\in \tau_*$ and $v'\in \tau_*'$ with $|v|=|v'|=1$ such that $\Ht(\theta_v\tau_*)\wedge \Ht(\theta_{v'}\tau_*')\le h$ and \[h^{-\alpha}\LCS^\bullet(\theta_v\tau_*,\theta_{v'}\tau_*')> \Ht(\theta_v\tau_*)\wedge \Ht(\theta_{v'}\tau_*') +1.\] 
        \item[\emph{(ii)}] It holds that 
        \[\left(\sum_{i=1}^{\rk_\varnothing(\tau_*)}1+\Ht(\theta_i\tau_*)\right)\wedge \left(\sum_{i=1}^{\rk_\varnothing(\tau_*')}1+\Ht(\theta_i\tau_*')\right)>h^{1-\alpha}.\]
    \end{enumerate}
\end{lem}

\begin{proof}
  We will show that if both situations do not happen, then  $\LCS^\bullet(\tau_*,\tau_*')\le h+1$. First, observe that if the second situation does not happen, then in particular for all $v\in \tau_*$ and $v'\in \tau_*'$ with $|v|=|v'|=1$, it holds that $\Ht(\theta_v\tau_*)\wedge \Ht(\theta_{v'}\tau_*')\le h$. Thus, for the first situation to not happen, it must hold that for all $v\in \tau_*$ and $v'\in \tau_*'$ with $|v|=|v'|=1$, 
    \[h^{-\alpha}\LCS^\bullet(\theta_v\tau_*,\theta_{v'}\tau_*')\leq \Ht(\theta_v\tau_*)\wedge \Ht(\theta_{v'}\tau_*')+1.\]
    Then, we observe that there is some $\ell\ge 0$, some distinct $1\le j_1,\dots, j_\ell\le \rk_\varnothing(\tau_*)$ and some distinct $1\le j_1',\dots,j_\ell'\le \rk_\varnothing(\tau_*')$ such that 
    \begin{align*}\LCS^\bullet(\tau_*,\tau_*')&=1+\sum_{i=1}^\ell \LCS^\bullet(\theta_{j_i}\tau_*,\theta_{j_i'}\tau_*')\\
    &\le 1+h^\alpha \sum_{i=1}^\ell 1+ \Ht(\theta_{j_i}\tau_*)\wedge \Ht(\theta_{j_i'}\tau_*')\\
    &\le 1+h^\alpha \left(\sum_{i=1}^{\rk_\varnothing(\tau_*)}1+\Ht(\theta_i\tau_*)\right)\wedge \left(\sum_{i=1}^{\rk_\varnothing(\tau_*')}1+\Ht(\theta_i\tau_*')\right)\le 1+h,
    \end{align*}
    which implies the lemma.
\end{proof}

\begin{proof}[Proof of \cref{prop:LCS_sb_tail}]
Fix $\alpha\in (0,1/3)$ such that $1-3\alpha=1/(1+\eps)$. We bound the probabilities of the two events in \cref{lem:cases_sb}. We will use that conditional on $\rk_\varnothing(\tau_*)$ (resp.~on $\rk_\varnothing(\tau_*')$), $\theta_1\tau_*,\dots, \theta_{\rk_\varnothing(\tau_*)}\tau_*$ (resp.~$\theta_1\tau_*',\dots, \theta_{\rk_\varnothing(\tau_*')}\tau_*'$) are independent copies of $\tau$ and $\tau'$ respectively.

The union bound implies that the probability of the first event is smaller or equal to $\E[\rk_\varnothing(\tau_*)]\E[\rk_\varnothing(\tau_*')]p_\alpha(h)$. Since the second moments of $\mu$ and $\mu'$ are both finite, $\rk_\varnothing(\tau_*)$ and $\rk_\varnothing(\tau_*')$ have finite mean. Then, \cref{prop:height=size} implies that the probability of the first event in \cref{lem:cases_sb} is $O(h^{-\gamma})$ for any $\gamma>0$, and in particular for $\gamma=\tfrac{2}{1+\varepsilon}$.

Next, observe that when $h$ is big enough so that $h^{1-3\alpha}<\tfrac{1}{2}h^{1-\alpha}$,
\begin{align*}\prob\left(\sum_{i=1}^{\rk_\varnothing(\tau_*)}1+\Ht(\theta_i\tau_*)>h^{1-\alpha}\right)&\le \prob(\rk_\varnothing(\tau_*)>h^{1-3\alpha})+\prob\left(\max_{1\le i \le \rk_\varnothing(\tau_*)}\Ht(\theta_i\tau_*)>h^{1-2\alpha}\right)\\
&\quad +\prob\left(\sum_{i=1}^{h^{1-3\alpha}}\Ht(\tau_i)\wedge h^{1-2\alpha} >\tfrac{1}{2} h^{1-\alpha}\right)
\end{align*}
for $\tau_1,\tau_2,\dots$ i.i.d.\ copies of $\tau$.
The first term is $O(h^{-(1-3\alpha)})$ by Markov's inequality. The second term is $O(h^{-(1-2\alpha)})$ by the union bound and \cref{single_tree_height}. The final term is $O(\exp(-\tfrac{1}{2}h^\alpha))$ by \cref{obj_1} and \cref{single_tree_height}. Thus, since $\tau_*$ and $\tau_*'$ are independent, the probability of the second event in \cref{lem:cases_sb} is $O(h^{-2(1-3\alpha)})=O(h^{-\frac{2}{1+\varepsilon}})$. 

Then, the result follows from the union bound and the separate bounds found above. 
\end{proof}

\subsection{Bootstrap for proving \cref{prop:height=size}}

Our proof of \cref{prop:height=size} relies on a bootstrap argument.
Specifically, we proceed by repeatedly applying the following lemma.

\begin{lem}\label{lem:bootstrap}
  For any $\nu \in (0,1)$ and $\varepsilon > 0$, there exist $c, C > 0$ such that for all $h\geq 1$,
  \begin{equation}\label{eq:bootstrap}
    p_{\varepsilon + 2\nu}(h) \leq C\frac{\log(h)}{h^{\varepsilon}}p_{\varepsilon}(h) + Ch^{2+\nu/4}\exp(-h^{c\nu^2}).
  \end{equation}
\end{lem}

Before getting into the proof of \cref{lem:bootstrap}, we verify that it indeed implies \cref{prop:height=size}.
This idea of the proof is natural: apply \cref{lem:bootstrap} repeatedly to obtain arbitrarily good polynomial bounds on $p_\eps(h)$ by induction.
For a function $f$ on $\R^+$, let us write $f(h)=O^*(h^\alpha)$ if there is a constant $C>0$ so that $f(h) \leq C h^\alpha (\log h)^{C}$ for $h$ large enough.

\begin{proof}[Proof of \cref{prop:height=size} using \cref{lem:bootstrap}]
  The lemma trivially implies that if $p_\eps(h) = O^*(h^\alpha)$ then $p_{\eps+2\nu}(h) = O^*(h^{\alpha-\eps})$, since the second term in \eqref{eq:bootstrap} decays super-polynomially.
  We claim that for any $k\in\N$ and $\eps>0$, we have $p_{\eps+2k\nu}(h) = O^*(h^{-k\eps})$.
  The proof is an easy induction on $k$, with the trivial base case $p_\eps(h)\leq 1 = O^*(h^0)$, and \cref{lem:bootstrap} giving the induction step.

  For any $\gamma>0$, take $k > \frac{\gamma}{\eps}$ and $\nu=\frac{\eps}{2k}$, to deduce $p_{2\eps}(h) = O^*(h^{-\gamma})$.
  Thus, $p_{2\eps}$ decays super-polynomially, and since $\eps$ is arbitrary, so does $p_\eps$.
\end{proof}

The remainder of \cref{sec:LCS_rooted_unconditioned} is devoted to proving \cref{lem:bootstrap}. Our strategy is to decompose the event $\left\{\Ht(\tau)\wedge\Ht(\tau')\leq h\, ;\, h^{-(\eps+2\nu)}\LCS^\bullet(\tau,\tau')> \Ht(\tau)\wedge\Ht(\tau')+1\right\}$ into a union of many events, where each event $E$ satisfies one of the following two properties:
\begin{enumerate}
    \item[(i)] there exists $C > 0$ such that $\prob(E) \leq C\frac{\log(h)}{h^\varepsilon}p_\varepsilon(h)$;
    \item[(ii)] there exists $C > 0$ and $\beta(\nu) >0$ such that $\prob(E) \leq Ce^{-h^{\beta(\nu)}}$.
\end{enumerate}
First, let us present these considered events. There are many cases, and so we recommend the reader skip past and read their definitions when they are discussed later in the section. Fix $\nu\in(0,1)$ and $\varepsilon>0$. For all $h\geq 1$, define the following events:
\begin{itemize}
    \item Let $A_h^{(1)}$ be the event that there exist $0\leq \tth\leq m-1\leq h$, and $u\in\tau$ and $u'\in\tau'$ with $|u|=|u'|=\tth$ such that \[\sum_{\substack{\varnothing\prec v\preceq u\\ \varnothing\prec v'\preceq u'}}\I{|v|=|v'|}\LCS^\bullet(\Trim_v\tau,\Trim_{v'}\tau')\wedge (h^{\varepsilon+\nu} m)\geq h^{\varepsilon+2\nu}m-1.\]
    
    \item Let $A_h^{(2)}$ be the event that there exist $0\leq \ell \leq h-1$, $u\in\tau$ and $u'\in\tau'$ with $|u|=|u'|=\ell$, and $1\leq i,j\leq \rk_u(\tau)$ and $1\leq i',j'\leq \rk_{u'}(\tau')$ with $j\neq i$ and $j'\neq i'$ such that $\LCS^\bullet(\theta_{u*i}\tau,\theta_{u'*i'}\tau')> h^\varepsilon(\ell+1)$, $\Ht(\theta_{u*j}\tau)\wedge\Ht(\theta_{u'*j'}\tau')\leq h$ and $h^{-\varepsilon}\LCS^\bullet(\theta_{u*j}\tau,\theta_{u'*j'}\tau')> \Ht(\theta_{u*j}\tau)\wedge\Ht(\theta_{u'*j'}\tau')+1$.
    \item Let $A_h^{(3)}$ be the event that there are $0\leq \ell \leq h-1$, $u\in\tau$ and $u'\in\tau'$ with $|u|=|u'|=\ell$, and $1\leq j\leq \rk_u(\tau)$ and $1\leq j'\leq \rk_{u'}(\tau')$ such that $\rk_u(\tau)\wedge \rk_{u'}(\tau')> \sqrt{h^\varepsilon(\ell+1)}$, $\Ht(\theta_{u*j}\tau)\wedge\Ht(\theta_{u'*j'}\tau')\leq h$ and $\tfrac{1}{h^{\varepsilon}}\LCS^\bullet(\theta_{u*j}\tau,\theta_{u'*j'}\tau') > \Ht(\theta_{u*j}\tau)\wedge\Ht(\theta_{u'*j'}\tau') \!+\! 1$.
    \item Let $A_h^{(4)}$ be the event that there exist $u\in\tau$ with $|u|\leq h-1$, and $1\leq i\leq \rk_u(\tau)$ such that $\Ht(\theta_{u*j}\tau)\leq h$ for all $1\leq j\leq \rk_u(\tau)$ with $j\neq i$, and \[\sum_{\substack{1\leq j\leq \rk_u(\tau) \\ j\neq i}}\big(1+\Ht(\theta_{u*j}\tau)\big)\geq -1+h^\nu \sup_{\substack{1\leq j\leq \rk_u(\tau) \\ j\neq i}}\big(1+\Ht(\theta_{u*j}\tau)\big).\]
    \item Let $A_h^{(4')}$ be the event that there exist $u'\in\tau'$ with $|u'|\leq h-1$, and $1\leq i'\leq \rk_{u'}(\tau')$ such that $\Ht(\theta_{u'*j}\tau')\leq h$ for all $1\leq j\leq \rk_{u'}(\tau')$ with $j\neq i'$, and \[\sum_{\substack{1\leq j\leq \rk_{u'}(\tau') \\ j\neq i'}}\big(1+\Ht(\theta_{u'*j}\tau')\big)\geq -1+h^\nu \sup_{\substack{1\leq j\leq \rk_{u'}(\tau') \\ j\neq i'}}\big(1+\Ht(\theta_{u'*j}\tau')\big).\]
    \item Let $A_h^{(5)}$ be the event that there exists $u\in\tau$ with $\ell:=|u|\leq h-1$ such that $\rk_u(\tau)\leq \sqrt{h^\varepsilon(\ell+1)}$ and \[\sum_{j=1}^{\rk_u(\tau)} (\#\theta_{u*j}\tau)\wedge (h^\varepsilon (\ell+1))\geq h^{\nu+\varepsilon} (\ell+1).\]
    \item Let $A_h^{(5')}$ be the event that there exists $u'\in\tau'$ with $\ell:=|u'|\leq h-1$ such that $\rk_{u'}(\tau')\leq \sqrt{h^\varepsilon(\ell+1)}$ and \[\sum_{j=1}^{\rk_{u'}(\tau')}(\#\theta_{u'*j}\tau')\wedge(h^\varepsilon (\ell+1))\geq h^{\nu+\varepsilon} (\ell+1).\]
\end{itemize}
\cref{lem:bootstrap} is then a direct consequence of the two propositions below. The first justifies that the considered events indeed cover the event that $\Ht(\tau)\wedge\Ht(\tau')\leq h$ and $h^{-(\eps+2\nu)}\LCS^\bullet(\tau,\tau')> \Ht(\tau)\wedge\Ht(\tau')+1$, and the second bounds the probability of each one of the events of the decomposition individually.

\begin{prop}
\label{prop:determ_bootstrap}
    Let $h\geq 1$. If $\Ht(\tau)\wedge \Ht(\tau')\leq h$ and if none of the events $A_h^{(1)}$, $A_h^{(2)}$, $A_h^{(3)}$, $A_h^{(4)}$, $A_h^{(4')}$, $A_h^{(5)}$, or $A_h^{(5')}$ occur, then $h^{-(\eps+2\nu)}\LCS^\bullet(\tau,\tau')\leq  \Ht(\tau)\wedge\Ht(\tau')+1$. 
\end{prop}

\begin{prop}
\label{prop:bound_bootstrap}
    The following bounds hold.
    \begin{enumerate}
        \item[\emph{(1)}] There exists $C^{(1)} > 0$ such that $\prob(A_h^{(1)}) \leq C^{(1)}h^2\exp(-\tfrac{1}{4}h^{\nu/2})$ for all $h \geq 1$.    
        \item[\emph{(2)}] There exists $C^{(2)} > 0$ such that $\prob(A_h^{(2)}) \leq C^{(2)}\frac{\log(h)}{h^{\varepsilon}}p_\varepsilon(h)$ for all $h \geq 1$.
        \item[\emph{(3)}] There exists $C^{(3)} > 0$ such that $\prob(A_h^{(3)}) \leq C^{(3)}\frac{\log(h)}{h^\varepsilon}p_\varepsilon(h)$ for all $h \geq 1$.
        \item[\emph{(4)}] There exist $c,C^{(4)} > 0$ such that $\prob(A_{h}^{(4)})+\prob(A_h^{(4')}) \leq C^{(4)}h^{2+\nu/4}\exp(-h^{c\nu^2})$ for all $h \geq 1$.
        \item[\emph{(5)}] There exists $C^{(5)} > 0$ such that $\prob(A_h^{(5)})+\prob(A_h^{(5')}) \leq C^{(5)}h\exp(-h^{\nu})$ for all $h \geq 1$.    
    \end{enumerate}
\end{prop}

The proofs of \cref{prop:determ_bootstrap,prop:bound_bootstrap} are respectively given in \cref{sec:determ_bootstrap,sec:bound_bootstrap}.

\subsection{Proof of the bootstrap: deterministic breakdown}
\label{sec:determ_bootstrap}

Let $\ttT^\bullet$ be a largest common rooted subtree between $\tau$ and $\tau'$, so that $\LCS^\bullet(\tau,\tau')=\#\ttT^\bullet$, and let $\phi:\ttT^\bullet\rightarrow \tau$ and $\phi':\ttT^\bullet\rightarrow \tau'$ be the associated isomorphic embeddings of $\ttT^\bullet$ into $\tau$ and $\tau'$ with $\phi(\ttr) = \phi'(\ttr) = \emptyset$, where $\ttr$ stands for the root of $\ttT^\bullet$. Here, we show \cref{prop:determ_bootstrap}, i.e., that if for some $h\geq 1$ it holds that $\Ht(\tau)\wedge\Ht(\tau')\leq h$ and none of the events $A_h^{(1)}$, $A_h^{(2)}$, $A_h^{(3)}$, $A_h^{(4)}$, $A_h^{(4')}$, $A_h^{(5)}$, or $A_h^{(5')}$ occur, then deterministically $h^{-(\eps+2\nu)}\#\ttT^\bullet \leq \Ht(\tau)\wedge\Ht(\tau')+1$. To do so, we first introduce some notation to describe the structure of $\ttT^\bullet$.

We will work with a path $(\ttu_0,\ldots,\ttu_{\ttH})$ in $\ttT^\bullet$ that we call the \emph{spine} of $\ttT^\bullet$, that we define inductively as follows. First, set $\ttu_0 = \ttr$. Then, for $i \geq 0$, consider the children of $\ttu_i$. If $\ttu_i$ does not have any children, then we set $\ttH = i$ and stop the construction. Otherwise, there are disjoint subtrees rooted at the children of $\ttu_i$, and we set $\ttu_{i+1}$ to be the child of $\ttu_i$ which is the root of the largest such tree (we break ties arbitrarily). Removing the edges of the spine partitions $\ttT^\bullet$ into $\ttH+1$ rooted trees $\ttS_0,\ldots,\ttS_{\ttH}$ such that the root of $\ttS_\ell$ is $\ttu_\ell$. For all $0\leq \ell \leq \ttH$, we denote by $\ttT_\ell$ the subtree of $\ttT^\bullet$ rooted at $\ttu_\ell$. The vertex set of $\ttT_\ell$ is the disjoint union of the vertex sets of $\ttT_{\ell+1}$ and $\ttS_\ell$. Finally, for $0 \le \ell < \ttH$, we denote by $\ttu_{\ell,1},\ldots,\ttu_{\ell,\ttk_\ell}$ the children of the root of $\ttS_\ell$, where $\ttk_\ell$+1 is thus the out-degree of $\ttu_\ell$ in $\ttT^\bullet$, and we denote by $\ttS_{\ell,\ttj}$ the subtree of $\ttT^\bullet$ (or equivalently of $\ttS_\ell$) rooted at $\ttu_{\ell,\ttj}$. We stress that $\ttu_{\ell+1} \notin \{\ttu_{\ell,1},\ldots,\ttu_{\ell,k_\ell}\}$. See \cref{fig:Tdot-spine} for an illustration of this full notation.

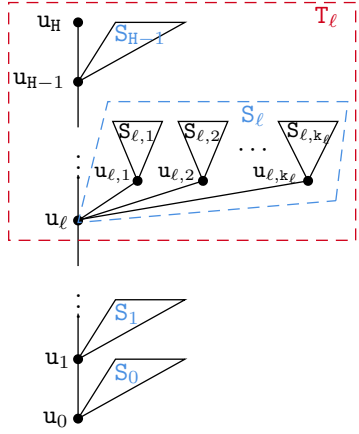
\begin{figure}
    \centering
    \begin{tikzpicture}[x=0.75pt,y=0.75pt,yscale=1,xscale=1]

\draw    (80,10) -- (80,40) ;
\draw [shift={(80,40)}, rotate = 90] [color={rgb, 255:red, 0; green, 0; blue, 0 }  ][fill={rgb, 255:red, 0; green, 0; blue, 0 }  ][line width=0.75]      (0, 0) circle [x radius= 2.01, y radius= 2.01]   ;
\draw [shift={(80,10)}, rotate = 90] [color={rgb, 255:red, 0; green, 0; blue, 0 }  ][fill={rgb, 255:red, 0; green, 0; blue, 0 }  ][line width=0.75]      (0, 0) circle [x radius= 2.01, y radius= 2.01]   ;
\draw    (80,110) -- (143,130) ;
\draw [shift={(143,130)}, rotate = 17.61] [color={rgb, 255:red, 0; green, 0; blue, 0 }  ][fill={rgb, 255:red, 0; green, 0; blue, 0 }  ][line width=0.75]      (0, 0) circle [x radius= 2.01, y radius= 2.01]   ;
\draw    (80,110) -- (110,130) ;
\draw [shift={(110,130)}, rotate = 33.69] [color={rgb, 255:red, 0; green, 0; blue, 0 }  ][fill={rgb, 255:red, 0; green, 0; blue, 0 }  ][line width=0.75]      (0, 0) circle [x radius= 2.01, y radius= 2.01]   ;
\draw    (80,40) -- (80,63) ;
\draw    (80,157) -- (80,180) ;
\draw [shift={(80,180)}, rotate = 90] [color={rgb, 255:red, 0; green, 0; blue, 0 }  ][fill={rgb, 255:red, 0; green, 0; blue, 0 }  ][line width=0.75]      (0, 0) circle [x radius= 2.01, y radius= 2.01]   ;
\draw    (80,180) -- (80,210) ;
\draw [shift={(80,210)}, rotate = 90] [color={rgb, 255:red, 0; green, 0; blue, 0 }  ][fill={rgb, 255:red, 0; green, 0; blue, 0 }  ][line width=0.75]      (0, 0) circle [x radius= 2.01, y radius= 2.01]   ;
\draw   (80,10) -- (133.88,40) -- (98.88,40) -- cycle ;
\draw   (80,40) -- (133.88,70) -- (98.88,70) -- cycle ;
\draw    (80,87) -- (80,110) ;
\draw [shift={(80,110)}, rotate = 90] [color={rgb, 255:red, 0; green, 0; blue, 0 }  ][fill={rgb, 255:red, 0; green, 0; blue, 0 }  ][line width=0.75]      (0, 0) circle [x radius= 2.01, y radius= 2.01]   ;
\draw   (110,130) -- (122.5,160) -- (97.5,160) -- cycle ;
\draw    (80,110) -- (80,133) ;
\draw   (196,130) -- (211,160) -- (181,160) -- cycle ;
\draw    (78,110) -- (196,130) ;
\draw [shift={(196,130)}, rotate = 9.62] [color={rgb, 255:red, 0; green, 0; blue, 0 }  ][fill={rgb, 255:red, 0; green, 0; blue, 0 }  ][line width=0.75]      (0, 0) circle [x radius= 2.01, y radius= 2.01]   ;
\draw   (80,180) -- (133.88,210) -- (98.88,210) -- cycle ;
\draw  [color={rgb, 255:red, 74; green, 144; blue, 226 }  ,draw opacity=1 ][dash pattern={on 4.5pt off 3pt}] (210,120) -- (216,170) -- (95,170) -- (80,109) -- cycle ;
\draw  [color={rgb, 255:red, 208; green, 2; blue, 27 }  ,draw opacity=1 ][dash pattern={on 4.5pt off 3pt}] (220,100) -- (220,220) -- (45,220) -- (45,100) -- cycle ;
\draw   (143,130) -- (155.5,160) -- (130.5,160) -- cycle ;

\draw (61,4.4) node [anchor=south west][inner sep=0.75pt]  [font=\small,color={rgb, 255:red, 0; green, 0; blue, 0 }  ,opacity=1 ]  {$\ttu_{0}$};
\draw (62,34.4) node [anchor=south west][inner sep=0.75pt]  [font=\small,color={rgb, 255:red, 0; green, 0; blue, 0 }  ,opacity=1 ]  {$\ttu_{1}$};
\draw (58,204.4) node [anchor=south west][inner sep=0.75pt]  [font=\small,color={rgb, 255:red, 0; green, 0; blue, 0 }  ,opacity=1 ]  {$\ttu_{\ttH}$};
\draw (46,174.4) node [anchor=south west][inner sep=0.75pt]  [font=\small,color={rgb, 255:red, 0; green, 0; blue, 0 }  ,opacity=1 ]  {$\ttu_{\ttH-1}$};
\draw (77,59.4) node [anchor=south west][inner sep=0.75pt]  [font=\small,color={rgb, 255:red, 0; green, 0; blue, 0 }  ,opacity=1 ]  {$\vdots $};
\draw (77,129.4) node [anchor=south west][inner sep=0.75pt]  [font=\small,color={rgb, 255:red, 0; green, 0; blue, 0 }  ,opacity=1 ]  {$\vdots $};
\draw (159,141.4) node [anchor=south west][inner sep=0.75pt]  [font=\small,color={rgb, 255:red, 0; green, 0; blue, 0 }  ,opacity=1 ]  {$\cdots $};
\draw (62,104.4) node [anchor=south west][inner sep=0.75pt]  [font=\small,color={rgb, 255:red, 0; green, 0; blue, 0 }  ,opacity=1 ]  {$\ttu_{\ell }$};
\draw (87,127.4) node [anchor=south west][inner sep=0.75pt]  [font=\footnotesize,color={rgb, 255:red, 0; green, 0; blue, 0 }  ,opacity=1 ]  {$\ttu_{\ell ,1}$};
\draw (119,127.4) node [anchor=south west][inner sep=0.75pt]  [font=\footnotesize,color={rgb, 255:red, 0; green, 0; blue, 0 }  ,opacity=1 ]  {$\ttu_{\ell ,2}$};
\draw (167,127.4) node [anchor=south west][inner sep=0.75pt]  [font=\footnotesize,color={rgb, 255:red, 0; green, 0; blue, 0 }  ,opacity=1 ]  {$\ttu_{\ell ,\ttk_{\ell }}$};
\draw (97,26.4) node [anchor=south west][inner sep=0.75pt]  [font=\small,color={rgb, 255:red, 74; green, 144; blue, 226 }  ,opacity=1 ]  {$\ttS_{0}$};
\draw (97,56.4) node [anchor=south west][inner sep=0.75pt]  [font=\small,color={rgb, 255:red, 74; green, 144; blue, 226 }  ,opacity=1 ]  {$\ttS_{1}$};
\draw (97,196.4) node [anchor=south west][inner sep=0.75pt]  [font=\small,color={rgb, 255:red, 74; green, 144; blue, 226 }  ,opacity=1 ]  {$\ttS_{\ttH-1}$};
\draw (161,156.4) node [anchor=south west][inner sep=0.75pt]  [font=\small,color={rgb, 255:red, 74; green, 144; blue, 226 }  ,opacity=1 ]  {$\ttS_{\ell }$};
\draw (198,206.4) node [anchor=south west][inner sep=0.75pt]  [font=\small,color={rgb, 255:red, 208; green, 2; blue, 27 }  ,opacity=1 ]  {$\ttT_{\ell }$};
\draw (99,147.4) node [anchor=south west][inner sep=0.75pt]  [font=\footnotesize,color={rgb, 255:red, 0; green, 0; blue, 0 }  ,opacity=1 ]  {$\ttS_{\ell ,1}$};
\draw (184,147.4) node [anchor=south west][inner sep=0.75pt]  [font=\footnotesize,color={rgb, 255:red, 0; green, 0; blue, 0 }  ,opacity=1 ]  {$\ttS_{\ell ,\ttk_{\ell }}$};
\draw (132,147.4) node [anchor=south west][inner sep=0.75pt]  [font=\footnotesize,color={rgb, 255:red, 0; green, 0; blue, 0 }  ,opacity=1 ]  {$\ttS_{\ell ,2}$};

\end{tikzpicture}
    \caption{Illustration of the spinal decomposition of $\ttT^\bullet$.}
    \label{fig:Tdot-spine}
\end{figure}

Recall that $\phi(\ttu_0)=\phi'(\ttu_0)=\varnothing$. It follows that $|\phi(\ttu_{\ttH})|=|\phi'(\ttu_{\ttH})|=\ttH$ and that for any $0\leq \ell\leq \ttH$, $\phi(\ttu_\ell)$ and $\phi'(\ttu_\ell)$ are the respective ancestors of $\phi(\ttu_{\ttH})$ and $\phi'(\ttu_{\ttH})$ of height $\ell$. In particular, we have
\begin{equation}
\label{height_LCS_vs_MIN}
\ttH\leq \Ht(\tau) \wedge \Ht(\tau').
\end{equation}
Moreover, for any $1\leq \ell\leq \ttH$, observe that $\phi$ and $\phi'$ respectively induce embeddings of $\ttS_{\ell-1}$ into $\Trim_{\phi(\ttu_\ell)}\tau$ and $\Trim_{\phi'(\ttu_\ell)}\tau'$, and of $\ttT_{\ell}$ into $\theta_{\phi(\ttu_\ell)}\tau$ and $\theta_{\phi'(\ttu_\ell)}\tau'$. It follows that
\begin{align}
\label{spinal_LCS}
\#\ttS_{\ell-1}&\leq \LCS^\bullet\big(\Trim_{\phi(\ttu_\ell)}\tau,\Trim_{\phi'(\ttu_\ell)}\tau'\big),\\
\label{spinal_above_LCS}
\#\ttT_{\ell}&\leq \LCS^\bullet\big(\theta_{\phi(\ttu_\ell)}\tau,\theta_{\phi'(\ttu_\ell)}\tau'\big).
\end{align}
Similarly, for any $0\leq\ell\leq \ttH-1$ and $1\leq \ttj\leq \ttk_{\ell}$, $\phi$ and $\phi'$ also induce embeddings of $\ttS_{\ell,\ttj}$ into $\theta_{\phi(\ttu_{\ell,\ttj})}\tau$ and $\theta_{\phi'(\ttu_{\ell,\ttj})}\tau'$, which yields that
\begin{equation}
\label{spinal_compo_LCS}
\#\ttS_{\ell,\ttj}\leq \LCS^\bullet\big(\theta_{\phi(\ttu_{\ell,\ttj})}\tau,\theta_{\phi'(\ttu_{\ell,\ttj})}\tau'\big).
\end{equation}
\smallskip

To prove that $h^{-(\eps+2\nu)}\#\ttT^\bullet \leq \Ht(\tau)\wedge\Ht(\tau')+1$, we proceed step by step with successive lemmas relating sizes of different subtrees of $\ttT^\bullet$. Informally, under the assumptions of the desired \cref{prop:determ_bootstrap}, we show the following in order: 
\begin{itemize}
    \item if all the $\#\ttS_\ell$ for $0\leq \ell\leq \ttH-1$ are small enough, then $\#\ttT^\bullet$ is small too (\cref{lem:B_cases});
    \item for $0\leq \ell\leq \ttH-1$, if all the $\# \ttS_{\ell,\ttj}$ for $1\leq \ttj\leq \ttk_\ell$ are small enough, then $\#\ttS_{\ell}$ is small (\cref{lem:A_cases});
    \item for $0\leq \ell \leq \ttH-1$, if $\#\ttT_{\ell+1}>h^\varepsilon (\ell+1)$ then all the $\# \ttS_{\ell,\ttj}$ for $1\leq \ttj\leq \ttk_\ell$ are small (\cref{lem:A1_cases});
    \item for $0\leq \ell\leq \ttH-1$, if at least one of the $\# \ttS_{\ell,\ttj}$ for $1\leq \ttj\leq \ttk_\ell$ is big but if $\#\ttT_{\ell+1}\leq h^\varepsilon(\ell+1)$, then $\#\ttS_{\ell}$ is still small (\cref{lem:A2_cases}).
\end{itemize}

\begin{lem}
\label{lem:B_cases}
Let $h\geq 1$, assume that $\Ht(\tau)\wedge\Ht(\tau')\leq h$ and that the event $A_h^{(1)}$ does not occur. If $h^{-(\varepsilon+\nu)}\# \ttS_\ell\leq \Ht(\tau)\wedge\Ht(\tau')+1$ for all $0\leq \ell\leq \ttH-1$, then $h^{-(\eps+2\nu)}\#\ttT^\bullet\leq  \Ht(\tau)\wedge\Ht(\tau')+1$.
\end{lem}

\begin{proof}
Here, set $M-1=\Ht(\tau)\wedge\Ht(\tau')\leq h$, $u=\phi(\ttu_{\ttH})$ and $u'=\phi'(\ttu_{\ttH})$. Since the vertex sets of $\ttS_0,\ldots,\ttS_{\ttH}=\{\ttu_{\ttH}\}$ partition the vertex set of $\ttT^\bullet$, we have 
\[\#\ttT^\bullet=1+\sum_{\ell=1}^{\ttH}\#\ttS_{\ell-1}=1+\sum_{\ell=1}^{\ttH}\#\ttS_{\ell-1}\wedge (h^{\varepsilon+\nu} M),\]
the second equality holding by assumption.
Then the inequality \eqref{spinal_LCS} yields that
\[\#\ttT^\bullet-1\leq \sum_{\ell=1}^{\ttH}\LCS^\bullet\big(\Trim_{\phi(\ttu_\ell)}\tau,\Trim_{\phi'(\ttu_\ell)}\tau'\big)\wedge (h^{\varepsilon+\nu} M).\]
Since $\phi(\ttu_\ell)$ and $\phi'(\ttu_\ell)$ are the respective ancestors of $u$ and $u'$ of height $\ell$, we can write 
\[\#\ttT^\bullet -1\leq \sum_{\substack{\varnothing\prec v\preceq u\\ \varnothing\prec v'\preceq u'}}\I{|v|=|v'|}\LCS^\bullet(\Trim_v\tau,\Trim_{v'}\tau')\wedge (h^{\varepsilon+\nu} M).\]
Recall that $1\leq M\leq h+1$ and $|u|=|u'|=\ttH$. Moreover, we know from \eqref{height_LCS_vs_MIN} that $\ttH\leq M-1$. Thus, the fact that $A_h^{(1)}$ is not realized yields that $\#\ttT^\bullet\leq h^{\eps+2\nu} M$, as desired.
\end{proof}

\begin{lem}
\label{lem:A_cases}
Let $h\geq 1$, assume that $\Ht(\tau)\wedge\Ht(\tau')\leq h$ and that the events $A_h^{(4)}$ and $A_h^{(4')}$ do not occur. For $0\leq \ell\leq \ttH-1$, if $h^{-\varepsilon}\# \ttS_{\ell,\ttj} \leq \Ht(\theta_{\phi(\ttu_{\ell,\ttj})}\tau)\wedge \Ht(\theta_{\phi'(\ttu_{\ell,\ttj})}\tau')+1$ for all $1\leq \ttj\leq \ttk_\ell$, then $h^{-(\eps+\nu)}\# \ttS_\ell\leq  \Ht(\tau)\wedge\Ht(\tau')+1$.
\end{lem}

\begin{proof}
Here, set $u=\phi(\ttu_\ell)$ and $u'=\phi'(\ttu_\ell)$, so that $|u|=|u'|=\ell$. We know that $\phi(\ttu_{\ell+1})$ and $\phi'(\ttu_{\ell+1})$ are respective children of $u$ and $u'$, so there are $1\leq i\leq \rk_u(\tau)$ and $1\leq i'\leq \rk_{u'}(\tau')$ such that $\phi(\ttu_{\ell+1})=u*i$ and $\phi'(\ttu_{\ell+1})=u'*i'$. Since $\{\ttu_\ell\}$ and the vertex sets of $\ttS_{\ell,1},\ldots,\ttS_{\ell,\ttk_\ell}$ partition the vertex set of $\ttS_\ell$, it holds that
\[\#\ttS_{\ell}-1=\sum_{\ttj=1}^{\ttk_\ell}\# \ttS_{\ell,\ttj}\leq h^\varepsilon\sum_{\ttj=1}^{\ttk_\ell}\left(1+\Ht(\theta_{\phi(\ttu_{\ell,\ttj})}\tau)\wedge \Ht(\theta_{\phi'(\ttu_{\ell,\ttj})}\tau') \right).\]
Now, recall that $\ttu_{\ell,1},\ldots,\ttu_{\ell,\ttk_\ell}$ are exactly the children of $\ttu_\ell$ in $\ttT^\bullet$ different from $\ttu_{\ell+1}$. Thus, $\phi(\ttu_{\ell,1}),\ldots,\phi(\ttu_{\ell,\ttk_\ell})$ are distinct children of $u$ in $\tau$, all different from $u*i$, and $\phi'(\ttu_{\ell,1}),\ldots,\phi'(\ttu_{\ell,\ttk_\ell})$ are distinct children of $u'$ in $\tau'$, all different from $u'*i'$. It follows that
\begin{align*}
h^{-\varepsilon}(\#\ttS_\ell-1)&\leq \sum_{\ttj=1}^{\ttk_\ell}1+\Ht(\theta_{\phi(\ttu_{\ell,\ttj})}\tau)\leq  \!\!\sum_{\substack{1\leq j\leq \rk_u(\tau)\\ j\neq i}}\!\! 1+\Ht(\theta_{u*j}\tau),\\
h^{-\varepsilon}(\#\ttS_\ell-1)&\leq \sum_{\ttj=1}^{\ttk_\ell}1+\Ht(\theta_{\phi'(\ttu_{\ell,\ttj})}\tau')\leq  \!\!\sum_{\substack{1\leq j\leq \rk_{u'}(\tau')\\ j\neq i'}}\!\! 1+\Ht(\theta_{u'*j}\tau').
\end{align*}
Then, the fact that neither $A_h^{(4)}$ nor $A_h^{(4')}$ occurs implies that
\[h^{-(\eps+\nu)}\#\ttS_\ell \leq  \Big(\sup_{\substack{1\leq j\leq \rk_u(\tau) \\ j\neq i}}\!\! 1+\Ht(\theta_{u*j}\tau)\Big)\wedge \Big(\sup_{\substack{1\leq j\leq \rk_{u'}(\tau') \\ j\neq i'}}\!\! 1+\Ht(\theta_{u'*j}\tau') \Big)\leq  \Ht(\tau)\wedge\Ht(\tau')+1.\]
Indeed, we have $\ell\leq \ttH-1\leq h-1$ by \eqref{height_LCS_vs_MIN}, and depending on whether $\Ht(\tau)$ is bigger than $\Ht(\tau')$, it holds that $\Ht(\theta_v\tau)\leq \Ht(\tau)\wedge \Ht(\tau')\leq h$ for all $v\in\tau$, or that $\Ht(\theta_v\tau')\leq \Ht(\tau)\wedge \Ht(\tau')\leq h$ for all $v\in\tau'$.
\end{proof}

\begin{lem}
\label{lem:A1_cases}
Let $h\geq 1$, assume that $\Ht(\tau)\wedge\Ht(\tau')\leq h$ and that the event $A_h^{(2)}$ does not occur. For $0\leq \ell\leq \ttH-1$, if $\# \ttT_{\ell+1}> h^\varepsilon (\ell+1)$, then $h^{-\varepsilon}\# \ttS_{\ell,\ttj} \leq  \Ht(\theta_{\phi(\ttu_{\ell,\ttj})}\tau)\wedge \Ht(\theta_{\phi'(\ttu_{\ell,\ttj})}\tau')+1$ for all $1\leq \ttj\leq \ttk_\ell$.
\end{lem}

\begin{proof}
As previously, set $u=\phi(\ttu_\ell)$ and $u'=\phi'(\ttu_\ell)$, so that $|u|=|u'|=\ell\leq h-1$, by \eqref{height_LCS_vs_MIN}. We know that $\phi(\ttu_{\ell+1})$ and $\phi'(\ttu_{\ell+1})$ are respective children of $u$ and $u'$, so there are $1\leq i\leq \rk_u(\tau)$ and $1\leq i'\leq \rk_{u'}(\tau')$ such that $\phi(\ttu_{\ell+1})=u*i$ and $\phi'(\ttu_{\ell+1})=u'*i'$. Then, the hypothesis on $\#\ttT_{\ell+1}$ and the inequality \eqref{spinal_above_LCS} entail that $\LCS^\bullet(\theta_{u*i}\tau,\theta_{u'*i'}\tau')> h^\varepsilon(\ell+1)$. Moreover, we have $\Ht(\theta_v\tau)\wedge \Ht(\theta_{v'}\tau')\leq \Ht(\tau)\wedge\Ht(\tau')\leq h$ for all $v\in\tau$ and $v'\in\tau'$. Therefore, the fact that $A_h^{(2)}$ is not realized ensures that for all $1\leq j\leq \rk_u(\tau)$ and $1\leq j'\leq \rk_{u'}(\tau')$ with $j\neq i$ and $j'\neq i'$, it holds that 
\[h^{-\varepsilon}\LCS^\bullet(\theta_{u*j}\tau,\theta_{u'*j'}\tau')\leq \Ht(\theta_{u*j}\tau)\wedge\Ht(\theta_{u'*j'}\tau')+1.\]
In particular, if $1\leq \ttj\leq \ttk_\ell$ then $\phi(\ttu_{\ell,\ttj})$ is a child of $u$ in $\tau$ different from $u*i$, and $\phi'(\ttu_{\ell,\ttj})$ is a child of $u'$ in $\tau'$ different from $u'*i'$, and so \[h^{-\varepsilon}\LCS^\bullet(\theta_{\phi(\ttu_{\ell,\ttj})}\tau,\theta_{\phi'(\ttu_{\ell,\ttj})}\tau')\leq \Ht(\theta_{\phi(\ttu_{\ell,\ttj})}\tau)\wedge\Ht(\theta_{\phi'(\ttu_{\ell,\ttj})}\tau')+1.\]
Then, the desired inequality follows from \eqref{spinal_compo_LCS}.
\end{proof}

\begin{lem}
\label{lem:A2_cases}
Let $h\geq 1$, assume that $\Ht(\tau)\wedge\Ht(\tau')\leq h$ and that the events $A_h^{(3)},A_h^{(5)},A_h^{(5')}$ do not occur. For $0\leq \ell \leq \ttH-1$ and $1\leq \ttj\leq \ttk_\ell$, if $h^{-\varepsilon}\# \ttS_{\ell,\ttj} > \Ht(\theta_{\phi(\ttu_{\ell,\ttj})}\tau)\wedge \Ht(\theta_{\phi'(\ttu_{\ell,\ttj})}\tau')+1$ and if $\#\ttT_{\ell+1}\leq h^\varepsilon(\ell+1)$, then $h^{-(\eps+\nu)}\# \ttS_\ell\leq \Ht(\tau)\wedge\Ht(\tau')+1$.
\end{lem}

\begin{proof}
First and foremost, note that for any $v\in\tau$ and $v'\in\tau'$, we have $\Ht(\theta_v\tau)\wedge \Ht(\theta_{v'}\tau')\leq \Ht(\tau)\wedge\Ht(\tau')\leq h$. We use this fact implicitly several times throughout the proof.

Again, set $u=\phi(\ttu_\ell)$ and $u'=\phi'(\ttu_\ell)$, so that $|u|=|u'|=\ell$. We know that $\phi(\ttu_{\ell,\ttj})$ and $\phi'(\ttu_{\ell,\ttj})$ are respective children of $u$ and $u'$, so there are $1\leq j\leq \rk_u(\tau)$ and $1\leq j'\leq \rk_{u'}(\tau')$ such that $\phi(\ttu_{\ell,\ttj})=u*j$ and $\phi'(\ttu_{\ell,\ttj})=u'*j'$. Then, the assumed lower bound on $\#\ttS_{\ell,\ttj}$ and the inequality \eqref{spinal_compo_LCS} entail that $h^{-\varepsilon}\LCS^\bullet(\theta_{u*j}\tau,\theta_{u'*j'}\tau')>  \Ht(\theta_{u*j}\tau)\wedge \Ht(\theta_{u'*j'}\tau')+1$. Therefore, the fact that $A_h^{(3)}$ is not realized yields that $\rk_u(\tau)\leq \sqrt{h^\varepsilon(\ell+1)}$ or $\rk_{u'}(\tau')\leq \sqrt{h^\varepsilon(\ell+1)}$. In fact, the symmetry of the roles of $\tau$ and $\tau'$ (since both $A_h^{(5)}$ and $A_h^{(5')}$ do not occur) allows us to assume that $\rk_u(\tau)\leq\sqrt{h^\varepsilon(\ell+1)}$ without loss of generality.

The vertex set of $\ttS_\ell$ is still partitioned by $\{\ttu_\ell\}$ and the vertex sets of $\ttS_{\ell,1},\ldots,\ttS_{\ell,\ttk_\ell}$. Furthermore, by construction of the spine of $\ttT^\bullet$, it holds that $\#\ttS_{\ell,\tti}\leq \#\ttT_{\ell+1}$ for all $1\leq\tti\leq\ttk_\ell$. Then, the hypothesis that $\#\ttT_{\ell+1}\leq h^\varepsilon(\ell+1)$ allows us to write
\[\#\ttS_\ell-1= \sum_{\tti=1}^{\ttk_\ell}\#\ttS_{\ell,\tti}=\sum_{\tti=1}^{\ttk_\ell}\#\ttS_{\ell,\tti}\wedge (h^\varepsilon(\ell+1)).\]
Using \eqref{spinal_compo_LCS}, we then get
\[\#\ttS_\ell-1\leq \sum_{\tti=1}^{\ttk_\ell}\LCS^\bullet(\theta_{\phi(\ttu_{\ell,\tti})}\tau,\theta_{\phi'(\ttu_{\ell,\tti})}\tau')\wedge(h^\varepsilon(\ell+1)) \leq \sum_{\tti=1}^{\ttk_\ell}(\#\theta_{\phi(\ttu_{\ell,\tti})}\tau)\wedge( h^\varepsilon(\ell+1)).\]
Next, recall that $\phi(\ttu_{\ell,1}),\ldots,\phi(\ttu_{\ell,\ttk_\ell})$ are distinct children of $u$ in $\tau$ to find that
\[\#\ttS_\ell-1\leq \sum_{j=1}^{k_u(\tau)}(\#\theta_{u*(j)}\tau)\wedge( h^\varepsilon(\ell+1)).\]
Then, since we know from \eqref{height_LCS_vs_MIN} that $|u|+1=\ell+1\leq \ttH\leq \Ht(\tau)\wedge\Ht(\tau')\leq h$, the non-realization of $A_h^{(5)}$ implies that $\#\ttS_\ell\leq h^{\eps+\nu}(\ell+1)+1\leq h^{\eps+\nu}\cdot \Ht(\tau)\wedge\Ht(\tau')+h^{\varepsilon+\nu}$.
\end{proof}

\begin{proof}[Proof of \cref{prop:determ_bootstrap}]
Combining \cref{lem:A_cases,lem:A1_cases,lem:A2_cases} yields that $h^{-(\eps+\nu)}\# \ttS_\ell\leq 1+\Ht(\tau)\wedge\Ht(\tau')$ for all $0\leq \ell\leq \ttH-1$. Then, applying \cref{lem:B_cases} concludes the proof.
\end{proof}

\subsection{Proof of the bootstrap: probabilistic bounds}
\label{sec:bound_bootstrap}

All that is left now is to prove \cref{prop:bound_bootstrap}. We prove points (1) to (5) separately. Nonetheless, we follow the same basic method each time: use a union bound to control the probability of the considered events by the expectation of a sum of indicator functions, compute the latter with the Many-to-One principle (\cref{many-to-one}), and conclude thanks to some of the estimates gathered in \cref{sec:preliminaries}.

To bound $\prob(A_h^{(1)})$, we combine the big jumps principle with \cref{lem:size-biased_tail} on the size of root-biased trees.

\begin{proof}[Proof of \cref{prop:bound_bootstrap}(1)]
    For any $m,h\geq 1$, any plane trees $t,t'$, and any $u\in t$ and $u'\in t'$, define $F_{h,m}(t,u;t',u')\in\{0,1\}$ as the indicator function such that 
    \[F_{h,m}(t,u;t',u')=1 \Longleftrightarrow \sum_{\substack{\varnothing\prec v\preceq u\\ \varnothing\prec v'\preceq u'}}\I{|v|=|v'|}\LCS^\bullet(\Trim_v t,\Trim_{v'} t')\wedge (h^{\varepsilon+\nu} m)\geq h^{\varepsilon+2\nu}m-1.\]
    Note that $F_{h,m}(t,u;t',u')=F_{h,m}(\Cut_u t,u;\Cut_{u'} t',u')$. Then, a union bound gives that
    \[\prob(A_h^{(1)})\leq \sum_{m=1}^{h+1}\sum_{\ell=0}^{m-1}\E\bigg[\sum_{u\in\tau, u'\in\tau'}\I{|u|=\ell}\I{|u'|=\ell}F(\Cut_u \tau,u;\Cut_{u'} \tau',u')\bigg].\]
    Then, noting the independence of $\tau$ and $\tau'$, applying the Many-to-One principle (\cref{many-to-one}) on $\tau$ and on $\tau'$ leads to
    \begin{equation}
\label{bound_bootstrap-1_tool}
    \prob(A_h^{(1)}) \leq \sum_{m=1}^{h+1} \sum_{\ell=0}^{m-1}\E\big[F_{h,m}(\Cut_{U_\ell}\tau_\infty,U_\ell;\Cut_{U_\ell'}\tau_\infty',U_\ell')\big]
    \end{equation}
    where $\big(\tau_\infty,(J_i)_{i\geq 1}\big)$ and $\big(\tau_\infty',(J_i')_{i\geq 1}\big)$ are two independent size-biased Bienaym\'{e} trees with offspring distribution $\mu$ and $\mu'$, respectively, and where $U_i = (J_1,\ldots,J_i)$ and $U_i' = (J_1',\ldots,J_i')$ for all $i\geq 0$.
    
    Now, set $\xi_{h,m,\ell}=\E\big[F_{h,m}(\Cut_{U_\ell}\tau_\infty,U_\ell;\Cut_{U_\ell'}\tau_\infty',U_\ell')\big]$ for any $h\geq m-1\geq \ell\geq 0$, and denote by $(\tau_{*,i})_{i\geq 1}$ and $(\tau_{*,i}')_{i\geq 1}$ two independent sequences of i.i.d.~root-biased Bienaym\'{e} trees with offspring distribution $\mu$ and $\mu'$ respectively. Using first \cref{root-biased_VS_size-biased} and then relying on that $\ell\leq m$ and $1\leq m \leq h+1$, we get
    \begin{align*}
    \xi_{h,m,\ell}&=\prob\left(\sum_{i=1}^\ell \LCS^\bullet(\tau_{*,i},\tau_{*,i}')\wedge  (h^{\e + \nu} m) \geq h^{\e + 2\nu}m-1\right)\\
    &\leq \prob\left(\sum_{i=1}^m \LCS^\bullet(\tau_{*,i},\tau_{*,i}')\wedge  (h^{\e + \nu} m^{1+\nu/2}) \geq \tfrac{1}{4}h^{\e + 3\nu/2}m^{1+\nu/2}\right)
    \end{align*}
    as soon as $h$ is big enough so that $h^{\varepsilon+2\nu}\geq 2$ and $h^{\nu/2}\geq \tfrac{1}{2}(h+1)^{\nu/2}$. Moreover, it holds
    \[\prob(\LCS^\bullet(\tau_{*,1}, \tau_{*,1}') \geq m) \leq \prob(\#\tau_{*,1}\geq m)\prob(\#\tau_{*,1}' \geq m) = O(m^{-1})\]
    thanks to the trivial bound $\LCS^\bullet(\tau_{*,1}, \tau_{*,1}')\leq (\#\tau_{*,1})\wedge (\#\tau_{*,1}')$ and to \cref{lem:size-biased_tail}. Then, applying \cref{obj_1} gives a constant $C > 0$ such that $\xi_{h,m,\ell}\leq C\exp(-\tfrac{1}{4}h^{\nu/2})$ for all $h$ big enough and for all $m,\ell\geq 0$ with $\ell\leq m-1\leq h$. Combining this with \eqref{bound_bootstrap-1_tool} yields that $\prob(A_h^{(1)})\leq C (h+1)^2 \exp(-\tfrac{1}{4}h^{\nu/2})$ for all $h$ large enough, and the desired bound follows.
\end{proof}

To bound $\prob(A_h^{(2)})$, we use \cref{single_tree_size} on the size of Bienaymé trees with finite variance offspring distributions.

\begin{proof}[Proof of \cref{prop:bound_bootstrap}(2)]
    For any $h\geq 1$ and $\ell\geq 0$, and any plane trees $t,t'$, define
    \begin{align*}
        F_{h}(t,t')&=\I{\Ht(t)\wedge \Ht(t')\leq h\, ;\, h^{-\varepsilon}\LCS^\bullet(t,t')> \Ht(t)\wedge\Ht(t')+1},\\
        G_{h,\ell}(t,t')&=\sum_{\substack{1 \leq i \ne j \leq \rk_\emptyset(t) \\ 1 \leq i' \ne j' \leq \rk_{\emptyset}(t')}}\!\! \I{\LCS^\bullet(\theta_{i} t,\theta_{i'}t')>h^\varepsilon(\ell+1)} F(\theta_j t,\theta_{j'}t').
    \end{align*}
    Note that $\E[F_h(\tau,\tau')]=p_\varepsilon(h)$. With this notation, observe that a union bound entails that
    \[\prob(A_h^{(2)})\leq \sum_{\ell=0}^{h-1}\E\bigg[\sum_{u\in\tau,u'\in\tau'}\I{|u|=\ell}\I{|u'|=\ell}G_{h,\ell}(\theta_u\tau,\theta_{u'}\tau')\bigg].\]
    As in the previous proof, we use the Many-to-One principle twice (\cref{many-to-one}) and the independence of $\tau$ and $\tau'$ to get that
    \begin{equation}\label{eq:A1i}
    \prob(A_h^{(2)}) \leq \sum_{\ell=0}^{h-1} \E\big[G_{h,\ell}(\tau,\tau')\big].
    \end{equation}
    Next, using the branching property of Bienaym\'{e} trees, we compute that
    \begin{align*}
     \E\big[G_{h,\ell}(\tau,\tau')\big]&= \ex\left[ \sum_{\substack{1 \leq i \ne j \leq \rk_\emptyset(\tau) \\ 1 \leq i' \ne j' \leq \rk_{\emptyset}(\tau')}}\!\! \prob\big(\LCS^\bullet(\tau,\tau')>h^\varepsilon(\ell+1)\big)\E\big[F_h(\tau,\tau')\big]\right]\\
    &\leq \E\big[ \rk_\emptyset(\tau)^2k_\emptyset(\tau')^2\big] \prob\big(\LCS^\bullet(\tau,\tau')>h^\varepsilon(\ell+1)\big)\E\big[F_h(\tau,\tau')\big]\\
    &= \E\big[\rk_\emptyset(\tau)^2\big]\E\big[\rk_\emptyset(\tau')^2\big] p_\varepsilon(h)\cdot \prob\big(\LCS^\bullet(\tau,\tau')>h^\varepsilon(\ell+1)\big)
    \end{align*}
    for any $h\geq 1$ and $\ell\geq 0$. By assumption, $\ex[\rk_\emptyset(\tau)^2] \ex[\rk_\emptyset(\tau')^2] < \infty$. Moreover, the inequality $\LCS^\bullet(\tau,\tau') \leq (\# \tau) \wedge (\# \tau')$ and \cref{single_tree_size} let us write 
    \[\prob\big(\LCS^\bullet(\tau,\tau')>h^\varepsilon(\ell+1)\big)\leq \prob\big(\# \tau \geq h^\varepsilon (\ell+1)\big)\prob\big(\# \tau' \geq h^\varepsilon (\ell+1)\big) \leq \frac{c}{h^\varepsilon (\ell+1)},\]
    where $c>0$ is some constant. Placing everything back into \eqref{eq:A1i}, we obtain
    \[\prob(A_h^{(2)}) \leq c\ex[\rk_\emptyset(\tau)^2]\ex[\rk_\emptyset(\tau')^2]\cdot p_\varepsilon(h)\cdot h^{-\varepsilon}\sum_{\ell=1}^{h}  \frac{1}{\ell},\]
    which gives us the desired result.
\end{proof}

To bound $\prob(A_h^{(3)})$, we rely on the finiteness of the variances of the offspring distributions.

\begin{proof}[Proof of \cref{prop:bound_bootstrap}(3)]
    Very similarly to the previous proof, for any $h\geq 1$ and $\ell\geq 0$, and any plane trees $t,t'$, write
    \begin{align*}
        F_{h}(t,t')&=\I{\Ht(t)\wedge \Ht(t')\leq h\, ;\, h^{-\varepsilon}\LCS^\bullet(t,t')> \Ht(t)\wedge\Ht(t')+1},\\
        G_{h,\ell}(t,t')&=\I{\rk_\varnothing(t)\wedge \rk_\varnothing(t')>\sqrt{h^\varepsilon(\ell+1)}}\sum_{\substack{1 \leq j \leq \rk_\emptyset(t) \\ 1 \leq  j' \leq \rk_{\emptyset}(t')}}\!\!  F(\theta_j t,\theta_{j'}t').
    \end{align*}
    Then, a union bound and then the Many-to-One principle (\cref{many-to-one}) yields that
    \begin{equation}\label{eq:A2i_unionbound}
        \prob(A_h^{(3)}) \leq \sum_{\ell = 0}^{h-1} \E\big[G_{h,\ell}(\tau,\tau')\big].
    \end{equation}
    From here, consider $h\geq 1$ and $\ell\geq 0$. Using the branching property of Bienaym\'{e} trees and the independence of $\tau$ and $\tau'$, we write
    \begin{align*}
    \E\big[G_{h,\ell}(\tau,\tau')\big] &= \ex\left[ \I{\rk_\varnothing(\tau)\wedge \rk_\varnothing(\tau')>\sqrt{h^\varepsilon(\ell+1)}}\sum_{1 \leq j \leq \rk_\emptyset(\tau)}\sum_{1 \leq j' \leq \rk_\emptyset(\tau')} \E\big[F_h(\tau,\tau')\big] \right] \\
    &= \E\Big[\I{\rk_\varnothing(\tau)>\sqrt{h^\varepsilon(\ell+1)}}k_\varnothing(\tau)\Big] \E\Big[\I{\rk_\varnothing(\tau')>\sqrt{h^\varepsilon(\ell+1)}}k_\varnothing(\tau')\Big] \E\big[F_h(\tau,\tau')\big]\\
    &\leq \ex\left[ \frac{k_\emptyset(\tau)^2}{\sqrt{h^\varepsilon (\ell+1)}} \right] \ex\left[ \frac{k_\emptyset(\tau')^2}{\sqrt{h^\varepsilon (\ell+1)}} \right] p_\varepsilon(h),
    \end{align*}
    where in the last inequality we use the trivial bound $\I{k_\emptyset(\tau) \geq x} \leq \frac{k_\emptyset(\tau)}{x}$. Injecting the above into \eqref{eq:A2i_unionbound}, we obtain
    $$
    \prob(A_h^{(3)}) \leq \ex[\rk_\emptyset(\tau)^2]\ex[\rk_\emptyset(\tau')^2]\cdot  p_\varepsilon(h)\cdot h^{-\varepsilon}\sum_{\ell = 1}^{h} \frac{1}{\ell},
    $$
    which gives us the desired result.
\end{proof}

To bound $\prob(A_h^{(4)})+\prob(A_{h}^{(4')})$, we conjugate the big jumps principle with \cref{single_tree_height} on the height of Bienaymé trees.

\begin{proof}[Proof of \cref{prop:bound_bootstrap}(4)]
    We shall prove the bound for $\prob(A_{h}^{(4)})$, which then, by symmetry, also holds for $\prob(A_{h}^{(4')})$, which then implies the desired result on the sum $\prob(A_{h}^{(4)})+\prob(A_{h}^{(4')})$.
    Using the union bound and the Many-to-One principle (\cref{many-to-one}), as was done in the previous proofs, we obtain
    \begin{align*}
        \prob(A_{h}^{(4)}) &\leq \sum_{\ell=0}^{h-1} \ex\left[ \sum_{u \in \tau}\I{|u| = \ell} F_h(\theta_u \tau) \right] =  \sum_{\ell=0}^{h-1} \E[F_h(\tau)] = h\E[F_h(\tau)],
    \end{align*}
    where for any $h\geq 1$ and plane tree $t$, $F_h(t)\in\{0,1\}$ stands for the indicator function that there exists $1\leq i\leq \rk_\varnothing(t)$ such that $\Ht(\theta_j t)\leq h$ for all $1\leq j\leq \rk_\varnothing(t)$ with $j\neq i$ and
    \[\sum_{\substack{1\leq j\leq \rk_\varnothing(t) \\ j\neq i}}(1+\Ht(\theta_j t))\geq -1+h^\nu\sup_{\substack{1\leq j\leq \rk_\varnothing(t) \\ j\neq i}}(1+\Ht(\theta_j t)).\]
    
    For the rest of the proof, denote by $(X_j)_{j\geq 1}$ a sequence of independent random variables, all with the same distribution as $1+\Ht(\tau)$, and denote
    \[B_k = \left\{ X_j \leq 2h \text{ for all $1\leq j\leq k$}\, ;\, \sum_{j=1}^{k} X_j \geq \tfrac{1}{2} h^\nu \sup_{1 \leq j \leq k} X_j \right\}\]
    for all $k\geq 1$. Taking the supremum over the support of $k_\emptyset(\tau)$ and applying the union bound, we obtain by branching property of Bienaymé trees that
    \[\E[F_h(\tau)]\leq \sup_{k \geq 1} (k+1)\prob(B_k) \leq 2\sup_{k \geq 1} k\prob(B_k),\]
    as soon as $h$ is large enough so that $\tfrac{1}{2}h^\nu\geq 1$. The rest of the proof is devoted to bounding the probability of the events $B_k$. Before doing so, note that we have $cx^{-1} \leq \prob(X_1 \geq x) \leq Cx^{-1}$ for some constants $c,C > 0$ by \cref{single_tree_height}.
    
     First, for $k<\tfrac{1}{2}h^\nu$, $B_k=\emptyset$. 
     Moreover, by using the inequality $\prob(X_1\le x)\le 1-c x^{-1}\le e^{-c/x}$, we get 
    \[\sup_{k\ge h^{2}} k \prob(B_k)\le  \sup_{k\ge h^{2}} k \prob\left(X_1\le 2h\right)^k \le \sup_{k\ge h^{2}} k e^{-c k/2h}.\] 
    For $h$ large enough, this equals $h^{2}e^{-c h/2}$. 
    
    From here on, assume that $\tfrac{1}{2}h^\nu \le k <h^{2}$. Since $\nu\in(0,1)$, we observe that 
    \begin{align*}
    \prob\big(\sup_{1\le j \le k}X_j<k^{1-\nu/8}\big) &\le e^{-ck^{\nu/8}}\le e^{-\frac{c}{2}h^{\nu^2 /8}}, \quad\text{ and }\\
    \prob\big(\sup_{1\le j \le k}X_j>\exp(\tfrac14 h^{\nu/2})\big) &\le C k e^{ - \frac{1}{4}h^{\nu/2}} \le C h^{2} e^{ - \frac{1}{4}h^{\nu/2}}.
    \end{align*}
    Thus, taking a union bound over the remaining possible values $m$ of $\sup_{1\le j \le k}X_j$, we find 
    \[\prob(B_k)\le e^{-\frac{c}{2}h^{\nu^2/8 }}+ C h^{2} e^{ - \frac{1}{4}h^{\nu/2}} + \sum_{m=k^{1-\delta}}^{\exp(\tfrac14 h^{\nu/2})} \prob\Big(\sum_{j=1}^k (X_j \wedge m) \ge \tfrac{1}{2}h^\nu m \Big). \]
    By \cref{obj_1} with $\gamma=\nu/8$, we find that for $m\ge k^{1+\nu/8}$ that 
    \[\prob\Big(\sum_{j=1}^k (X_j \wedge m) \ge \tfrac{1}{2}h^\nu m  \Big)\le Ce^{-\frac{1}{2}h^\nu}.\]
    Again by \cref{obj_1} with $\gamma=\nu/8$, we find that for $k^{1-\nu/8}\le m\le k^{1+\nu/8}$ 
    \[\prob\Big(\sum_{j=1}^k (X_j \wedge m) \ge \tfrac{1}{2} h^\nu m  \Big) \le \prob\Big(\sum_{j=1}^k (X_j \wedge k^{1+\nu/8}) \ge \tfrac{1}{2}h^\nu k^{1-\nu/8}  \Big)\le Ce^{-\frac{1}{2}h^\nu k^{-\nu/4}},\] 
    which is at most $Ce^{-\frac{1}{2}h^{\nu/2}}$ because $k\le h^2$. Then, combining the terms yields the result. 
\end{proof}

To bound $\prob(A_h^{(5)})+\prob(A_{h}^{(5')})$, we associate the big jumps principle with \cref{single_tree_size} on the size of standard Bienaymé trees.

\begin{proof}[Proof of \cref{prop:bound_bootstrap}(5)]
   
    We shall prove the bound for $\prob(A_{h}^{(5)})$, and the bound for the sum $\prob(A_{h}^{(5)})+\prob(A_{h}^{(5')})$ then follows by symmetry. Let $(\tau_j)_{j\geq 1}$ be a sequence of i.i.d.~Bienaym\'{e} trees with offspring distribution $\mu$.
    
    Applying successively the union bound, the Many-to-One principle (\cref{many-to-one}), and the branching property of Bienaymé trees (as was done with all the previous lemmas), write
    \begin{align*}
    \prob(A_{h}^{(5)}) &\leq \sum_{\ell = 0}^{h-1} \prob\Big( \rk_\emptyset(\tau) \leq \sqrt{h^\varepsilon(\ell+1)}\, ;\, \sum_{j=1}^{k_\emptyset(\tau)} \# \theta_j\tau \wedge \big(h^\varepsilon (\ell+1)\big) \geq h^\nu \big(h^\varepsilon (\ell+1)\big)\Big)\\
    &\leq \sum_{\ell = 0}^{h-1} \prob\Big( \sum_{j=1}^{\sqrt{h^\varepsilon(\ell+1)}} \# \tau_j \wedge \big(h^\varepsilon (\ell+1)\big) \geq h^\nu \big(h^\varepsilon (\ell+1)\big)\Big).
    \end{align*}
    From here, a direct application of \cref{obj_0-1and1-2} with $\alpha=1/2$ (allowed by \cref{single_tree_size}) yields a constant $C > 0$ such that
    \[\prob(A_{h}^{(5)}) \leq \sum_{\ell = 0}^{h-1} C\exp(-h^\nu) = Ch\exp(-h^\nu). \qedhere\]
\end{proof}

\section{Unrooted conditioned setting, and reduction to common subtrees with a bounded number of leaves}\label{sec:Unrooted_LCS}

\subsection{Strategy}
The desired \cref{thm:main} can be understood by saying that for large critical Bienaymé trees, the \emph{size} of a largest (unrooted) common subtree is roughly proportional to the size of a largest common subtree with three leaves. As an intermediate step, we show in this section that the size of a largest common subtree can be approximated by a multiple of the total length (that is, the number of edges) of a largest common subtree with boundedly many leaves. Hence, for two trees $t,t'$ and an integer $N\geq 3$, let us denote by $\LCS_N(t,t')$ the maximum total length of a (unrooted) common subtree of $t$ and $t'$ with at most $N$ leaves: 
\begin{equation}
\label{expr_dscr_LCS_N}
\LCS_N(t,t')=\max\big\{\#\ttt -1 \, :\, \ttt \text{ common subtree of $t,t'$ with at most $N$ leaves} \big\}.
\end{equation}
The goal of this section is to prove the following preliminary result for \cref{thm:main}.
\begin{thm}
\label{thm:size-to-length}
Keep the same assumptions and notation as in \cref{thm:main}. Then for all $\delta>0$, there exists an integer $N\geq 3$ such that
\[\P\big(\big|\LCS(\tau_n,\tau_n')-c_{\mu,\mu'}\LCS_N(\tau_n,\tau_n')\big|>\delta\sqrt{n}\big)\to 0.\]
\end{thm}
In the next section, our main result then follows by showing that we can in fact pick $N=3$.

To prove this theorem, we mostly work with \emph{unconditioned} Bienaymé trees. Thus, let us fix two critical offspring distributions $\mu,\mu'$ such that $\sum_{k\geq 0}k^{2+\kappa}\mu(k)<\infty$ and $\sum_{k\geq 0}k^2\mu'(k)<\infty$ for some $\kappa>0$, and denote by $\tau,\tau'$ two independent Bienaymé trees of respective offspring distributions $\mu$ and $\mu'$. Recall that for a plane tree $t$, we respectively denote by $\#t$, $\Ht(t)$, and $\Delta(t)$ its size, height, and maximum offspring size. The following proposition represents the bulk of the work needed to show the desired \cref{thm:size-to-length}.
\begin{prop}
\label{prop:size-to-length_estimate}
For all $\varepsilon\in (0,1/4)$, $\delta>0$, and $m\geq 1$, there exists an integer $N\geq 3$ such that for all $n$ large enough, it holds that
\begin{multline*}
\P\big(\#\tau\vee\#\tau'\leq n\, ;\, \Ht(\tau)\leq n^{1/2+\varepsilon} \, ;\, \Delta(\tau)\leq n^{1/2-2\varepsilon}\, ;\, \Delta(\tau')\leq n^{1/2+\varepsilon}\, ;\,\\
\big|\LCS(\tau,\tau')-c_{\mu,\mu'}\LCS_N(\tau,\tau')\big|>\delta\sqrt{n}\big)\leq n^{-m}.
\end{multline*}
\end{prop}

\begin{proof}[Proof of \cref{thm:size-to-length} using \cref{prop:size-to-length_estimate}]
By \cref{lem:maximum_degree}, we have some $\varepsilon \in (0,1/4)$ such that $\P(\Delta(\tau_n)\geq n^{1/2-2\varepsilon})\to 0$. Then, we get from the same lemma that $\P(\Delta(\tau_n')\geq n^{1/2+\varepsilon})\to 0$ and from \cref{prop:conditioned_tree_height} that $\P(\Ht(\tau_n)\geq n^{1/2+\varepsilon})\to 0$. Therefore, by the law of total probability and the union bound, \cref{thm:size-to-length} follows if we find an integer $N\geq 3$ such that $\xi_n(N)\to 0$, where $\xi_n(N)$ stands for the probability
\begin{multline*}
\xi_n(N) \coloneqq \P\big(\Ht(\tau_n)\leq n^{1/2+\varepsilon} \, ;\, \Delta(\tau_n)\leq n^{1/2-2\varepsilon}\, ;\, \Delta(\tau_n')\leq n^{1/2+\varepsilon}\, ;\, \\
\big|\LCS(\tau_n,\tau_n') - c_{\mu,\mu'}\LCS_N(\tau_n,\tau_n')\big|>\delta\sqrt{n}\big).
\end{multline*}

Recall that the conditional law of $(\tau,\tau')$ given $\#\tau=\#\tau'=n$ is the law of $(\tau_n,\tau_n')$. Applying \cref{prop:size-to-length_estimate} with $m=4$ provides us with an integer $N\geq 3$ such that
\[\xi_n(N)\leq  \P(\#\tau=n)^{-1}\P(\#\tau'=n)^{-1}\cdot n^{-4}\]
for all $n$ large enough, by crudely bounding the conditional probability $\xi_n(N)$ by the ratio of the probabilities. Now, it follows from \cref{single_tree_size} that there is a constant $c>0$ such that $\xi_n(N)\leq c( n^{3/2})^2 n^{-4}= cn^{-1}$ for all $n$ large enough. This concludes the proof.
\end{proof}



\subsection{Proof of \cref{prop:size-to-length_estimate}}

In what follows, we fix $\varepsilon\in (0,1/4)$ and $\tilde{\delta}>0$ and explain how we will prove \cref{prop:size-to-length_estimate}. The intuition behind \cref{prop:size-to-length_estimate} (and \cref{thm:size-to-length}) is that, with high probability, a largest common subtree of $\tau$ and $\tau'$ consists approximately of a tree with at most $N$ leaves, which we hereafter call \emph{skeleton}, whose vertices are decorated with basically i.i.d.~small-sized masses stemming from pendant forests of micro-sized trees. (Recall that, in the next section, we show that $N=3$ suffices.) The proportionality relation between the size of the common subtree and the total length of its skeleton then appears as a consequence of the law of large numbers. Thus, it is natural to define the skeleton as those edges that separate that largest common subtree into two masses that are larger than micro-sized. Then, to justify our heuristic intuition, we need to show that with high probability, on the one hand, the small-sized decorating masses do not violate the law of large numbers, and on the other hand, the skeleton indeed has boundedly many leaves.
\smallskip

To give precise meaning to small-sized, micro-sized, and boundedly, let us choose $\alpha\in(0,1/4)$ and an integer $d\geq 3$. We shall understand a tree to be small-sized when its size is smaller than $n^{1/2-\alpha}$, micro-sized when it is smaller than $n^{1/2-2\alpha}$, and its number of leaves to be bounded when it is not larger than the number of vertices at height $d$ in the infinite rooted $d$-regular tree, that is, $d(d-1)^{d-1}$. To verify the law of large numbers heuristic above and prove \cref{prop:size-to-length_estimate}, we need to show that it is very unlikely that any common subtree of $\tau$ and $\tau'$ contains one of the four bad structures `sausage/twig', `skewer', `flower', `bush' that we present below.
\smallskip

First, we introduce some notation to manipulate trees seen as graphs. Let $\ttt$ be a tree. For $\ttu,\ttv\in\ttt$, let $\llbracket \ttu,\ttv\rrbracket$ stand for the path between $\ttu$ and $\ttv$, and set $\rrbracket \ttu,\ttv\llbracket=\llbracket \ttu,\ttv\rrbracket\setminus\{\ttu,\ttv\}$. For a non-empty subtree $\ttt'\subseteq \ttt$ and $\ttu\in\ttt'$, denote by $\mathrm{Compo}(\ttt-\ttt';\ttu)$ the component of $\ttu$ in the forest obtained by removing all the edges of $\ttt'$ from $\ttt$, and rooting it at $\ttu$. Note that $\{ \mathrm{Compo}(\ttt-\ttt',\ttu) : \ttu\in\ttt' \}$ partitions $\ttt$.
\smallskip

We now introduce the four structures whose absence validates the heuristic description of a largest common subtree; see \cref{fig:barbecue}.
\medskip

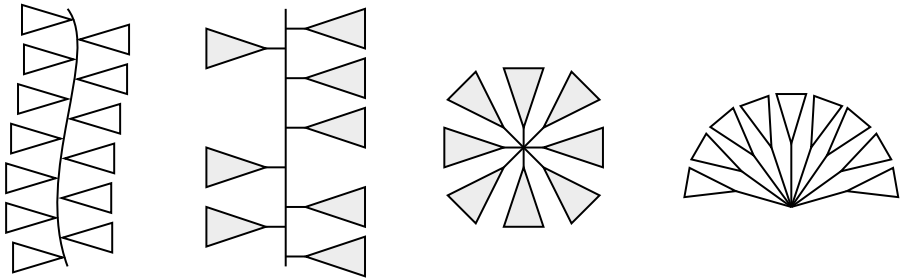
\begin{figure}[b]
    \centering
    \tikzset{every picture/.style={line width=0.75pt}} 

\begin{tikzpicture}[x=0.75pt,y=0.75pt,yscale=-1,xscale=1]

\draw   (460,102) -- (477.9,83) -- (485.4,96) -- cycle ;
\draw    (70,20) .. controls (87.4,42.83) and (52.2,102.43) .. (70,150) ;
\draw    (180,20) -- (180,150) ;
\draw    (180,30) -- (190,30) ;
\draw   (72,25.5) -- (47,33) -- (47,18) -- cycle ;
\draw   (73,45.5) -- (48,53) -- (48,38) -- cycle ;
\draw   (70,65.5) -- (45,73) -- (45,58) -- cycle ;
\draw   (66.5,85.5) -- (41.5,93) -- (41.5,78) -- cycle ;
\draw   (76,35.5) -- (101,28) -- (101,43) -- cycle ;
\draw   (75,55.5) -- (100,48) -- (100,63) -- cycle ;
\draw   (71.5,75.5) -- (96.5,68) -- (96.5,83) -- cycle ;
\draw   (68.5,95.5) -- (93.5,88) -- (93.5,103) -- cycle ;
\draw   (67,115.5) -- (92,108) -- (92,123) -- cycle ;
\draw   (67.5,135.5) -- (92.5,128) -- (92.5,143) -- cycle ;
\draw   (64,105.5) -- (39,113) -- (39,98) -- cycle ;
\draw   (64,125.5) -- (39,133) -- (39,118) -- cycle ;
\draw   (67.5,145.5) -- (42.5,153) -- (42.5,138) -- cycle ;
\draw  [fill={rgb, 255:red, 237; green, 237; blue, 237 }  ,fill opacity=1 ] (190,30) -- (220,20) -- (220,40) -- cycle ;
\draw  [fill={rgb, 255:red, 237; green, 237; blue, 237 }  ,fill opacity=1 ] (170,40) -- (140,50) -- (140,30) -- cycle ;
\draw    (170,40) -- (180,40) ;
\draw    (180,55) -- (190,55) ;
\draw  [fill={rgb, 255:red, 237; green, 237; blue, 237 }  ,fill opacity=1 ] (190,55) -- (220,45) -- (220,65) -- cycle ;
\draw    (180,80) -- (190,80) ;
\draw  [fill={rgb, 255:red, 237; green, 237; blue, 237 }  ,fill opacity=1 ] (190,80) -- (220,70) -- (220,90) -- cycle ;
\draw  [fill={rgb, 255:red, 237; green, 237; blue, 237 }  ,fill opacity=1 ] (170,100) -- (140,110) -- (140,90) -- cycle ;
\draw    (170,100) -- (180,100) ;
\draw  [fill={rgb, 255:red, 237; green, 237; blue, 237 }  ,fill opacity=1 ] (170,130) -- (140,140) -- (140,120) -- cycle ;
\draw    (170,130) -- (180,130) ;
\draw    (180,120) -- (190,120) ;
\draw  [fill={rgb, 255:red, 237; green, 237; blue, 237 }  ,fill opacity=1 ] (190,120) -- (220,110) -- (220,130) -- cycle ;
\draw    (180,145) -- (190,145) ;
\draw  [fill={rgb, 255:red, 237; green, 237; blue, 237 }  ,fill opacity=1 ] (190,145) -- (220,135) -- (220,155) -- cycle ;
\draw    (290,90) -- (310,90) ;
\draw    (300,100) -- (300,80) ;
\draw  [fill={rgb, 255:red, 237; green, 237; blue, 237 }  ,fill opacity=1 ] (310,90) -- (340,80) -- (340,100) -- cycle ;
\draw  [fill={rgb, 255:red, 237; green, 237; blue, 237 }  ,fill opacity=1 ] (290,90) -- (260,100) -- (260,80) -- cycle ;
\draw    (290,80) -- (310,100) ;
\draw    (290,100) -- (310,80) ;
\draw  [fill={rgb, 255:red, 237; green, 237; blue, 237 }  ,fill opacity=1 ] (300,100) -- (310,130) -- (290,130) -- cycle ;
\draw  [fill={rgb, 255:red, 237; green, 237; blue, 237 }  ,fill opacity=1 ] (300,80) -- (290,50) -- (310,50) -- cycle ;
\draw  [fill={rgb, 255:red, 237; green, 237; blue, 237 }  ,fill opacity=1 ] (290,100) -- (275.86,128.28) -- (261.72,114.14) -- cycle ;
\draw  [fill={rgb, 255:red, 237; green, 237; blue, 237 }  ,fill opacity=1 ] (290,80) -- (261.72,65.86) -- (275.86,51.72) -- cycle ;
\draw  [fill={rgb, 255:red, 237; green, 237; blue, 237 }  ,fill opacity=1 ] (310,80) -- (324.14,51.72) -- (338.28,65.86) -- cycle ;
\draw  [fill={rgb, 255:red, 237; green, 237; blue, 237 }  ,fill opacity=1 ] (310,100) -- (338.28,114.14) -- (324.14,128.28) -- cycle ;
\draw    (435,88) -- (435,120) ;
\draw   (435,88) -- (427.5,63) -- (442.5,63) -- cycle ;
\draw    (445,90) -- (435,120) ;
\draw   (445,90) -- (446.5,63.94) -- (460.6,69.07) -- cycle ;
\draw    (453,94) -- (435,120) ;
\draw   (453,94) -- (463.32,70.03) -- (474.82,79.67) -- cycle ;
\draw   (425,90) -- (409.4,69.07) -- (423.5,63.94) -- cycle ;
\draw    (425,90) -- (435,120) ;
\draw    (416,94) -- (435,120) ;
\draw   (416,94) -- (394.18,79.67) -- (405.68,70.03) -- cycle ;
\draw   (410,102) -- (384.6,96) -- (392.1,83) -- cycle ;
\draw   (407,112) -- (381.08,115.04) -- (383.68,100.27) -- cycle ;
\draw   (463,112) -- (486.32,100.27) -- (488.92,115.04) -- cycle ;
\draw    (410,102) -- (435,120) ;
\draw    (460,102) -- (435,120) ;
\draw    (407,112) -- (435,120) ;
\draw    (463,112) -- (435,120) ;

\end{tikzpicture}
    \caption{Illustration of the four bad cases: (from left to right) sausages/twigs, skewers, flowers and bushes. {Gray trees are larger than micro-sized.}}
    \label{fig:barbecue}
\end{figure}

\begin{enumerate}
    \item[\textbf{Sausage/Twig.}] A largest common subtree does not contain a path where the sum of the sizes of the small-sized components of the forest obtained by removing the edges of that path is either too big (like a sausage) or too small (like a twig) compared to its length times $c_{\mu,\mu'}=\E\big[\LCS^\bullet(\tau_*,\tau_*')\big]$. To be precise, we shall bound the probability of the bad event $B_n^{(1)}$ that there exist $u,v\in \tau$ and $u',v'\in\tau'$ with \[\llbracket u,v\rrbracket=(u=u_0,u_1,\dots, u_\ell=v)\quad\text{ and }\quad\llbracket u',v'\rrbracket=(u'=u'_0,u'_1,\dots, u'_\ell=v')\]
    such that if $\tau_{*,i}=\mathrm{Compo}(\tau-\llbracket u,v \rrbracket, u_i)$ (resp.~$\tau_{*,i}'=\mathrm{Compo}(\tau'-\llbracket u',v' \rrbracket, u'_i)$)  stands for the tree stemming out the $i$-th vertex of the path from $u$ to $v$ in $\tau$ (resp.~from $u'$ to $v'$ in $\tau'$) for all $0\leq i\leq \ell$, then
    \[\bigg|c_{\mu,\mu'}\cdot\ell-\sum_{i=0}^{\ell} \LCS^\bullet\big(\tau_{*,i},\tau_{*,i}')\wedge n^{1/2-\alpha}\bigg|> \tilde{\delta} \sqrt{n},\]
    for a suitable $\tilde{\delta}>0$.
    \item[\textbf{Skewer.}]  Common subtrees do not have a path containing $d$ vertices with a larger than micro-sized mass attached (like a skewer). Formally, we shall bound the probability of the event $B_n^{(2)}$ that $\tau$ and $\tau'$ have a common subtree $\ttT$ equipped with two vertices $\ttu,\ttv\in\ttT$ such that at least $d$ distinct vertices on the path between $\ttu$ and $\ttv$ are each adjacent to a component with size bigger than $n^{1/2-2\alpha}$ of the forest obtained by removing all the vertices of the path between $\ttu$ and $\ttv$ from $\ttT$. 
    \item[\textbf{Flower.}] Common subtrees do not contain a vertex at which are attached at least $d$ subtrees that are larger than micro-sized (like a flower). Formally, we shall bound the probability of the event $B_n^{(3)}$ that $\tau$ and $\tau'$ have a common subtree $\ttT$ equipped with a vertex $\ttu\in\ttT$ such that the forest obtained by removing $\ttu$ from $\ttT$ has at least $d$ distinct components of size bigger than $n^{1/2-2\alpha}$.
    \item[\textbf{Bush.}]  Common subtrees do not contain a vertex for which the sum of the sizes of the micro-sized subtrees attached to it is larger than small-sized (like a bush). More precisely, we shall bound the probability of the event $B_n^{(4)}$ that $\tau$ and $\tau'$ have a common subtree $\ttT$ equipped with a vertex $\ttu\in\ttT$ such that if $L_1,\ldots,L_k$ stands for the sizes of the components of the forest obtained by removing $\ttu$ from $\ttT$, then
    \[\sum_{i=1}^k L_i\wedge n^{1/2-2\alpha}> n^{1/2-\alpha}-1.\]
\end{enumerate}
Furthermore, let us define the typical event \[A_n=\big\{\#\tau\vee\#\tau'\leq n\, ;\, \Ht(\tau)\leq n^{1/2+\varepsilon}  ;\, \Delta(\tau)\leq n^{1/2-2\varepsilon}\, ;\, \Delta(\tau')\leq n^{1/2+\varepsilon}\big\}.\]

The following proposition, that we prove in the remainder of this section, is the main ingredient of the proof of \cref{prop:size-to-length_estimate}. 

\begin{prop}
\label{prop:no_bad_stuff}
If $\varepsilon\in (0,1/4)$ and $\alpha\in (0,\varepsilon/4)$ then for all $m\geq 1$, there exists an integer $d\geq 3$ such that for all $\tilde{\delta}>0$, it holds that $\P(A_n\cap B_n^{(1)})+\P(A_n\cap B_n^{(2)})+\P(A_n\cap B_n^{(3)})+\P(A_n\cap B_n^{(4)})\leq 4n^{-m}$ for all $n$ large enough.
\end{prop}

We first show how \cref{prop:size-to-length_estimate} follows from \cref{prop:no_bad_stuff}. 

We need one more piece of notation for this proof. Let $\ttt$ be a tree. We construct another tree by collapsing each maximal chain of degree-$2$ vertices of $\ttt$ into a single edge. More precisely, let $\mathrm{Collap}(\ttt)$ be the graph whose vertex-set is
the set $V$ of vertices of $\ttt$ with degree different from $2$, and whose edge-set is the set of pairs $\{\ttu,\ttv\}$ of distinct elements of $V$ such that $\llbracket\ttu,\ttv\rrbracket\cap V=\{\ttu,\ttv\}$. We stress that, while the vertex-set of $\mathrm{Collap}(\ttt)$ is a subset of $\ttt$, $\mathrm{Collap}(\ttt)$ is not a subtree of $\ttt$. However, note that every vertex of $\mathrm{Collap}(\ttt)$ has the same degree in $\mathrm{Collap}(\ttt)$ as in $\ttt$. Finally, observe that $\mathrm{Collap}(\ttt)$ and the paths $\rrbracket\ttu,\ttv\llbracket$, for pairs $\{\ttu,\ttv\}$ that are edges of $\mathrm{Collap}(\ttt)$, partition $\ttt$. 

\begin{proof}[Proof of \cref{prop:size-to-length_estimate} using \cref{prop:no_bad_stuff}]
Let $\varepsilon\in (0,1/4)$, $\delta>0$, and $m\geq 1$. We choose an arbitrary $\alpha\in(0,\varepsilon/4)$ and denote by $d\geq 3$ an integer as in \cref{prop:no_bad_stuff}. Then, set $N=d(d-1)^{d-1}$ and $\tilde{\delta}=\delta/N$. Let $n$ be large enough so that $c_{\mu,\mu'}\leq \tilde{\delta} \sqrt{n}$ and $2n^{-\alpha}\leq \tilde{\delta}$, and $\P(A_n\cap B_n^{(1)})+\P(A_n\cap B_n^{(2)})+\P(A_n\cap B_n^{(3)})+\P(A_n\cap B_n^{(4)})\leq 4n^{-m}$; such integers exist by \cref{prop:no_bad_stuff}. In the following, we show a deterministic result: if $B_n^{(1)}$, $B_n^{(2)}$, $B_n^{(3)}$, and $B_n^{(4)}$ do not occur, then
\begin{equation}
\label{goal_size_to_length}
|\LCS(\tau,\tau')-c_{\mu,\mu'}\LCS_N(\tau,\tau')|\leq 6N\tilde{\delta}\sqrt{n}=6\delta\sqrt{n}.
\end{equation}
From this, it will follow that
\[\P\big(A_n\, ;\, |\LCS(\tau,\tau')-c_{\mu,\mu'}\LCS_N(\tau,\tau')|>6\delta\sqrt{n}\big)\leq 4n^{-m},\]
which implies the desired inequality.

\begin{figure}
    \includegraphics[page=2,scale=0.4]{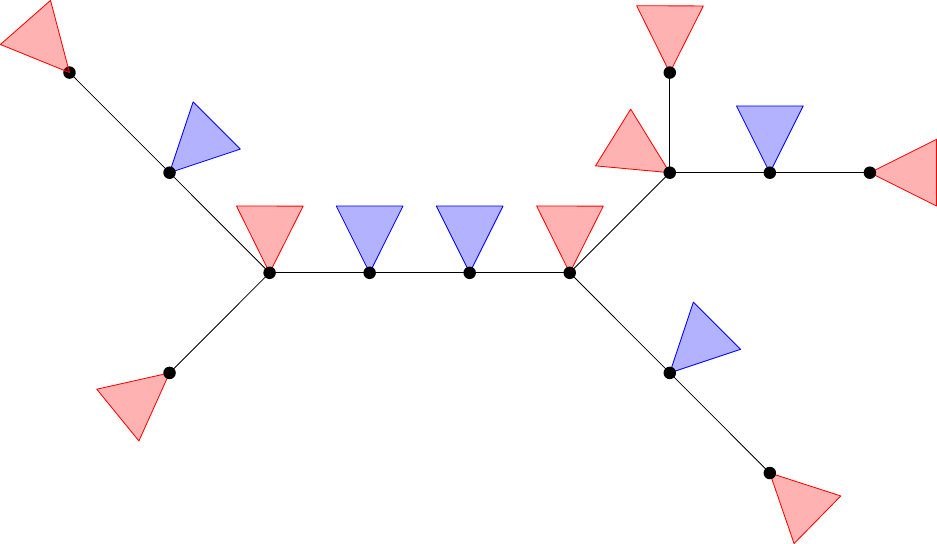}\hspace{2em}
    \includegraphics[page=3,scale=0.4]{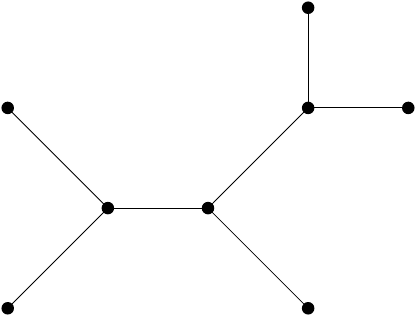}
    \caption{On the left, a depiction of the common subtree $\ttT$. The black vertices are the vertices in $\ttR$, and $\mathrm{Collap}(\ttR)$ is depicted on the right. The red triangles represent the trees in $\ttT$ off vertices in $\mathrm{Collap}(\ttR)$ and the blue triangles represent the trees in $\ttT$ off the vertices in $\ttR\setminus\mathrm{Collap}(\ttR)$.}
    \label{fig:LCS<->skeleton}
\end{figure}

\medskip
\noindent \textbf{Lower bound.}
Let $\ttR$ be a common subtree of $\tau$ and $\tau'$ with $M$ vertices of degree $1$, where $M\leq N$, and let $\phi:\ttR\to \tau$ and $\phi':\ttR\to\tau'$ be the corresponding embeddings. Then, $\mathrm{Collap}(\ttR)$  is a tree with $M$ vertices with degree $1$ and no vertices of degree $2$, so $\#\mathrm{Collap}(\ttR)\le 2M-2$, where the upper bound is attained if and only if $\mathrm{Collap}(\ttR)$ only has vertices of degree $1$ and $3$. Thus, $\#\mathrm{Collap}(\ttR)\le 2N$.

For any $\ttw\in\ttR\setminus\mathrm{Collap}(\ttR)$, there exists a unique edge $\{\ttu,\ttv\}$ of $\mathrm{Collap}(\ttR)$ such that $\ttw\in\rrbracket \ttu,\ttv\llbracket$. We choose a common rooted subtree $\ttT_\ttw$ of $\mathrm{Compo}(\tau-\llbracket \phi(\ttu),\phi(\ttv)\rrbracket;\phi(\ttw))$ and $\mathrm{Compo}(\tau'-\llbracket \phi'(\ttu),\phi'(\ttv)\rrbracket;\phi'(\ttw))$ with maximum size, namely
\[\#\ttT_\ttw=\LCS^\bullet\Big(\mathrm{Compo}(\tau-\llbracket \phi(\ttu),\phi(\ttv)\rrbracket;\phi(\ttw)),\mathrm{Compo}(\tau'-\llbracket \phi'(\ttu),\phi'(\ttv)\rrbracket;\phi'(\ttw))\Big).\]
Since the event $B_n^{(1)}$ does not happen, it holds that
\begin{equation}
\label{step_lower-bound_size_to_length}
2n^{1/2-\alpha}+\sum_{\ttw\in\rrbracket \ttu,\ttv\llbracket}\#\ttT_\ttw\geq c_{\mu,\mu'}\#\rrbracket \ttu,\ttv\llbracket -\tilde{\delta}\sqrt{n},
\end{equation}
where we need the extra term of $2n^{1/2-\alpha}$ to take into account the full trees stemming from  the end points of  $\llbracket \ttu,\ttv\rrbracket$.
\smallskip

Now, construct a new tree $\ttT$ by identifying $\ttw$ and the root of $\ttT_\ttw$ for each $\ttw\in\ttR\setminus\mathrm{Collap}(\ttR)$; see \cref{fig:LCS<->skeleton} for an example, assuming that the red triangles only contain their black root and that the blue triangles represent the $\ttT_\ttw$ for $\ttw\in\ttR\setminus\mathrm{Collap}(\ttR)$. Then, it is easy to see that $\ttT$ is a common subtree of $\tau$ and $\tau'$. 
Together with \eqref{step_lower-bound_size_to_length}, this yields that
\begin{align*}
\LCS(\tau,\tau')&\geq \sum_{\substack{\{\ttu,\ttv\}\\ \text{edge of }\mathrm{Collap}(\ttR)}}\sum_{\ttw\in\rrbracket\ttu,\ttv\llbracket}\#\ttT_\ttw\\
&\geq \sum_{\substack{\{\ttu,\ttv\}\\ \text{edge of }\mathrm{Collap}(\ttR)}}\big(c_{\mu,\mu'}\#\rrbracket \ttu,\ttv\llbracket -\tilde{\delta}\sqrt{n}-2n^{1/2-\alpha}\big)\\
&=c_{\mu,\mu'}\#\big(\ttR\setminus \mathrm{Collap}(\ttR)\big) -\big(\tilde{\delta}\sqrt{n}+2n^{1/2-\alpha}\big)(\#\mathrm{Collap}(\ttR)-1)\\
&\geq c_{\mu,\mu'}(\#\ttR -1) -\big(c_{\mu,\mu'}+\tilde{\delta}\sqrt{n}+2n^{1/2-\alpha}\big)\#\mathrm{Collap}(\ttR).
\end{align*}
Since $\#\mathrm{Collap}(\ttR)\leq 2N$ and because $n$ is large enough so that $c_{\mu,\mu'}\leq \tilde{\delta}\sqrt{n}$ and $2n^{-\alpha}\leq \tilde{\delta}$, we get that $\LCS(\tau,\tau')\geq c_{\mu,\mu'}(\#\ttR-1)-6N\tilde{\delta}\sqrt{n}$. Taking the supremum over $\#\ttR$ provides the desired lower bound.

\medskip
\noindent \textbf{Upper bound.}
Let us denote by $\ttT$ a (unrooted) common subtree with maximum size, and by $\phi:\ttT\rightarrow \tau$ and $\phi':\ttT\rightarrow \tau'$ the corresponding embeddings. Define $\ttR_0$ as the set of endpoints of those edges of $\ttT$ whose removal disconnects $\ttT$ into two components with sizes bigger than $n^{1/2-2\alpha}$. If $\ttR_0$ is non-empty, then it is a subtree of $\ttT$ and we set $\ttR=\ttR_0$; if $\ttR_0$ is empty, then we set $\ttR=\{\ttr\}$ where $\ttr$ is a centroid of $\ttT$ (i.e.,~a vertex whose removal splits the tree into components containing at most half of the vertices of $\ttT$).
Note that as a subtree of $\ttT$, $\ttR$ is a common subtree of $\tau$ and $\tau'$. See \cref{fig:LCS<->skeleton} for an illustration of the situation. We now show that $\ttR$ has at most $N$ vertices of degree $1$ and that $\#\ttT\leq c_{\mu,\mu'}(\#\ttR-1) +6N\tilde{\delta}\sqrt{n}$. Indeed, these two properties imply that $\#\ttT\leq c_{\mu,\mu'}\LCS_N(\tau,\tau')+6N\tilde{\delta}\sqrt{n}$, as desired.

\medskip
Let $\ttu\in\ttR$. By definition of $\ttR$, the number of incident edges to $\ttu$ in $\ttR$ is smaller than or equal to the number of components with sizes bigger than $n^{1/2-2\alpha}$ in the forest obtained by removing $\ttu$ from $\ttT$. As we assume that $B_n^{(3)}$ is not realized, it follows that the degree of $\ttu$ in $\ttR$ is smaller than $d$. 
Moreover, if $\ttu,\ttv,\ttw\in\ttR$ are distinct and if $\ttw$ has degree at least $3$ in $\ttR$, then the forest obtained by removing $\ttw$ from $\ttT$ has at least $3$ components with size bigger than $n^{1/2-2\alpha}$, so one of them avoids $\{\ttu,\ttv\}$. Thus, if $\ttw$ is also on the path $\llbracket \ttu,\ttv\rrbracket$ between $\ttu$ and $\ttv$, then it is adjacent to a component with size bigger than $n^{1/2-2\alpha}$ of the forest obtained by removing all vertices of $\llbracket \ttu,\ttv\rrbracket$ from $\ttT$. Since $B_n^{(2)}$ does not occur, it follows that for any $\ttu,\ttv\in\ttR$, there are less than $d$ vertices in $\rrbracket\ttu,\ttv\llbracket$ that have degree at least $3$ in $\ttR$. 
Therefore, both the maximum degree and the diameter of $\mathrm{Collap}(\ttR)$ are at most $d$. This allows us to recursively construct an embedding of $\mathrm{Collap}(\ttR)$ into the truncation at height $d$ of the infinite rooted $d$-regular tree. Since this truncated tree has exactly $d(d-1)^{d-1}$ leaves, we get that $\ttR$ has at most $d(d-1)^{d-1}=N$ vertices of degree $1$, as desired. Furthermore, since $d\geq 3$, we also obtain
\begin{equation}
\label{size_combinatorial_skeleton}
\#\mathrm{Collap}(\ttR)\leq 1+\sum_{h=1}^d d(d-1)^{h-1}\leq 2N.
\end{equation}

Now, we want to prove that $\#\ttT\leq c_{\mu,\mu'}(\#\ttR-1)+6N\tilde{\delta}\sqrt{n}$. Let $\ttu\in\ttR$, let $\ttu_1,\ldots,\ttu_k$ be the neighbours of $\ttu$ in $\ttT$, and for each $1\leq i\leq k$, denote by $L_i$ the size of the component containing $\ttu_i$ of the forest obtained by removing $\ttu$ from $\ttT$.
With this notation, we can write
\[\#\mathrm{Compo}(\ttT-\ttR;\ttu)=1+\sum_{i=1}^k \I{\ttu_i\notin\ttR} L_i.\]
By definition of $\ttR$, it holds that $\ttu_i\in\ttR$ if and only if $L_i>n^{1/2-2\alpha}$ and $1+\sum_{j\neq i}L_j>n^{1/2-2\alpha}$. When $\ttR=\ttR_0$, there exists $1\leq j\leq k$ such that $\ttu_j\in\ttR$, so that $L_j>n^{1/2-2\alpha}$, which implies that for any $1\leq i\leq k$ with $i\neq j$, we have that $\ttu_i\notin\ttR$ if and only if $L_i\leq n^{1/2-2\alpha}$. 
On the other hand, when $\ttR=\{\ttr\}$, $\ttu=\ttr$ is a centroid of $\ttT$, so for any $1\leq i\leq k$, we have $L_i\leq 1+\sum_{j\neq i}L_j$, which yields that $L_i\leq n^{1/2-2\alpha}$ if $\ttu_i\notin\ttR$. 
Thus, in every case, we have
\[\#\mathrm{Compo}(\ttT-\ttR;\ttu)=1+\sum_{i=1}^k \I{\ttu_i\notin\ttR} L_i\leq 1+\sum_{i=1}^k L_i\wedge n^{1/2-2\alpha}.\]
Then, the fact that $B_n^{(4)}$ is not realized implies that \[\#\mathrm{Compo}(\ttT-\ttR;\ttu)\leq n^{1/2-\alpha}.\]

Now, let $\ttu,\ttv\in\mathrm{Collap}(\ttR)$ be neighbours in $\mathrm{Collap}(\ttR)$. Let $\ttw$ be a vertex on the path $\rrbracket \ttu,\ttv\llbracket$ strictly between $\ttu$ and $\ttv$ in $\ttT$. In particular, $\ttw\in \ttR\setminus \mathrm{Collap}(\ttR)$ and thus $\ttw$ has degree $2$ in $\ttR$. Then, observe that $\phi$ and $\phi'$ respectively induce embeddings of $\mathrm{Compo}(\ttT-\ttR;\ttw)=\mathrm{Compo}(\ttT-\llbracket\ttu,\ttv\rrbracket;\ttw)$ into $\mathrm{Compo}(\tau-\llbracket \phi(\ttu),\phi(\ttv)\rrbracket;\phi(\ttw))$ and $\mathrm{Compo}(\tau'-\llbracket \phi'(\ttu),\phi'(\ttv)\rrbracket;\phi'(\ttw))$ that map the root $\ttw$ to the roots $\phi(\ttw)$ and $\phi'(\ttw)$. Therefore,
\begin{multline*}
\#\mathrm{Compo}(\ttT-\ttR;\ttw)\\
\leq n^{1/2-\alpha}\wedge \LCS^\bullet\Big(\mathrm{Compo}\big(\tau-\llbracket \phi(\ttu),\phi(\ttv)\rrbracket;\phi(\ttw)\big)\, ,\, \mathrm{Compo}\big(\tau'-\llbracket \phi'(\ttu),\phi'(\ttv)\rrbracket;\phi'(\ttw)\big)\Big).
\end{multline*}
Remark that the distances between $\phi(\ttu),\phi(\ttv)$ and $\phi'(\ttu),\phi'(\ttv)$ are the same, namely, equal to the distance $\ell$ between $\ttu,\ttv$ in $\ttT$.
Moreover, for $1\leq i\leq \ell-1$, if $\ttw$ is the $i$-th vertex on the path from $\ttu$ to $\ttv$, then $\phi(\ttw)$ (resp.~$\phi'(\ttw)$) is the $i$-th vertex on the path from $\phi(\ttu)$ to $\phi(\ttv)$ on $\tau$ (resp.~from $\phi'(\ttu)$ to $\phi'(\ttv)$ on $\tau'$). Hence, the fact that $B_n^{(1)}$ is not realized yields that 
\[\sum_{\ttw\in \rrbracket \ttu,\ttv\llbracket}\#\mathrm{Compo}(\ttT-\ttR;\ttw)\leq c_{\mu,\mu'}\ell+\tilde{\delta}\sqrt{n}=c_{\mu,\mu'}\#\rrbracket \ttu,\ttv\llbracket+c_{\mu,\mu'}+\tilde{\delta}\sqrt{n}.\]
Finally, we conclude by writing
\begin{align*}
\#\ttT&=\sum_{\ttu\in\mathrm{Collap}(\ttR)}\#\mathrm{Compo}(\ttT-\ttR;\ttu)+\sum_{\substack{\{\ttu,\ttv\}\\ \text{edge of }\mathrm{Collap}(\ttR)}}\sum_{\ttw\in\rrbracket \ttu,\ttv\llbracket}\#\mathrm{Compo}(\ttT-\ttR;\ttw)\\
&\leq \sum_{\ttu\in\mathrm{Collap}(\ttR)}n^{1/2-\alpha}+\sum_{\substack{\{\ttu,\ttv\}\\ \text{edge of }\mathrm{Collap}(\ttR)}}\big(c_{\mu,\mu'}\#\rrbracket \ttu,\ttv\llbracket+c_{\mu,\mu'}+\tilde{\delta}\sqrt{n}\big)\\
&=c_{\mu,\mu'}\#\big(\ttR\setminus \mathrm{Collap}(\ttR)\big) +c_{\mu,\mu'}(\#\mathrm{Collap}(\ttR)-1) -\tilde{\delta}\sqrt{n}+\big(\tilde{\delta}\sqrt{n}+n^{1/2-\alpha}\big)\#\mathrm{Collap}(\ttR)\\
&\leq c_{\mu,\mu'}(\#\ttR-1) +\big(\tilde{\delta}\sqrt{n}+n^{1/2-\alpha}\big)\#\mathrm{Collap}(\ttR).
\end{align*}
Using \eqref{size_combinatorial_skeleton} and that $n$ is large enough so that $c_{\mu,\mu'}\leq \tilde{\delta}\sqrt{n}$ and $n^{-\alpha}\leq \tilde{\delta}$, we find that $\#\ttT\leq c_{\mu,\mu'}(\#\ttR-1)+4N\tilde{\delta}\sqrt{n}\leq c_{\mu,\mu'}(\#\ttR-1)+6N\tilde{\delta}\sqrt{n}$. This concludes the proof.
\end{proof}

In the next four subsections, we bound the probability of each of the four bad events, from which \cref{prop:no_bad_stuff} then follows.

\subsection{Bad event (1): sausage/twig}
Recall that $B_n^{(1)}$ is the bad event that there exist $u,v\in \tau$ and $u',v'\in\tau'$ with the respective paths between them denoted as
\[\llbracket u,v\rrbracket=(u=u_0,u_1,\dots, u_\ell=v)\quad\text{ and }\quad\llbracket u',v'\rrbracket=(u'=u'_0,u'_1,\dots, u'_\ell=v')\]
such that if for each $i=0,\dots,\ell$ we denote by
\[L_i=\LCS^\bullet\big(\mathrm{Compo}(\tau-\llbracket u,v \rrbracket, u_i),\mathrm{Compo}(\tau'-\llbracket u',v' \rrbracket, u'_i)\big)\]
the maximum size of a common rooted subtree of the trees that stem off these paths from the respective vertices $u_i$ and $u_i'$, then it holds that 
\[\bigg|c\ell-\sum_{i=0}^{\ell} L_i\wedge n^{1/2-\alpha}\bigg|> \tilde{\delta} \sqrt{n},\]
where we write $c=c_{\mu,\mu'}$ in all this subsection to ease notation. Moreover, define the event $A_n^{(1)}:=\{\Ht(\tau) \vee \Ht(\tau') \leq n-1 \,;\, \Ht(\tau) \leq n^{1/2+\e}\}$, which contains the typical event $A_n$. The goal of this subsection is to establish the following bound on the probability of $B_n^{(1)}\cap A_n$. 

\begin{prop}
\label{prop:no_sausages}
    If $\alpha,\varepsilon>0$ satisfy $\alpha+\varepsilon<1/4$, then for all $\tilde\delta>0$, there exists a constant $C_{\tilde\delta}>0$ such that $\prob(A_n^{(1)}\cap B^{(1)}_n)\le C_{\tilde\delta}\exp(-\tfrac{1}{14}\delta n^{\alpha})$ for all integers $n\geq 1$.
\end{prop} 
    
We first show the following deterministic lemma. 

\begin{lem}
\label{deter_lemma:no_sausages}
If $B_n^{(1)}$ occurs and if $n$ is large enough so that $14\max(c,n^{1/2-\alpha})<\tilde{\delta}\sqrt{n}$, then there exist $w \in \tau$,  $w'\in \tau'$, and $0\leq k\leq \Ht(\tau)-1$ such that at least one of the following two events happens:
\begin{enumerate}
\item[\emph{(i)}] $\displaystyle{\Bigg|ck- \!\sum_{\substack{\varnothing \prec x\preceq w \\ \varnothing \prec x'\preceq w'}}\! \I{|w|-|x|=|w'|-|x'|<k} \LCS^\bullet\left(\Trim_{x}\tau,\Trim_{x'}\tau'\right)\wedge n^{1/2-\alpha}\Bigg|>\tfrac{\tilde\delta}{14} \sqrt{n} }$;
\item[\emph{(ii)}] $\displaystyle{\Bigg|ck- \!\sum_{\substack{\varnothing \prec x\preceq w \\ \varnothing \prec x'\preceq w'}}\! \I{|w|-|x|=k-1-(|w'|-|x'|)<k} \LCS^\bullet\left(\Trim_{x}\tau,\Trim_{x'}\tau'\right)\wedge n^{1/2-\alpha}\Bigg|>\tfrac{\tilde\delta}{14} \sqrt{n} }$.
\end{enumerate}
\end{lem}

\begin{proof}
\begin{figure}
\input{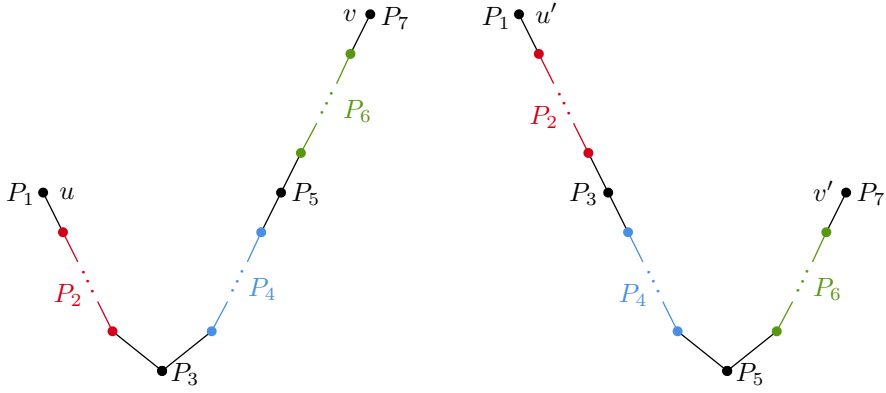}
\caption{We partition the path from $u$ to $v$ in $\tau$ and the path from $u'$ to $v'$ in $\tau'$, of the same length, into $7$ parts.}
\label{fig:partition}
\end{figure}
    Suppose that $B_n^{(1)}$ occurs, so that we can find $u,v\in \tau$ and $u',v'\in \tau'$ that satisfy the assumptions. We first observe that for any partition $P_1 \sqcup P_2 \sqcup \dots  \sqcup P_7=\{0,\dots, \ell\}$, the triangle inequality implies that if $B_n^{(1)}$ occurs, then there is $1\le j \le 7$ such that 
    \[\left| c|P_j|-\sum_{i\in P_j} L_i\wedge n^{1/2-\alpha}\right|> \tfrac{1}{7}(\tilde\delta \sqrt{n}-c)\geq \tfrac{1}{14}\tilde\delta \sqrt{n}. \]
    Call this event $B_n^{(1)}(P_j)$.  Observe that by our assumption that $c\vee n^{1/2-\alpha}<\tfrac{1}{14}\tilde\delta \sqrt{n}$ we have that $B_n^{(1)}(P_j)=\emptyset$ if $|P_j|=1$.  

    We now define the partition $P_1 \sqcup P_2 \sqcup \dots  \sqcup P_7$ that we will use to prove the lemma; also see \cref{fig:partition}. Choose $P_1=\{0\}$ and $P_7=\{\ell\}\setminus\{0\}$. Next, let $m$ be the unique index such that $u_m$ is an ancestor of $u$ and $v$. 
    Note that if $0\leq i\leq m-1$ then $u_i\preceq u$ and $\overleftarrow{u_i}=u_{i+1}$, but if $m\leq i\leq \ell-1$ then $u_{i+1}\preceq v$ and $\overleftarrow{u_{i+1}}=u_i$. Respectively, let $m'$ be the unique index such that $u_{m'}'$ is an ancestor of $u'$ and $v'$. 
    Assume without loss of generality that $m\le m'$. Set $P_3=\{m\}\setminus\{0,\ell\}$ and $P_5=\{m'\}\setminus\{0,\ell, m\}$.  Thus, $B_n^{(1)}(P_j)=\emptyset$ for $j=1,3,5,7$. Also set $P_2=\{1,\dots, m-1\}$, $P_4=\{m+1,\dots, m'-1\}$ and $P_6=\{m'+1, \dots, \ell-1\}$. We need to show that if $B_n^{(1)}(P_2)$, $B_n^{(1)}(P_4)$ or $B_n^{(1)}(P_6)$ happens, then one of the events from the lemma happens for some $w$, $w'$ and $k$. 
    
    First, observe that $B_n^{(1)}(P_2)$ implies that the first event in the lemma happens for $w=u$, $w'=u'$ and $k=m-1$. Similarly, $B_n^{(1)}(P_6)$ implies that the first event in the lemma happens for $w=v$, $w'=v'$ and $k=\ell-m'-1$. Finally, $B_n^{(1)}(P_4)$ implies that the second event in the lemma happens for $w=u_{m'}$, $w'=u_{m}'$, and $k=m'-m-1$.
    
    The only thing left to do is to check that $m$, $m'-m$, and $\ell-m'$ are bounded by $\Ht(\tau)$. Since $u_m\preceq u=u_0$, we have $m=|u|-|u_m|\leq |u|\leq \Ht(\tau)$. Likewise, $u_m\preceq v=u_\ell$ so $(\ell-m')+(m'-m)=\ell-m=|v|-|u_m|\leq |v|\leq \Ht(\tau)$.
\end{proof}

\begin{proof}[Proof of \cref{prop:no_sausages}]
    For any $k\geq 0$, for any two plane trees $t, t'$, and for any $w\in t$ and $w'\in t'$, we define $F_{k}^{\fw}(t,w;t',w')\in \{0,1\}$ and $F_k^{\bw}(t,w;t',w')\in\{0,1\}$ as the indicator functions of the respective inequalities
    \begin{align*}
    \Bigg|ck-\sum_{\substack{\varnothing \prec x\preceq w \\ \varnothing \prec x'\preceq w'}} \I{|w|-|x|=|w'|-|x'|<k} \LCS^\bullet\left(\Trim_{x} t,\Trim_{x'} t'\right)\wedge n^{1/2-\alpha}\Bigg|&>\tfrac{1}{14}\tilde\delta \sqrt{n},\\
    \Bigg|ck-\sum_{\substack{\varnothing \prec x\preceq w \\ \varnothing \prec x'\preceq w'}} \I{|w|-|x|=k-1-(|w'|-|x'|)<k} \LCS^\bullet\left(\Trim_{x} t,\Trim_{x'} t'\right)\wedge n^{1/2-\alpha}\Bigg|&>\tfrac{1}{14}\tilde\delta \sqrt{n}.
    \end{align*}
    
    By \cref{deter_lemma:no_sausages} and the Many-to-One principle (\cref{many-to-one}), 
    \begin{equation}\label{bound_B4}
    \prob(A_n^{(1)}\cap B^{(1)}_n)\le \sum_{h=0}^{n-1} \sum_{h'=0}^{n-1}\sum_{k=0}^{n^{1/2+\varepsilon}-1} \E\big[F_k^\fw(\tau_\infty,U_h; \tau'_\infty, U_{h'}')+F_k^\bw(\tau_\infty,U_h; \tau'_\infty, U_{h'}')\big],
    \end{equation}
    where $\big(\tau_\infty,(J_i)_{i\geq 1}\big)$ and $\big(\tau_\infty',(J_i')_{i\geq 1}\big)$ are two independent size-biased Bienaym\'{e} trees with offspring distribution $\mu$ and $\mu'$ respectively, and where $U_i = (J_1,\ldots,J_i)$ and $U_i' = (J_1',\ldots,J_i')$ for all $i\geq 0$. Now, denote by $\tau_{*}$ and $\tau_{*}'$ two independent root-biased Bienaym\'{e} trees with offspring distribution $\mu$ and $\mu'$ respectively, and let $(X_j)_{j\geq 0}$ be a sequence of independent copies of $\LCS^\bullet(\tau_*, \tau_*')$. With this notation, \cref{root-biased_VS_size-biased} yields that
    \[
        \E\big[F_k^\fw(\tau_\infty,U_h; \tau'_\infty, U_{h'}')+F_k^\bw(\tau_\infty,U_h; \tau'_\infty, U_{h'}')\big]\le 2 \prob\left(\left|ck-\sum_{j=0}^{k-1} X_j \wedge n^{1/2-\alpha}\right|>\tfrac{1}{14}\tilde\delta \sqrt{n} \right). 
    \]
    
    We start by bounding 
    \[\prob\left(\sum_{j=0}^{k-1} X_j \wedge n^{1/2-\alpha}<ck-\tfrac{1}{14}\tilde\delta \sqrt{n} \right). \]
    We want to centre the sum, so we first study the difference between $c=\E[X_j]$ and $\E[X_j\wedge n^{1/2-\alpha}]$. Using \cref{prop:LCS_sb_tail}, we find that 
    \[
    \E[X_j]-\E[X_j\wedge n^{1/2-\alpha}]\le \sum_{h\ge n^{1/2-\alpha}}\prob(X_j\ge h)\le C\sum_{h\ge n^{1/2-\alpha}} h^{-\frac{2}{1+\xi}}=O(n^{-\frac{1-\xi}{1+\xi}(1/2-\alpha)} ),
    \]
    for any $\xi>0$. Since $\alpha<1/4-\varepsilon$, we may pick $\xi$ small enough so that we find that \[n^{1/2+\varepsilon}\big(\E[X_j]-\E[X_j\wedge n^{1/2-\alpha}]\big)<\tfrac{1}{28}\tilde\delta \sqrt{n}\]
    for all $n$ large enough. It follows that for such $n$, if $k \le n^{1/2+\varepsilon}$ then 
    \[\prob\left(\sum_{j=0}^{k-1} X_j \wedge n^{1/2-\alpha}<ck-\tfrac{\tilde\delta}{14} \sqrt{n} \right)\le \prob\left(\sum_{j=0}^{k-1} \big( \E[X_j \wedge n^{1/2-\alpha} ]-X_j \wedge n^{1/2-\alpha}\big)>\tfrac{\tilde\delta}{28} \sqrt{n} \right) .\]
    We will bound this probability using Hoeffding's inequality for random variables that are bounded from above; see~\cite[Corollary 2.7]{FGL_hoeffding}. Using \cref{prop:LCS_sb_tail}, we find that 
    \[\E\big[(X_j \wedge n^{1/2-\alpha})^2\big]\le \sum_{0\le h\le n^{1-2\alpha}}\prob(X_j\ge \sqrt{h} )\le C n^{\tfrac{(1-2\alpha)\xi}{1+\xi}}\]
    for some constant $C>0$, so that if $k \le n^{1/2+\varepsilon}$ then
    \begin{align*}\prob\left(\sum_{j=0}^{k-1} \big( \E[X_j \wedge n^{1/2-\alpha} ]-X_j \wedge n^{1/2-\alpha}\big)>\tfrac{1}{28}\delta \sqrt{n} \right)&\le 2 \exp\left(-\frac{\tilde\delta^2 n }{28^2 C k n^{\tfrac{(1-2\alpha)\xi}{1+\xi}} }\right)\\&\leq 2\exp\left(-\frac{\tilde{\delta}^2}{28^2 C} n^{1-1/2-\varepsilon-\tfrac{(1-2\alpha)\xi}{1+\xi} }\right),\end{align*}
    which is $O\big(\exp(-n^{1/2-2\eps})\big)$ for $\xi>0$ small enough. 
    
    Thanks to \cref{prop:LCS_sb_tail}, we can now use \cref{obj_0-1and1-2} to bound
    \[\prob\left(\sum_{j=0}^{k-1} X_j \wedge n^{1/2-\alpha}>ck+\tfrac{1}{14}\tilde\delta \sqrt{n} \right)\le \prob\left(\sum_{j=0}^{k-1} (X_j-c) \wedge n^{1/2-\alpha}>\tfrac{1}{14}\tilde\delta \sqrt{n} \right).\]
     Indeed, $k\leq n^{1/2+\varepsilon}$ and $1<\frac{1/2+\varepsilon}{1/2-\alpha}<2$, since $0<\alpha<1/4-\varepsilon$. Thus, this probability is $O\big(\exp(-\tfrac{1}{14}\tilde\delta n^\alpha)\big)$ as $n$ goes to infinity.
     
     Plugging the two bounds into \eqref{bound_B4} yields the statement.
   \end{proof}

\subsection{Bad event (2): skewer}

Recall that $B_n^{(2)}$ is the bad event that $\tau$ and $\tau'$ have a common subtree $\ttT$ equipped with a path $\llbracket \ttu,\ttv\rrbracket=(\ttu_0,\ttu_1,\ldots,\ttu_\ell)$ between two distinct vertices $\ttu=\ttu_0$ and $\ttv=\ttu_\ell$ such that, letting $\ttT_{i,1},\ldots,\ttT_{i,\ttk_i}$ for each $0\leq i\leq \ell$ be the components adjacent to $\ttu_i$ of the forest obtained by removing all the vertices of $\llbracket \ttu,\ttv\rrbracket$ from $\ttT$, the set $\{0\leq i\leq \ell : \max_{1\leq j\leq \ttk_i}\#\ttT_{i,j}>n^{1/2-2\alpha}\}$ has at least $d$ elements. Next, define the event $A_n^{(2)}:=\{\Ht(\tau) \vee \Ht(\tau') \leq n-1 \,;\, \Ht(\tau) \leq n^{1/2+\e}\}$ that contains the typical event $A_n$. Our goal is to show the following result.

\begin{prop}
\label{prop:no_skewers}
    If $\alpha, \e > 0$ satisfy $4\alpha+\varepsilon< 1/2$, then for all $m\geq 1$, there exists $d \in \N$ such that $\prob(A_n^{(2)} \cap B_n^{(2)}) \leq n^{-m}$ for all $n$ large enough. 
\end{prop}

For convenience, we introduce some notation. For any plane trees $t$ and $t'$, set
\[\LCS^*(t,t')=\max_{\substack{1\leq i\leq \rk_\varnothing(t)\\ 1\leq i'\leq \rk_\varnothing(t')}}\LCS^\bullet(\theta_i t,\theta_{i'}t'),\]
that is, the maximum size of a common rooted subtree between arbitrary subtrees of $t$ and $t'$ above children of their roots. Then, for any $k\geq 0$, and for $u\in t$ and $u'\in t'$, define
\begin{align*}
N_{k}^{\fw}(t,u;t',u')&=\sum_{\substack{\varnothing\prec v\preceq u \\ \varnothing \prec v'\preceq u'}}\I{|u|-|v|=|u'|-|v'|<k} \I{\LCS^*(\Trim_v t,\Trim_{v'} t')>n^{1/2-2\alpha}},\\
N_{k}^{\bw}(t,u;t',u')&=\sum_{\substack{\varnothing\prec v\preceq u \\ \varnothing \prec v'\preceq u'}}\I{|u|-|v|=\ell-1-(|u'|-|v'|)<k} \I{\LCS^*(\Trim_v t,\Trim_{v'} t')>n^{1/2-2\alpha}}.
\end{align*}
In words, $N_k^{\fw}(t,u;t',u')$ counts the number of ranks $i\in\{1,\ldots,k\}$ such that one can match a subtree of $t\setminus \llbracket u,\varnothing\rrbracket$ and a subtree of $t'\setminus \llbracket u',\varnothing\rrbracket$ respectively attached to the $(i+1)$th vertex of the paths from $u$ and $u'$ to their ancestors from $k$ generations ago, such that the maximum size of a corresponding common rooted subtree is larger than $n^{1/2-2\alpha}$ (i.e., not micro-sized). Similarly, $N_k^{\bw}(t,u;t',u')$ counts the number of ranks for whom one can find such matchings yielding larger than micro-sized common rooted subtrees when the previous paths are explored in opposite directions in the two trees $t$ and $t'$. The definition of $N_k^{\fw}$ and $N_k^{\bw}$ is motivated by the following deterministic lemma.

\begin{lem}
\label{deter-lemma:no_skewers}
Assume that $d\geq 8$. If $B_n^{(2)}$ occurs, then there exist $w \in \tau$, $w' \in \tau'$, and $1 \leq k \leq \Ht(\tau)$ such that $N_{k-1}^{\fw}(\tau,w;\tau',w')\geq d/7$ or $N_{k-1}^{\bw}(\tau,w;\tau',w')\geq d/7$.
\end{lem}

\begin{proof}
Let $\phi$ and $\phi'$ be the embeddings of $\ttT$ into $\tau$ and $\tau'$ respectively. Set $u=\phi(\ttu)$ and $u'=\phi'(\ttu)$, $v=\phi(\ttv)$ and $v'=\phi'(\ttv)$, and $u_i=\phi(\ttu_i)$ and $u_i'=\phi'(\ttu_i)$ for each $0\leq i\leq \ell$, so that $u=u_0$, $u'=u_0'$, $v=u_\ell$, and $v'=u_\ell'$. Analogously to the study of the previous bad event, denote by $m$ and $m'$ the unique indices such that $u_m$ and $u_{m'}'$ are ancestors of $u$ and $v$ and of $u'$ and $v'$ respectively. Recall that if $0\leq i\leq m-1$ then $u_i\preceq u$ and $\overleftarrow{u_i}=u_{i+1}$, but if $m\leq i\leq \ell-1$ then $u_{i+1}\preceq v$ and $\overleftarrow{u_{i+1}}=u_i$. Without loss of generality, we assume that $m \leq m'$, and we partition $\{1, \dots, \ell-1\}\setminus \{m,m'\} = P_2 \sqcup P_4 \sqcup P_6 \sqcup \{\ell\}$ where $P_2 = \{1, \dots, m-1\}$, $P_4 = \{m+1, \dots, m'-1\}$ and $P_6 = \{m'+1, \dots, \ell-1\}$. See \cref{fig:skewers}. From this partition, since $d\ge 8$, the realization of $B_n^{(2)}$ implies that $P_2$, $P_4$ or $P_6$ has at least $d/7$ elements $i$ such that $\max_{1\leq j\leq \ttk_i}\#\ttT_{i,j}>n^{1/2-2\alpha}$. Call the corresponding events $B_n^{(2)}(P_\iota)$ for $\iota=2,4,6$.

For any $0\leq i\leq \ell$ and $1\leq j\leq \ttk_j$, by definition of embeddings, $\phi(\ttT_{i,j})$ (resp.~$\phi'(\ttT_{i,j})$) is contained in a component adjacent to $u_i$ (resp.~to $u_i'$) of the forest obtained by removing the vertices of $\llbracket u,v\rrbracket$ from $\tau$ (resp.~of $\llbracket u',v'\rrbracket$ from $\tau'$). Denoting by $\ttu_{i,j}$ the neighbour of $\ttu_i$ in $\ttT_{i,j}$, $\phi(\ttu_{i,j})$ and $\phi(\ttu_{i,j})$ are neighbours of $u_i$ and $u_i'$, outside of $\llbracket u,v\rrbracket$ and $\llbracket u',v'\rrbracket$ respectively. In particular, if $1\leq i\leq m-1$ then $\phi(\ttu_{i,j})$ must be a child of $u_i=\overleftarrow{u_{i-1}}$ distinct from $u_{i-1}$, meaning that $\phi$ induces a root-preserving embedding of $\ttT_{i,j}$ into $\theta_{k}\Trim_{u_{i-1}}\tau$ for some $1\leq k\leq \rk_{u_i}(\tau)$. Similarly, if $m+1\leq i\leq \ell-1$ then $\phi(\ttu_{i,j})$ must be a child of $u_i=\overleftarrow{u_{i+1}}$ distinct from $u_{i+1}$, meaning that $\phi$ induces a root-preserving embedding of $\ttT_{i,j}$ into $\theta_{k}\Trim_{u_{i+1}}\tau$ for some $1\leq k\leq \rk_{u_i}(\tau)$. Of course, analogous properties hold in $\tau'$, which allows us to obtain the following facts:
\begin{align*}
    \text{if }i\in P_2,\quad &\text{ then }\quad \#\ttT_{i,j}\leq \LCS^*\big(\Trim_{u_{i-1}}\tau,\Trim_{u_{i-1}'}\tau'\big);\\
    \text{if }i\in P_4,\quad &\text{ then }\quad \#\ttT_{i,j}\leq \LCS^*\big(\Trim_{u_{i+1}}\tau,\Trim_{u_{i-1}'}\tau'\big);\\
    \text{if }i\in P_6,\quad &\text{ then }\quad \#\ttT_{i,j}\leq \LCS^*\big(\Trim_{u_{i+1}}\tau,\Trim_{u_{i+1}'}\tau'\big).    
\end{align*}
See \cref{fig:skewers} for an illustration. It follows that $B_n^{(2)}(P_2)$ implies that $N_{m-1}^{\fw}(\tau,u;\tau',u')\geq d/7$, that $B_n^{(2)}(P_4)$ implies that $N_{m'-m-1}^{\bw}(\tau,u_{m'};\tau',u_{m}')\geq d/7$, and that $B_n^{(2)}(P_6)$ implies that $N_{\ell-m'-1}^{\fw}(\tau,v;\tau',v')\geq d/7$. Recall from \cref{deter_lemma:no_sausages} that $m$, $m'-m$, and $\ell-m'$ are bounded by $\Ht(\tau)$, concluding the proof.
\end{proof}

\begin{figure}[hbtp]
    \centering
    \input{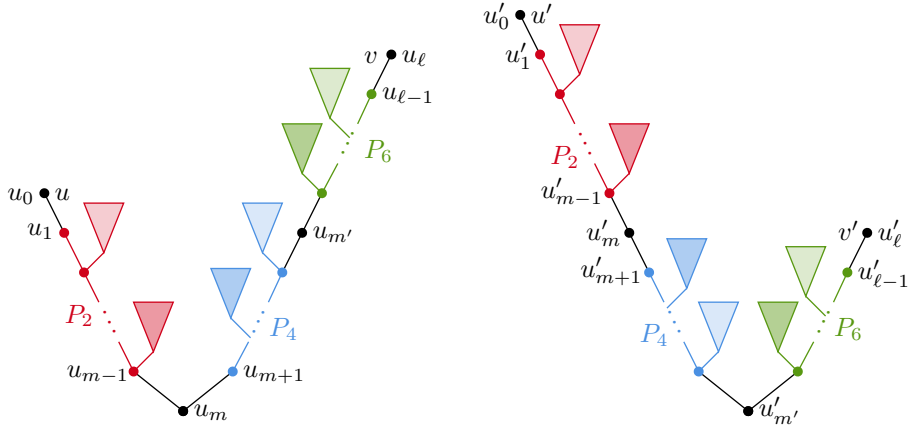}
    \caption{The paths from $u$ to $v$ in $\tau$ and $u'$ to $v'$ in $\tau'$, along with the matching hanging subtrees. }
    \label{fig:skewers}
\end{figure}

\begin{proof}[Proof of \cref{prop:no_skewers}]
Using the Many-to-One principle (\cref{many-to-one}), \cref{deter-lemma:no_skewers} implies that when $d\geq 10$,
\[
\prob(A_n^{(2)}\cap B_n^{(2)}) \leq \sum_{h=0}^{n-1}\sum_{h'=0}^{n-1}\sum_{k=1}^{n^{1/2+\varepsilon}} \P\big(N_{k-1}^{\fw}(\tau_\infty,U_h; \tau_\infty', U_{h'}') \vee N_{k-1}^{\bw}(\tau_\infty,U_h; \tau_\infty', U_{h'}')\geq d/7 \big),
\]
where $\big(\tau_\infty,(J_i)_{i\geq 1}\big)$ and $\big(\tau_\infty',(J_i')_{i\geq 1}\big)$ are two independent size-biased Bienaym\'{e} trees with offspring distribution $\mu$ and $\mu'$ respectively, and where $U_i = (J_1,\ldots,J_i)$ and $U_i' = (J_1',\ldots,J_i')$ for all $i\geq 0$. Now denote by $(\tau_{*,i})_{i\geq 1}$ and $(\tau_{*,i}')_{i\geq 1}$ two independent sequences of i.i.d.~root-biased Bienaymé trees with offspring distributions $\mu$ and $\mu'$ respectively. We then get from \cref{root-biased_VS_size-biased} that
\[\prob(A_n^{(2)}\cap B_n^{(2)}) \leq 2n^2\sum_{k=1}^{n^{1/2+\varepsilon}}\P\bigg(\sum_{i=1}^{k-1} \I{\LCS^*(\tau_{*,i},\tau_{*,i}')>n^{1/2-2\alpha}}\geq d/7\bigg) \]
Leveraging the independence of the $(\tau_{*,i},\tau_{*,i}')$ for $i\geq 1$ and union bounding over $d/7$ choices of elements of $\{1,\ldots,\ell-1\}$, we have
\[\prob(A_n^{(2)}\cap B_n^{(2)}) \leq 2n^2 \cdot n^{1/2+\varepsilon} \Big(n^{1/2+\varepsilon}\P\big(\LCS^*(\tau_{*,1},\tau_{*,1}')>n^{1/2-2\alpha}\big)\Big)^{d/7}.\]
Now recall that $\LCS^*(\tau_{*,1},\tau_{*,1}')$ is defined as a maximum taken over the children of the roots of $\tau_{*,1}$ and $\tau_{*,1}'$, so that another union bound yields that
\[\P\big(\LCS^*(\tau_{*,1},\tau_{*,1}')>n^{1/2-2\alpha}\big)\leq \E\big[\rk_\varnothing(\tau_{*,1})\big]\E\big[\rk_\varnothing(\tau_{*,1}')\big]\P\big(\LCS^\bullet(\tau,\tau')>n^{1/2-2\alpha}\big)\]
by branching property of root-biased Bienaymé trees (\cref{def:root-biased}). Moreover, recall from \eqref{first-moment_root-biased} that $\E\big[\rk_\varnothing(\tau_{*,1})\big]\E\big[\rk_\varnothing(\tau_{*,1}')\big]<\infty$ since $\mu$ and $\mu'$ have finite second moments.

Finally, since $4\alpha+\varepsilon<1/2$ by assumption, we can fix $\xi>0$ (depending only on $\alpha$ and $\varepsilon$) such that $\tfrac{1}{2}+\varepsilon-\tfrac{1-4\alpha}{1+\xi}<0$. Then, by \cref{cor:LCS-tail}, there is a constant $C > 0$ such that 
\[
\prob(A_n^{(2)} \cap B_n^{(2)}) \leq 2C n^{5/2+\e} \Big(n^{1/2+\varepsilon-\frac{2}{1+\xi}(1/2-2\alpha)} \Big)^{d/7} = 2C n^{5/2+\e + \left(1/2+\e -\frac{1-4\alpha}{1+\xi}\right)d/7}.
\]
Given any $m\geq 1$, the right-hand side is bounded by $2Cn^{-m}$ for sufficiently large $d$.
\end{proof}

\subsection{Bad event (3): flower}
Recall that $B_n^{(3)}$ is the bad event that $\tau$ and $\tau'$ have some common subtree $\ttT$ with a vertex $\ttu \in \ttT$ such that the forest obtained by deleting $\ttu$ has at least $d$ distinct components of size larger than $n^{1/2-2\alpha}$. We consider the event $A_n^{(3)}:=\{\Ht(\tau)\vee\Ht(\tau')\leq n-1 ; \Delta(\tau)\leq n^{1/2-2\varepsilon} ; \Delta(\tau')\leq n^{1/2+\varepsilon}\}$ which contains the typical event $A_n$. Here, our goal is to prove the following result. 

\begin{prop}
\label{prop:no_flowers}
    If $0 < 4\alpha < \varepsilon < 1/2$ then for all $m \geq 1$, there exists $d \in \N$ such that $\prob(A_n^{(3)}\cap B_n^{(3)}) \leq n^{-m}$ for all $n$ large enough.
\end{prop}

We begin with a deterministic statement.

\begin{lem}
\label{deter-lemma:no_flowers}
If $B_n^{(3)}$ occurs, then there exist $u \in \tau$, $u' \in \tau'$, $1 \leq i_1 < ... < i_{d-2} \leq \rk_{u}(\tau)$, and $1 \leq j_1 < ... < j_{d-2} \leq \rk_{v}(\tau')$ such that $\LCS^\bullet(\theta_{u*i_\ell}\tau,\theta_{u'*j_\ell}\tau') > n^{1/2-2\alpha}$ for all $1 \leq \ell \leq d-2$.
\end{lem}

\begin{proof}
Let $\phi$ and $\phi'$ be the embeddings of $\ttT$ into $\tau$ and $\tau'$ respectively. Denote by $\ttT_1,...,\ttT_d$ distinct components of the forest obtained by deleting $\ttu$ with size larger than $n^{1/2-2\alpha}$. For each $1\leq \ell\leq d$, denote by $\ttu_\ell$ the neighbour of $\ttu$ that belongs to $\ttT_\ell$, and root $\ttT_\ell$ at $\ttu_\ell$. By definition of embeddings, $\phi(\ttu)$ and $\phi(\ttu_\ell)$ (resp.~$\phi'(\ttu)$ and $\phi'(\ttu_\ell)$) are neighbours in $\tau$ (resp.~$\tau'$). However, in a rooted tree, the neighbours of a vertex are exactly its children and potentially its unique parent. Hence, without loss of generality, we can assume that for all $1\leq \ell\leq d-2$, $\phi(\ttu_\ell)=\phi(\ttu)*i_\ell$ and $\phi'(\ttu_\ell)=\phi'(\ttu)*j_\ell$, for some integers $1 \leq i_1 < ... < i_{d-2} \leq \rk_{u}(\tau)$ and $1 \leq j_1 < ... < j_{d-2} \leq \rk_{v}(\tau')$. It then follows that for all $1\leq \ell\leq d-2$, $\ttT_\ell$ is a common \emph{rooted} subtree of $\theta_{\phi(\ttu_\ell)}\tau$ and $\theta_{\phi'(\ttu_\ell)}\tau'$ and thus $\LCS^\bullet(\theta_{\phi(\ttu_\ell)}\tau,\theta_{\phi'(\ttu_\ell)}\tau')\geq \#\ttT_\ell>n^{1/2-2\alpha}$. Choosing $u=\phi(\ttu)$ and $u'=\phi'(\ttu)$ concludes the proof.
\end{proof}

\begin{proof}[Proof of \cref{prop:no_flowers}]
For any plane trees $t$ and $t'$, define
\[F(t,t')=\I{\rk_\varnothing(t)\leq n^{1/2-2\varepsilon}}\I{\rk_\varnothing(t')\leq n^{1/2+\varepsilon}} \sum_{\substack{1 \leq i_1 < \ldots < i_{d-2} \leq \rk_\emptyset(t) \\ 1 \leq j_1 < \ldots < j_{d-2} \leq \rk_\emptyset(t')}}\prod_{\ell=1}^{d-2} \I{\LCS^\bullet(\theta_{i_\ell}t , \theta_{j_\ell}t') > n^{1/2-2\alpha}}.\]
Recall that on $A_n^{(3)}$, we have $\Ht(\tau)\vee \Ht(\tau')\leq n-1$. Applying \cref{deter-lemma:no_flowers} and then the Many-to-One principle (\cref{many-to-one}), we write
\[
\prob(A_n^{(3)} \cap B_n^{(3)})\leq \sum_{h=0}^{n-1} \sum_{h'=0}^{n-1} \E\bigg[\sum_{u\in\tau,u'\in\tau'}\I{|u|=h}\I{|u'|=h'} F(\theta_u \tau,\theta_{u'}\tau')\bigg]=n^2\E\big[F(\tau,\tau')\big].
\]
Next, using the branching property of Bienaymé trees along with a union bound, we obtain
\begin{align*}
\E\big[F(\tau,\tau')\big] &\leq
    \begin{multlined}[t]
    \E\Big[\I{\rk_\varnothing(\tau)\leq n^{1/2-2\varepsilon};\rk_\varnothing(\tau')\leq n^{1/2+\varepsilon}}\rk_\varnothing(\tau)^{d-2}\rk_\varnothing(\tau')^{d-2}\Big]\\
    \cdot \prob\big(\LCS^\bullet(\tau,\tau')>n^{1/2-2\alpha}\big)^{d-2}
    \end{multlined}\\
&\leq \Big(n^{1-\varepsilon}\prob\big(\LCS^\bullet(\tau,\tau')>n^{1/2-2\alpha}\big)\Big)^{d-2}.
\end{align*}
Then, thanks to \cref{cor:LCS-tail}, we have a constant $C > 0$ such that
\[\prob(A_n^{(3)}\cap B_n^{(3)}) \leq C^{d-2}n^{2+(d-2)\big(1-\varepsilon - \frac{1-4\alpha}{1+\xi}\big)},\]
where $\xi:=\frac{\varepsilon-4\alpha}{2(1-\varepsilon)}>0$ by assumption. Note that $1-\varepsilon-\frac{1-4\alpha}{1+\xi}<1-\varepsilon-\frac{1-4\alpha}{1+2\xi}=0$. From here, we can choose $d$ large enough so that $(d-2)(1-\varepsilon - \frac{1-4\alpha}{1+\xi})\leq -m-2$. This then entails that $\prob(A_n^{(3)}\cap B_n^{(3)}) \leq C^{d-2}n^{-m}$, concluding the proof.
\end{proof}

\subsection{Bad event (4): bush}
Recall that $B_n^{(4)}$ is the bad event that $\tau$ and $\tau'$ have a common subtree $\ttT$ equipped with a vertex $\ttu\in\ttT$ such that if $\ttT_1,\ldots,\ttT_k$ stands for the components of the forest obtained by removing $\ttu$ from $\ttT$ then
\[\sum_{i=1}^k \#\ttT_i\wedge n^{1/2-2\alpha} > n^{1/2-\alpha}-1.\]
Here, we shall also consider the event $A_n^{(4)}:=\big\{\Ht(\tau)\vee\Ht(\tau')\leq n-1 \,;\, \Delta(\tau)\leq n^{1/2-\varepsilon}\big\}$, which contains the typical event $A_n$. Our goal is to show the following result.

\begin{prop}
\label{prop:no_bushes}
If $0<3\alpha<\varepsilon<1/2$ then for all $m\geq1$, it holds that $\P(A_n^{(4)}\cap B_n^{(4)})\leq n^{-m}$ for all $n$ large enough.
\end{prop}

We first reason deterministically by showing the following lemma.

\begin{lem}
\label{deter-lemma:no_bushes}
Assume that $2+4n^{1/2-2\alpha}<n^{1/2-\alpha}$. If $B_n^{(4)}$ occurs, then at least one of the two following situations happens.
\begin{enumerate}
    \item[\emph{(i)}] There exist $v\in\tau$ and $v'\in\tau'$ such that $n^{-\alpha}\LCS^\bullet(\theta_v\tau,\theta_{v'}\tau')>\Ht(\theta_v\tau)\wedge\Ht(\theta_{v'}\tau')+1$.
    \item[\emph{(ii)}] There exists $u\in\tau$ such that $\sum_{j=1}^{\rk_u(\tau)}\big(\Ht(\theta_{u*j}\tau)+1\big)\wedge n^{1/2-3\alpha}>\tfrac{1}{2}n^{1/2-2\alpha}$.
\end{enumerate}
\end{lem}

\begin{proof}
Let us assume that (i) does not happen to prove (ii). Let $\phi$ and $\phi'$ be the embeddings of $\ttT$ into $\tau$ and $\tau'$ respectively. For each $1\leq i\leq k$, root $\ttT_i$ at the unique neighbour $\ttu_i$ of $\ttu$ that belongs to $\ttT_i$.  By the same arguments as in the proof of \cref{deter-lemma:no_flowers}, we can assume without loss of generality that for all $1\leq i\leq k-2$, $\phi(\ttu_i)$ and $\phi'(\ttu_i)$ are children of $\phi(\ttu)$ and $\phi'(\ttu)$, respectively, and so that $\ttT_i$ is a common \emph{rooted} subtree of $\theta_{\phi(\ttu_i)}\tau$ and $\theta_{\phi'(\ttu_i)}\tau$. Then, \[\#\ttT_i\leq\LCS^\bullet(\theta_{\phi(\ttu_i)}\tau,\theta_{\phi'(\ttu_i)}\tau')\leq n^\alpha \Ht(\theta_{\phi(\ttu_i)}\tau)+n^\alpha,\]
since (i) does not happen. Therefore,
\begin{align*}
n^{1/2-\alpha}-1&<2n^{1/2-2\alpha}+\sum_{i=1}^{k-2}n^{1/2-2\alpha}\wedge (n^\alpha\Ht(\theta_{\phi(\ttu_i)}\tau)+n^\alpha)\\
&\leq 2n^{1/2-2\alpha}+\sum_{j=1}^{\rk_{\phi(\ttu)}(\tau)}n^{1/2-2\alpha}\wedge (n^\alpha\Ht(\theta_{\phi(\ttu)*j}\tau)+n^\alpha)\\
&=2n^{1/2-2\alpha}+n^\alpha\sum_{j=1}^{\rk_{\phi(\ttu)}(\tau)}n^{1/2-3\alpha}\wedge \big(\Ht(\theta_{\phi(\ttu)*j}\tau)+1\big).
\end{align*}
Using $1+2n^{1/2-2\alpha}<\tfrac{1}{2}n^{1/2-\alpha}$, rearranging the last inequality gives (ii) with $u=\phi(\ttu)$.
\end{proof}

\begin{proof}[Proof of \cref{prop:no_bushes}]
For any plane trees $t$ and $t'$, define
\begin{align*} 
F(t,t')&=\I{\Ht(t)\wedge\Ht(t')\leq n\, ;\, n^{-\alpha}\LCS^\bullet(t,t')>\Ht(t)\wedge\Ht(t')+1},\\
G(t)&=\I{\Delta(t)\leq n^{1/2-\varepsilon}}\I{\sum_{j=1}^{\rk_\varnothing(t)}(\Ht(\theta_j t)+1)\wedge n^{1/2-3\alpha}>\tfrac{1}{2}n^{1/2-2\alpha}}.
\end{align*}
Thanks to \cref{deter-lemma:no_bushes}, when $n$ is large enough, we can write
\begin{align*}
\xi&:=\P(A_n^{(4)}\cap B_n^{(4)})\\
&\leq \E\bigg[\I{\Ht(\tau)\leq n-1;\Ht(\tau')\leq n-1} \sum_{\substack{v\in\tau \\ v'\in\tau'}} F(\theta_{v}\tau,\theta_{v'}\tau')\bigg]
+\E\bigg[\I{\Ht(\tau)\leq n-1}\sum_{u\in\tau}G(\theta_u\tau)\bigg]\\
&\leq \sum_{h=0}^{n-1}\sum_{h'=0}^{n-1} \E\bigg[\sum_{\substack{v\in\tau \\ v'\in\tau'}}\I{|v|=h}\I{|v'|=h'}F(\theta_{v}\tau,\theta_{v'}\tau')\bigg]+\sum_{h=0}^{n-1}\E\bigg[\sum_{u\in\tau}\I{|u|=h}G(\theta_u\tau)\bigg].
\end{align*}
Next, applying the Many-to-One principle (\cref{many-to-one}) gives us
\[\P(A_n^{(4)}\cap B_n^{(4)})\leq n^2\E\big[F(\tau,\tau')\big]+n\E\big[G(\tau)\big].\]
On the one hand, \cref{prop:height=size} yields that $\E\big[F(\tau,\tau')\big]\leq C_{m,\alpha} n^{-m-2}$ for some constant $C_{m,\alpha}>0$. On the other hand, if $(\tau_j)_{j\geq 1}$ is a sequence of i.i.d.~Bienaymé trees with offspring distribution $\mu$ then by branching property,
\[\E\big[G(\tau)\big]\leq \P\bigg(\sum_{j=1}^{n^{1/2-\varepsilon}}\big(\Ht(\tau_j)+1\big)\wedge n^{1/2-3\alpha}>\tfrac{1}{2}n^{1/2-2\alpha}\bigg).\]
Then, combining \cref{single_tree_height} and \cref{obj_1} entails that $\E\big[G(\tau)\big]\leq C\exp(-\tfrac{1}{4}n^\alpha)$ for some constant $C>0$ (note that we use the assumption $3\alpha<\varepsilon<1/2$ to be able to choose $\gamma=\tfrac{\varepsilon-3\alpha}{1/2-\varepsilon}>0$ in the application of \cref{obj_1}). Finally, when $n$ is large enough,
\[\P(A_n^{(4)}\cap B_n^{(4)})\leq C_{m,\alpha} n^2 n^{-m-2}+C n \exp(-\tfrac{1}{4}n^\alpha)\leq 2C_{m,\alpha}n^{-m}.\qedhere\]
\end{proof}

\section{Scaling limit}\label{sec:Scaling_limit}

In this section, we work with compact real trees, defined as follows.
\begin{defn}
\label{real-tree_def}
A compact metric space $(T,d)$ is a \emph{real tree} if for all $x,y\in T$, it holds that:
\begin{enumerate}
    \item[(a)] There is a unique isometric embedding $g:[0,d(x,y)]\to T$ such that $g(0)=x$ and $g(d(x,y))=y$; then write $\llbracket x,y\rrbracket=g([0,d(x,y)])$ and call it the \emph{geodesic from $x$ to $y$}.
    \item[(b)] If a function $h:[0,1]\to T$ is continuous and injective, then $h([0,1])=\llbracket h(0),h(1)\rrbracket$.
\end{enumerate}
\end{defn}

We first recall some necessary definitions and facts about these trees, then proceed to the main proposition of this section, which completes the proof of \cref{thm:main}.

\subsection{Compact real trees}\label{sec:compact_real_tree}
Let $\bbT_\R$ be the set of (equivalence classes of) compact real trees equipped the Gromov--Hausdorff topology. Recall that the Gromov--Hausdorff distance can be defined as follows. For $\rT_1 = (T_1, d_1)$ and $ \rT_2 = (T_2, d_2)$ two elements of $\bbT_\R$, a \emph{correspondence} $\cR$ between $\rT_1$ and $\rT_2$ is a measurable subset of $T_1 \times T_2$ such that for every $x_1 \in T_1$ there exists some $x_2 \in T_2$ with $(x_1, x_2) \in \cR$, and for every $y_2 \in T_2$ there exists some $y_1 \in T_1$ with $(y_1, y_2) \in \cR$. Then, the \emph{distortion} of a correspondence $\cR$ is 
\begin{equation}\label{def-distortion}
    \mathrm{dis}(\cR) = \sup\big\{ |d_1(x_1, y_1) - d_2(x_2, y_2)| : (x_1, x_2), (y_1, y_2) \in \cR \big\}.
\end{equation}
The \emph{Gromov--Hausdorff distance} between $\rT_1$ and $\rT_2$ is $\dGH(\rT_1, \rT_2) := \frac{1}{2}\inf \mathrm{dis}(\cR)$, where the infimum runs over all correspondences between $\rT_1$ and $\rT_2$.
\smallskip 

Let $\rT=(T,d)$ be a compact real tree. We first need the notion of branch-points for real trees. For any $x,y\in T$, recall that $\llbracket x,y\rrbracket$ stands for the geodesic from $x$ to $y$ on $\rT$. From \cref{real-tree_def}, we have the following:
\begin{equation}
\label{geo_distance}
\llbracket x,y\rrbracket=\{z\in T\, :\, d(x,y)=d(x,z)+d(z,y)\}.
\end{equation}
Observe that for any $x,y,w\in T$, we have $w\in\llbracket x,y\rrbracket$ if and only if $\llbracket x,w\rrbracket \cap \llbracket w,y\rrbracket = \{w\}$, and their union is equal to $\llbracket x,y\rrbracket$ in this case.
Let $x,y,z\in T$. We will now identify the vertex $b$ were the paths between pairs of $x,y,z$ meet and call this their branch-point. The set $\llbracket x,y\rrbracket$ is a compact subset of $T$ so there exists $b\in \llbracket x,y\rrbracket$ such that $d(z,b)=d(z,\llbracket x,y\rrbracket)$. In particular, $\llbracket z,b\rrbracket$ and $\llbracket x,y\rrbracket$ only meet at $b$, and $\llbracket x,b\rrbracket\cup\llbracket b,y\rrbracket=\llbracket x,y\rrbracket$. It follows that $b$ belongs to the three geodesics $\llbracket x,y\rrbracket$, $\llbracket y,z\rrbracket$, and $\llbracket z,x\rrbracket$. Then, we can write
\[\llbracket x,y\rrbracket\cap \llbracket y,z \rrbracket \cap\llbracket z,x\rrbracket=(\llbracket x,b\rrbracket\cup \llbracket b,y\rrbracket)\cap (\llbracket y,b\rrbracket\cup \llbracket b,z\rrbracket) \cap (\llbracket z,b\rrbracket\cup \llbracket b,x\rrbracket)=\{b\}.\]

To sum up, for any $x,y,z\in T$, the intersection of the three geodesics $\llbracket x,y\rrbracket$, $\llbracket y,z\rrbracket$, and $\llbracket z,x\rrbracket$ is a singleton $\{b\}$. We call $b$ the \emph{branch-point of $x,y,z$ in $T$}. It satisfies the identity
\begin{equation}
\label{distance-to-branch-point}
2d\big(z,b\big)=2d\big(z,\llbracket x,y\rrbracket\big)=d(x,z)+d(y,z)-d(x,y).
\end{equation}

Next, we define the set of \emph{branch-points of $\rT$} as 
\[\mathrm{Br}(\rT)=\{x\in T\, :\, T\setminus\{x\}\text{ has at least three connected components}\}.\]
Note that any vertex can be a branch-point between three points but not every (such) vertex is in $\mathrm{Br}(\rT)$. More precisely, a point $b\in T$ belongs to $\mathrm{Br}(\rT)$ if and only if there exist $x,y,z\in T\setminus\{b\}$ such that $b$ is the branch-point of $x,y,z$.
\smallskip

Let $x_1,\ldots,x_N\in T$ where $N\in \N$. We define the \emph{subtree of $T$ spanned by $x_1,\ldots,x_N$} as follows:
\[\mathrm{Span}(\rT\, ;\, x_1,\ldots,x_N)=\bigcup_{1\leq i,j\leq N}\llbracket x_i,x_j\rrbracket.\]
It is a closed subset of $T$ and, by \eqref{distance-to-branch-point}, it is easy to see that for any $z\in T$,
\[d\big(z,\mathrm{Span}(\rT\, ;\, x_1,\ldots,x_N)\big)=\tfrac{1}{2}\min_{1\leq i,j\leq N}\big(d(x_i,z)+d(x_j,z)-d(x_i,x_j)\big).\]
Thus, inductively, we see that the \emph{total length} of the spanned tree is given by
\begin{align}
\label{formula_length_span}
\mathrm{Len}(\rT\, ;\, x_1\ldots,x_N)&=d\big(x_N,\mathrm{Span}(\rT\, ;\, x_1,\ldots,x_{N-1})\big)+\mathrm{Len}(\rT\, ;\, x_1\ldots,x_{N-1})\\
\nonumber
&=\tfrac{1}{2}\sum_{k=2}^N \min_{1\leq i,j\leq k-1}\big(d(x_i,x_k)+d(x_j,x_k)-d(x_i,x_j)\big).
\end{align}
In particular, using the triangle inequality on $d$, we see that for any $x_1,x_2,x_3\in T$,
\begin{equation}
\label{formula_length_span3}
\mathrm{Len}(\rT\, ;\, x_1,x_2,x_3)=\tfrac{1}{2}d(x_1,x_2)+\tfrac{1}{2}d(x_2,x_3)+\tfrac{1}{2}d(x_3,x_1).
\end{equation}

For any $N\in\N$ and any $\rT=(T,d)\in \bbT_\R$ and $\rT'=(T',d') \in \bbT_\R$, set the length of their longest common subtrees spanned by at most $N$ leaves to be
\begin{multline}
\label{expr_real-tree_LCS_N}
\LCS_N(\rT,\rT')=\sup\big\{\mathrm{Len}(\rT\, ;\, x_1\ldots,x_N) \, : \, \, x_1,\ldots,x_N\in T\text{ and }x_1',\ldots,x_N'\in T'\\
\text{such that }  d(x_i,x_j)=d'(x_i',x_j')\text{ for all }1\leq i,j\leq N \big\}.
\end{multline}
One can readily observe that if $1\leq N\leq M$, then $\LCS_N(\rT,\rT')\leq \LCS_M(\rT,\rT')$. Also, note that $\lambda \LCS_N(\rT,\rT') = \LCS_N(\lambda\cdot \rT, \lambda\cdot \rT')$ for any $\lambda > 0$, where we recall that $\lambda\cdot \rT=(T,\lambda d)$. 
We can view a discrete tree as a real tree by replacing the edges by line segments of length~$1$. Then, our original expression \eqref{expr_dscr_LCS_N} for $\LCS_N$ agrees with the more general expression \eqref{expr_real-tree_LCS_N} of $\LCS_N$ for real trees.
\smallskip 

Finally, we recall some facts about the Brownian continuum random tree $\sT$. 
First, there exists a uniform measure on $\sT$ which is almost surely fully supported~\cite{aldous1991continuumI}. Next, we note that $\sT$ is almost surely binary in the sense that, for every $x \in \mathrm{Br}(\sT)$, $\mathrm{Br}(\sT) \setminus \{x\}$ has exactly three connected components~\cite{aldous1991continuumI}. Finally, we point out that $\mathrm{Br}(\sT)$ is at most countable, as is the case for all compact real trees (see e.g.~\cite[Lemma 3.1]{duquesne2007growth}).
\smallskip 

The following lemma follows directly from \eqref{geo_distance}. 
\begin{lem}
\label{5points_1path} 
Let $\rT=(T,d)$ be a compact real tree. Let $x_1,x_2\in T$ and let $y\in\llbracket x_1,x_2\rrbracket$. If $\tilde{x}_1\in\llbracket x_1,y\rrbracket$ and $\tilde{x}_2\in \llbracket y,x_2\rrbracket$, then $y\in\llbracket \tilde{x}_1,\tilde{x}_2\rrbracket$.
\end{lem}

\subsection{Proof of \cref{thm:main}}\label{sec:proof_main} In this section, we finally complete the proof of \cref{thm:main}. To do so, we use the following last intermediate result, stating that we can replace $\LCS_N$ with $\LCS_3$.

\begin{prop}\label{prop:LCSN-to-LCS3}
    Let $\sT$ and $\sT'$ be two independent Brownian continuum random trees such that $(n^{-1/2}\cdot\tau_n,n^{-1/2}\cdot\tau_n') \xrightarrow[n\to\infty]{\mathrm d}  (\frac{2}{\sigma}\cdot\sT, \frac{2}{\sigma'}\cdot\sT')$. Then, jointly with this convergence,
\[\frac{1}{\sqrt{n}}\LCS_N(\tau_n,\tau_n') \xrightarrow[n\to\infty]{\mathrm d} \LCS_3\big(\tfrac{2}{\sigma}\cdot \sT,\tfrac{2}{\sigma'}\cdot \sT'\big)\] 
for any integer $N\geq 3$. 
\end{prop}

Combining this proposition with \cref{thm:size-to-length} yields our main result \cref{thm:main}.

\begin{proof}[Proof of \cref{thm:main}]
     \cref{thm:size-to-length,prop:LCSN-to-LCS3} yield that $n^{-1/2}\LCS(\tau_n,\tau_n')$ converges in law towards $c_{\mu,\mu'}\LCS_3\big(\tfrac{2}{\sigma}\cdot \sT,\tfrac{2}{\sigma'}\cdot \sT'\big)$, jointly with $(n^{-1/2}\cdot\tau_n,n^{-1/2}\cdot\tau_n') \convdist (\frac{2}{\sigma}\cdot\sT, \frac{2}{\sigma'}\cdot\sT')$. Comparing \eqref{Phi} with \eqref{formula_length_span3} and \eqref{expr_real-tree_LCS_N}, we see that $\Phi = \LCS_3$.
\end{proof}

A minor difficulty in proving \cref{prop:LCSN-to-LCS3} is caused by the discontinuity of $\LCS_N$ when $N\geq 4$. Indeed, observe that when $T_n$ consists of two cherries with legs of length $1$, connected by a line-segment of length $1+1/n$ and $T$ consists of two cherries with legs of length $1$, connected by a line-segment of length $1$, then $T_n\to T$ in the Gromov--Hausdorff topology, yet for $N\ge 4$ it holds that $\LCS_N(T_n,T)=4$ while $\LCS_N(T,T)=5$. Or, more generally, pairs of macroscopic branch-points at a different distance in $T_n$ and $T'_n$ may lie at the same distance from each other in their limits leading to a much longer common subtree. However, as pairs of macroscopic branch-points in $\frac{2}{\sigma}\cdot\sT$ and $\frac{2}{\sigma'}\cdot\sT'$ almost surely are not at the same distance from each other, this difficulty is easy to overcome in the following three steps.

\begin{lem}\label{lem:upper-semi-cts}
    For all $N\geq 3$, $\LCS_N$ is upper semi-continuous on $\bbT_\R\times\bbT_\R$.
\end{lem}
\begin{lem}\label{lem:LCS3-cts}
The function $\LCS_3$ is continuous on $\bbT_\R\times\bbT_\R$.
\end{lem}
\begin{lem}\label{lem:N-eq-3-CRT} 
For all $N\geq 4$, it holds that $\LCS_N(\frac{2}{\sigma}\cdot\sT,\frac{2}{\sigma'}\cdot\sT')=\LCS_3(\frac{2}{\sigma}\cdot\sT,\frac{2}{\sigma'}\cdot\sT')$ almost surely.
\end{lem}

\begin{proof}[Proof of \cref{prop:LCSN-to-LCS3} given \cref{lem:upper-semi-cts,lem:LCS3-cts,lem:N-eq-3-CRT}] Using Skorokhod's representation theorem, we may work on a probability space where the convergences $n^{-1/2}\cdot\tau_n \convdist \frac{2}{\sigma}\cdot\sT$ and $n^{-1/2}\cdot\tau_n' \convdist \frac{2}{\sigma'}\cdot\sT'$ hold almost surely. Then, we can write
\begin{multline*}
\LCS_3\big(\tfrac{2}{\sigma}\cdot\sT,\tfrac{2}{\sigma'}\cdot\sT'\big)=\lim \LCS_3(n^{-1/2}\cdot\tau_n,n^{-1/2}\cdot\tau_n')\leq \liminf \LCS_N(n^{-1/2}\cdot\tau_n,n^{-1/2}\cdot\tau_n')\\
\leq \limsup \LCS_N(n^{-1/2}\cdot\tau_n,n^{-1/2}\cdot\tau_n')\leq \LCS_N\big(\tfrac{2}{\sigma}\cdot\sT,\tfrac{2}{\sigma'}\cdot\sT'\big)=\LCS_3\big(\tfrac{2}{\sigma}\cdot\sT,\tfrac{2}{\sigma'}\cdot\sT'\big),
\end{multline*}
where the first equality holds by \cref{lem:LCS3-cts}, the last inequality holds by \cref{lem:upper-semi-cts} and the final equality holds by \cref{lem:N-eq-3-CRT}.
\end{proof}

\begin{proof}[Proof of \cref{lem:upper-semi-cts}]
Let $\rT=(T,d_T)$ and $\rS=(S,d_S)$ be two compact real trees. Also, for $n \geq 1$, let $\rT^n=(T^n,d_T^n)\in\bbT_\R$ and $\rS^n=(S^n,d_S^n)\in\bbT_\R$ such that $\rT^n\to\rT$ and $\rS^n\to\rS$. By the definition of $\LCS_N$, for any $n\geq 1$, there exist $x_1^n,\ldots,x_N^n\in T^n$ and $y_1^n,\ldots,y_N^n\in S^n$ such that
\begin{itemize}
    \item $d_T^n(x_i^n,x_j^n)=d_S^n(y_i^n,y_j^n)$ for all $1\leq i,j\leq N$, and
    \item $\mathrm{Len}(\rT^n\, ;\, x_1^n,\ldots,x_N^n)\geq \LCS_N(\rT^n,\rS^n)-1/n$.
\end{itemize}
Next, we choose two sequences of correspondences $\{ \cR_T^n\}_{n\geq 1}$ and $\{\cR_S^n\}_{n\geq 1}$, respectively between $\rT^n$ and $\rT$ and between $\rS^n$ and $\rS$, such that $\mathrm{dis}(\cR_T^n)\to 0$ and $\mathrm{dis}(\cR_S^n)\to 0$. Then, for any $n\geq 1$ and $1\leq i\leq N$, we pick $\hat{x}_i^n\in T$ and $\hat{y}_i^n\in S$ such that $(x_i^n,\hat{x}_i^n)\in\cR_T^n$ and $(y_i^n,\hat{y}_i^n)\in\cR_S^n$. Observe that for all $1\leq i,j\leq N$,
\[\big|d_T(\hat{x}_i^n,\hat{x}_j^n)-d_S(\hat{y}_i^n,\hat{y}_j^n)\big|\leq \big|d_T^n({x}_i^n,{x}_j^n)-d_S^n({y}_i^n,{y}_j^n)\big|+\mathrm{dis}(\cR_T^n)+\mathrm{dis}(\cR_S^n)=\mathrm{dis}(\cR_T^n)+\mathrm{dis}(\cR_S^n)\]
and thus that $\big|d_T(\hat{x}_i^n,\hat{x}_j^n)-d_S(\hat{y}_i^n,\hat{y}_j^n)\big|\to 0$. Potentially by extracting subsequences, the compactness of $T$ and $S$ allows us to assume that there exist $x_i\in T$ and $y_i\in S$ such that $\hat{x}_i^n\to x_i$ and $\hat{y}_i^n\to y_i$ for all $1\leq i\leq N$. It follows that $d_T(x_i,x_j)=d_S(y_i,y_j)$ for all $1\leq i,j\leq N$, and so it holds that
\[\LCS_N(\rT,\rS)\geq \mathrm{Len}(\rT\, ;\, x_1,\ldots,x_N).\]

Now, using the definitions of the total length of a spanned tree \eqref{formula_length_span} and of the distortion of a correspondence \eqref{def-distortion}, we have
\[\mathrm{Len}(\rT\, ;\, \hat{x}_1^n,\ldots,\hat{x}_N^n)\geq \mathrm{Len}(\rT^n\, ;\, x_1^n,\ldots,x_N^n)-\tfrac{3}{2}(N-1)\mathrm{dis}(\cR_T^n)\]
for all $n\geq 1$. Moreover, it is clear from \eqref{formula_length_span} that the function $(z_1,\ldots,z_N)\in T\times \ldots\times T\longmapsto \mathrm{Len}(\rT\, ;\, z_1,\ldots,z_N)\in \R$ is continuous, which implies that
\[\mathrm{Len}(\rT\, ;\, x_1,\ldots,x_N)=\lim_{n\to\infty}\mathrm{Len}(\rT\, ;\, \hat{x}_1^n,\ldots,\hat{x}_N^n).\]
Therefore, we obtain that
\[\LCS_N(\rT,\rS)\geq \limsup_{n\to\infty}\mathrm{Len}(\rT^n\, ;\, x_1^n,\ldots,x_N^n)\geq \limsup_{n\to\infty}\LCS_N(\rT^n,\rS^n).\qedhere\]
\end{proof}

\begin{proof}[Proof of \cref{lem:LCS3-cts}]
Given \cref{lem:upper-semi-cts}, it remains only to show that $\LCS_3$ is lower semi-continuous. Let $\rT^n$, $\rS^n$, $\rT$, $\rS$, $\{\cR_T^n\}_{n \geq 1}$, and $\{ \cR_S^n\}_{n \geq 1}$ be as in the proof of \cref{lem:upper-semi-cts}. 
For all $n\geq 1$, we can choose $\hat{x}_1^n,\hat{x}_2^n,\hat{x}_3^n\in T$ and $\hat{y}_1^n,\hat{y}_2^n,\hat{y}_3^n\in S$ such that
\begin{itemize}
    \item $d_T(\hat{x}_i^n,\hat{x}_j^n)=d_S(\hat{y}_i^n,\hat{y}_j^n)$ for all $1\leq i,j\leq 3$, and
    \item $\mathrm{Len}(\rT\, ;\, \hat{x}_1^n,\hat{x}_2^n,\hat{x}_3^n)\geq \LCS_3(\rT,\rS)-1/n$.
\end{itemize}

Now, for any $n\geq 1$ and $1\leq i\leq 3$, we choose $x_i^n\in T^n$ and $y_i^n\in S^n$ such that $(x_i^n,\hat{x}_i^n)\in\cR_T^n$ and $(y_i^n,\hat{y}_i^n)\in\cR_S^n$. In particular, \eqref{formula_length_span3} ensures that
\begin{equation}
\label{II_tool_final}
\liminf_{n\to\infty}\mathrm{Len}(\rT^n\, ;\, x_1^n,x_2^n,x_3^n) = \liminf_{n\to\infty} \mathrm{Len}(\rT \, ;\, \hat{x}_1^n,\hat{x}_2^n,\hat{x}_3^n)\geq \LCS_3(\rT,\rS).
\end{equation}
Then, denote by $b^n$ (resp.~by $a^n$) the branch-point of $x_1^n,x_2^n,x_3^n$ in $\rT^n$ (resp.~of $y_1^n,y_2^n,y_3^n$ in $\rS^n$). Using \eqref{distance-to-branch-point}, we obtain that for all $1\leq i\leq 3$,
\begin{equation}
\label{II_tool}
\big|d_T^n(b^n,x_i^n)-d_S^n(a^n,y_i^n)\big|\to 0.
\end{equation}

Finally, for any $n\geq 1$ and $1\leq i\leq 3$, we define $\tilde{x}_i^n$ (resp.~$\tilde{y}_i^n)$ as the unique point of the geodesic from $b^n$ to $x_i^n$ on $\rT^n$ (resp.~from $a^n$ to $y_i^n$ on $\rS^n$) such that 
\[d_T^n(b^n,\tilde{x}_i^n)=d_T^n(b^n,x_i^n)\wedge d_S^n(a^n,y_i^n)\quad\text{ and }\quad d_S^n(a^n,\tilde{y}_i^n)=d_T^n(b^n,x_i^n)\wedge d_S^n(a^n,y_i^n).\]
It follows from \cref{5points_1path} that $b^n$ (resp.~$a^n$) is the branch-point of $\tilde{x}_1^n,\tilde{x}_2^n,\tilde{x}_3^n$ in $\rT^n$ (resp.~of $\tilde{y}_1^n,\tilde{y}_2^n,\tilde{y}_3^n$ in $\rS^n$), so using \eqref{geo_distance}, we have that for any distinct $1\leq i,j\leq 3$,
\begin{align*} d_T^n(\tilde{x}_i^n,\tilde{x}_j^n)=d_T^n(b^n,\tilde{x}_i^n)+d_T^n(b^n,\tilde{x}_j^n)=d_S^n(a^n,\tilde{y}_i^n)+d_S^n(a^n,\tilde{y}_j^n)=d_S^n(\tilde{y}_i^n,\tilde{y}_j^n).
\end{align*}
This implies that 
\[\LCS_3(\rT^n,\rS^n)\geq \mathrm{Len}(\rT^n\, ;\, \tilde{x}_1^n,\tilde{x}_2^n,\tilde{x}_3^n).\] 
Now, since \eqref{II_tool} implies that for all $1 \leq i \leq 3$, 
$\big|d_T^n(b^n,\tilde{x}_i^n)-d_T^n(b^n,x_i^n)\big|\to 0$,
we also have $\big|d_T^n(\tilde{x}_i^n,\tilde{x}_j^n)-d_T^n(x_i^n,x_j^n)\big|\to 0$ 
for all $1\leq i,j\leq 3$. Combined with \eqref{formula_length_span3}, this entails that
\[\Big|\mathrm{Len}(\rT^n\, ;\, \tilde{x}_1^n,\tilde{x}_2^n,\tilde{x}_3^n)-\mathrm{Len}(\rT^n\, ;\, x_1^n,x_2^n,x_3^n)\Big|\to 0.\]
These two statements followed by \eqref{II_tool_final} allow us to conclude
\begin{align*}
    \liminf_{n\to\infty}\LCS_3(\rT^n,\rS^n)\geq \liminf_{n\to\infty} \mathrm{Len}(\rT^n\, ;\, \tilde{x}_1^n,\tilde{x}_2^n,\tilde{x}_3^n) 
    &= \liminf_{n\to\infty} \mathrm{Len}(\rT^n\, ;\, x_1^n,x_2^n,x_3^n) \\
    &\geq \LCS_3(\rT,\rS). \qedhere
\end{align*}
\end{proof}

\begin{proof}[Proof of \cref{lem:N-eq-3-CRT}]
Recall that $\LCS_3(\frac{2}{\sigma}\cdot\sT,\frac{2}{\sigma'}\cdot\sT')\leq  \LCS_N(\frac{2}{\sigma}\cdot\sT,\frac{2}{\sigma'}\cdot\sT')$ holds generally. Let us further assume that $\LCS_3(\frac{2}{\sigma}\cdot\sT,\frac{2}{\sigma'}\cdot\sT')<\LCS_N(\frac{2}{\sigma}\cdot\sT,\frac{2}{\sigma'}\cdot\sT')$, which we aim to show has probability zero. This event implies that one of the following two events occurs:
\begin{itemize}
    \item either there exists $x\in \sT$ such that $\sT\setminus \{x\}$ has at least four connected components;
    \item or there are $b_1,b_2\in \mathrm{Br}(\sT)$ and $b_1',b_2'\in\mathrm{Br}(\sT')$ with $\frac{2}{\sigma}d(b_1,b_2)=\frac{2}{\sigma'}d'(b_1',b_2')>0$.
\end{itemize}
The probability of the first event occurring is zero since $\sT$ is almost surely binary~\cite{aldous1991continuumI}. 

Now, let $\{(X_i,Y_i,Z_i)\}_{i\geq 1}$ be independent samples of uniform points in $\sT$ and for each $i \geq 1$, let $B_i \in \sT$ be the branch-point of $X_i,Y_i,Z_i$. Since the uniform measure on $\sT$ has full support~\cite{aldous1991continuumI}, and since $\mathrm{Br}(\sT)$ is countable~\cite[Lemma 3.1]{duquesne2007growth}, we have 
\[\{ B_i\, :\, i\geq 1 \}=\mathrm{Br}(\sT)\]
almost surely. Similarly, let $\{(X_i',Y_i',Z_i')\}_{i\geq 1}$ be independent samples of uniform points in $\sT'$ and denote by $B_i'$ the branch-point of $X_i',Y_i',Z_i'$ for all $i\geq 1$, so that $\{B_i'\, :\, i\geq 1\}=\mathrm{Br}(\sT')$ almost surely. Therefore,
\begin{align*}
\P\big(\LCS_3(\sT,\sT')<\LCS_N(\sT,\sT')\big)&\leq \P\big(\exists i,i',j,j'\geq 1,\, \sigma' d(B_i,B_j)=\sigma d'(B_i',B_j')>0\big)\\
&\leq \sum_{i,i',j,j'\geq 1}\P\big(\sigma 'd(B_1,B_2)=\sigma d'(B_1',B_2')>0\big)=0,
\end{align*}
where the final equality holds since the conditional law of $d(B_1,B_2)$ given $\{d(B_1,B_2)>0\}$ is absolutely continuous with respect to the Lebesgue measure (see e.g.\ \cite[Lemma 21]{aldous1993continuumIII}).
\end{proof}

\section{Open problems}
\label{sec:open-problems}

The study of common subtrees of random trees is still in its infancy, and a wealth of natural  questions remain. Here, we list a few. 

\begin{enumerate}
    \item We believe that our methods can be used to prove that, under the assumptions of Theorem~\ref{thm:main},
    \[\frac{1}{\sqrt{n}}\LCS^\bullet(\tau_n,\tau_n') \xrightarrow[n\to\infty]{\mathrm d} c_{\mu,\mu'}\min\big\{\Ht(\tfrac{2}{\sigma}\cdot\sT),\Ht(\tfrac{2}{\sigma'}\cdot\sT')\big\},\]
    jointly with \eqref{Bienayme->CRT},
    and that this may even hold when both $\mu$ and $\mu'$ have only a finite second moment. 
    \item What is the size of the largest common subtree if the offspring distributions have power-law tails and no second moment; say $\mu([k,\infty))\sim c k^{-\alpha}$  and $\mu'([k,\infty))\sim c'k^{-\alpha'}$ with $1<\alpha\le \alpha'<2$ and $c,c'>0$? In this case, the largest common Y-shaped subtree of $\tau_n$ and $\tau_n'$ has length of order $n^{1-1/\alpha}\ll \sqrt{n}$. However, both subtrees contain a star with size of order $n^{1/\alpha'}\gg \sqrt{n}$. This suggests that the largest common subtree in $\tau_n$ and $\tau'_n$ is very wide, in contrast to the setting of Theorem~\ref{thm:main}, which gives a very tall largest common subtree. We conjecture that the largest common subtree of $\tau_n$ and $\tau'_n$ is of the same order as the minimum of the maximum degrees in the two trees, and that $n^{-1/\alpha'}\LCS(\tau_n,\tau'_n)$ has a random limit under rescaling. 
    \item In the case where $\mu$ has finite $(2+\kappa)$th moment and $\mu'$ has infinite $2$nd moment, neither a very tall nor a very wide largest common subtree is possible, so we conjecture that $\LCS(\tau_n,\tau'_n)\ll\sqrt{n}$. In this case, can the size of the largest common subtree be much larger than both the largest degree in $\tau_n$ and the height of $\tau'_n$\,?  
    \item What is the size of the largest common subtree of independent size-conditioned Bienaymé trees $\tau_n$ and $\tau'_N$ with $n\ll N$? In the finite variance case, the order of the smallest Y-shaped tree in $\tau_n$ is a lower bound, but if $N$ is large enough, there are more complicated shapes one can extract from $\tau_n$ and find in $\tau_N'$. If $N$ grows exponentially in $n$ at an appropriately large rate, one should even be able to find an isomorphic copy of $\tau_n$ in $\tau_N'$. 
    \item What happens if we allow small distortions between the two subtrees, allowing them to not be exactly isomorphic? For example, what happens if we allow replacing an edge by a path of length $2$ at most $k$ times? For fixed $k$, this should not change the conclusion of Theorem~\ref{thm:main}, but what if $k\gg1$? In particular, what is the threshold $k^*_n$ for the largest common (perturbed) subtree to be much larger than $\sqrt{n}$\,?
    \item The largest common subtree problem can be extended to trees with additional structures, such as vertex colourings or edge-displacements, that need to agree in a common subtree. How does this change the size and shape of the largest common subtree?
    \item Two independent supercritical Bienaymé trees that are conditioned to be infinite will contain a common subtree of infinite size. In this case, one may wonder how the largest common subtree restricted to the first $n$ generations grows with $n$.  
    
\end{enumerate}

\appendix
\section{Proof of \cref{thm:a_star_is_born}} \label{sec:a_star_is_born}

To prove \cref{thm:a_star_is_born}, we shall consider a family of critical offspring distributions with finite variance for which two independent identically distributed Bienaymé trees conditioned to have size $n$ contain a common subtree of size much larger than $\sqrt{n}$ with high probability.

Fix $1<\lambda<2$ and let $\mu$ be a critical offspring distribution for which $\mu(k)\sim ck^{-3}\log^{-\lambda}(k)$ for some $c>0$ as $k\to \infty$ and $\mu(k)>0$ for all $k\ge 0$. Observe that the variance of $\mu$ is finite by the assumption on the tail of $\mu$. 

The following proposition implies \cref{thm:a_star_is_born}.

\begin{prop}\label{prop:bigstar}
For $\mu$ as above and $\tau_n$ and $\tau'_n$ two independent Bienaymé trees with offspring distribution $\mu$ conditioned to have size $n$, for any $\varepsilon>0$,
there is $\delta>0$ such that
\[\liminf_{n\to\infty} \P\big(\LCS(\tau_n,\tau'_n)>\delta \log^{1-\lambda/2}(n)\sqrt{n}\big)>1-\varepsilon .\]
\end{prop}

To prove the proposition, we will show that for any $\varepsilon>0$, for some $\beta>0$, $\tau_n$ and $\tau'_n$ both contain a vertex with out-degree at least $\beta ({n\log^{-\lambda}(n)})^{1/2}$ with probability at least $1-\varepsilon$. We will also show that the pendant subtrees at these two vertices are roughly distributed as independent $\mu$-Bienaymé trees. The following lemma then implies that the large vertices and their pendant trees yield a common subtree that is as large as claimed by the proposition.

\begin{lem}\label{lem:bigstar}
For $\Delta\ge 1$, let $\tilde{\tau}_\Delta$ be a random plane tree such that $\rk_\varnothing(\tilde{\tau}_\Delta)=\Delta$ and such that its subtrees $\theta_1\tilde{\tau}_\Delta,\ldots,\theta_{\Delta}\tilde{\tau}_\Delta$ are independent Bienaymé trees with offspring distribution $\mu$. Let $\tilde{\tau}'_\Delta$ be an independent copy of $\tilde{\tau}_\Delta$. 
Then, there is $\gamma > 0$ such that 
$\LCS(\tilde{\tau}_\Delta,\tilde{\tau}'_\Delta)>\gamma \Delta \log \Delta$ with high probability as $\Delta\to\infty$. 
\end{lem}
We start by proving the lemma and then use it to prove \cref{prop:bigstar}.
\begin{proof}
For $h\ge 1$, let \[N_{\Delta,h}=\#\{1\le i\le \Delta :\Ht(\theta_i\tilde{\tau}_\Delta)\ge h\}\quad \text{ and }\quad N'_{\Delta,h}=\#\{1\le i\le \Delta :\Ht(\theta_i\tilde{\tau}'_\Delta)\ge h\}.\]
Then, it holds that 
\begin{equation}\label{eq:bigstar}
\LCS(\tilde{\tau}_\Delta,\tilde{\tau}'_\Delta)
\ge \sum_{h\ge 1} 
N_{\Delta,h}\wedge N'_{\Delta,h}.
\end{equation}
Indeed, observe that for $h_1\ge h_2\ge\dots \ge h_\Delta$ and $h'_1\ge h'_2\ge\dots \ge h'_\Delta$ the non-increasing rearrangement of the heights of $\theta_1\tilde{\tau}_\Delta,\dots ,\theta_\Delta\tilde{\tau}_\Delta$ and $\theta_1\tilde{\tau}'_\Delta,\dots ,\theta_\Delta\tilde{\tau}'_\Delta$ respectively,
\[\LCS(\tilde{\tau}_\Delta,\tilde{\tau}'_\Delta)\ge \sum_{i=1}^\Delta h_i\wedge h'_i=\sum_{h\ge 1}\#\{i: h_i\wedge h'_i\ge h\},\]
so that \eqref{eq:bigstar} follows by observing that $N_{\Delta,h}\wedge N'_{\Delta,h}=\#\{i: h_i\wedge h'_i\ge h\}$. 

By \cref{single_tree_height}, there exists a constant $c>0$ such that  $\prob(\Ht(\theta_i\tilde{\tau}_\Delta)\ge h)\ge c/h$ for all $h\geq 1$, so $N_{\Delta,h}$ stochastically dominates a $\operatorname{Binomial}(\Delta, c/h)$ random variable. Applying the Chernoff bound and union bounding over $1\le h \le \lceil \Delta^{1/4}\rceil$, we find that with high probability as $\Delta\to\infty$, it holds that $N_{\Delta,h}\wedge N'_{\Delta,h}>\frac{c\Delta}{2h}$ for all $1\le h\le \lceil \Delta^{1/4}\rceil $. Then, \eqref{eq:bigstar} implies that  
\[\LCS(\tilde{\tau}_\Delta,\tilde{\tau}'_\Delta)
\ge \sum_{h=1}^{\lceil \Delta^{1/4}\rceil} \frac{c\Delta}{2h}\ge \frac{c}{10}\Delta\log \Delta 
\]
with high probability, as claimed. 
\end{proof}

\begin{proof}[Proof of \cref{prop:bigstar}]
In this proof, we use standard random walk encodings of random trees and common related terminology. For an introduction, see e.g.\ Le Gall's survey~\cite{legall2005random}. 

Let $\eps>0$. Let $X^{n}=(X_k, 0\le k \le n)$ be a random walk with $X_0=0$ and steps with law $\nu(\cdot)\coloneqq \mu(\cdot+1)$, so that, for $Y^n$ defined as $X^n$ conditioned on the event that $X_n=-1$ and $V$ the Vervaart transform, $V(Y^n)$ is distributed as the \L ukasiewicz path of $\tau_n$. A key property that we use is as follows: if for some $j,k\ge 1$ it holds that $Y_k-Y_{k-1}\ge j-1$ and $m=\inf\{i\geq k :Y_i=Y_k-j\}\le n$, then $(Y_k, Y_{k+1},\dots, Y_m)$ encodes the first $j$ trees rooted at a vertex with degree $Y_k-Y_{k-1}$ in $\tau_n$.

Let $\phi_k(\ell)=\P(X_k=\ell)$. Then, for any event $\cE_n$ that is measurable with respect to $Y_1,\dots,Y_{\lfloor n/2\rfloor}$, we have 
\begin{align*}\P(\cE_n(Y^n))&=\frac{\P(\cE_n(X^n),X_n=-1 )}{\P(X_n=-1 )}\\
&=\frac{\E[\one_{\cE_n(X^n)}\P(X_n=-1\mid X_1,\dots,X_{\lfloor n/2\rfloor})]}{\P(X_n=-1 )}\\
&=\E\left[\one_{\cE_n(X^n)}\frac{\phi_{n-\lfloor n/2\rfloor}(-X_{\lfloor n/2\rfloor}-1) }{\phi_n(-1)}
\right].
\end{align*}
By the estimate \eqref{local-limit_thm} from the local limit theorem \cite{petrov2012sums}, there is a constant $C>0$ such that 
\[\frac{\phi_{n-\lfloor n/2\rfloor}(\ell) }{\phi_n(-1)}\le \sqrt{C}\]
for all $\ell\in \Z$ and all $n\ge 1$, so that 
\[\P(\cE_n(Y^n))\le \sqrt{C}\P(\cE_n(X^n) ).\]
Similarly, for $(Y')^n$ an independent copy of $Y^n$ and $(X')^n$ an independent copy of $X^n$, for any event $\cE_n$ that is measurable with respect to $Y_1,\dots,Y_{\lfloor n/2\rfloor}$ and $Y'_1,\dots,Y'_{\lfloor n/2\rfloor}$, 
\begin{equation}
\label{eq:like_independent}
\P(\cE_n(Y^n,(Y')^n))\le C\P(\cE_n(X^n,(X')^n)).
\end{equation}

We now define the event $\cE_n$. We observe that by choice of $\mu(k)$, we have $\P(X_1\ge m)\sim \tfrac{c}{2}m^{-2}\log^{-\lambda}(m)$, so we may pick $\beta>0$ such that for $n$ large enough, \[\P\left( X_1\ge \beta \sqrt{n/\log^\lambda(n)} \right)>4\log(9C/\eps)n^{-1}.\] Set $\Delta_n=\beta \sqrt{n/\log^\lambda(n)}$ and let \[T_1(X^n)=\inf\left\{1\le i \le n: X_{i}-X_{i-1}\ge \Delta_n\right\} \] 
with the convention that $\inf\emptyset = \infty$. By a standard property of the \L ukasiewicz path, the increments of $Y^n$ correspond to the out-degrees in $\tau_n$ minus $1$, so $T_1(Y^n)\le n$ implies that $\tau_n$ has a vertex with out-degree exceeding $\Delta_n$. 
By our choice for $\beta$, \[\P(T_1(X^n)\geq n/4) \le \left(1-\frac{4\log(9C/\eps)}{n}\right)^{n/4-1}\le \frac{\eps}{8C}.\] If $T_1(X^n)< n/4$, set $T_2(X^n)=\inf\{i\geq T_1(X^n): X_i<X_{T_1(X^n)}-\Delta_n\}$, and otherwise set $T_2(X^n)=\infty$. Since $X^n$ is a centred random walk with steps with finite variance, it converges in distribution to a Brownian motion when time is rescaled by $n$ and space by $n^{1/2}$. Thus, if $T_1(X^n)<n/4$ then $T_2(X^n)-T_1(X^n)=o(n)$ with high probability, and in particular, for $n$ large enough, $\P(T_1(X^n)<n/4 \,;\, T_2(X^n)<n/2)>1-\eps/(4C)$. 

On the event that $T_1(Y^n)<n/4$ and $T_2(Y^n)<n/2$, the first $\Delta_n$ subtrees rooted at the large out-degree vertex corresponding to time $T_1(Y^n)$ are measurable with respect to $Y_1,\dots, Y_{\lfloor n/2\rfloor}$.  Let $\cE_n(Y^n,(Y')^n)$ be the bad event that $T_1(Y^n)>n/4$ or $T_2(Y^n)>n/2$ or $T_1((Y')^n)>n/4$ or $T_2((Y')^n)>n/2$ or the largest common subtree of the subtree of $\tau_n$ encoded by \[(Y_{T_1(Y^n)},Y_{{T_1(Y^n)}+1}, \dots, Y_{T_2(Y^n)})\] on one hand and the subtree of $\tau'_n$ encoded by \[(Y'_{T_1((Y')^n)},Y'_{{T_1((Y')^n)}+1}, \dots, Y'_{T_2((Y')^n)})\] on the other hand is smaller than $\delta \Delta_n\log \Delta_n$, for $\delta$ as in \cref{lem:bigstar}. In $X^n$, these trees are i.i.d.~copies of $\tau$, so that \cref{lem:bigstar} and the bounds on $T_1(X^n)$ and $T_2(X^n)$ above imply that $\P(\cE_n(X^n,(X')^n)<\eps/C$ for $n$ large enough. The statement then follows from \eqref{eq:like_independent}.
\end{proof}

\bibliographystyle{abbrv}
\bibliography{biblio}

\end{document}